\documentclass[a4paper, 10pt, notitlepage]{article}

\usepackage{amsthm, amsmath, amssymb, latexsym}
\usepackage{mathrsfs,cite}
\usepackage[multiple]{footmisc}
\usepackage{graphicx}

\theoremstyle{plain}
\newtheorem{theorem}{Theorem}[section]	
\newtheorem{lemma}[theorem]{Lemma}	
\newtheorem{corollary}[theorem]{Corollary}

\theoremstyle{definition}
\newtheorem{definition}[theorem]{Definition}
\newtheorem{example}[theorem]{Example}

\theoremstyle{remark}
\newtheorem{remark}[theorem]{Remark}

\usepackage{chngcntr}
\counterwithin{theorem}{section}	

\usepackage[a4paper]{geometry}	

\numberwithin{equation}{section}

\DeclareMathOperator{\xWsp}{\textrm{W}}
\DeclareMathOperator{\xL}{\textrm{L}}
\DeclareMathOperator{\xR}{\mathbb{R}}
\DeclareMathOperator{\xC}{\textrm{C}}

\DeclareMathOperator{\co}{co}
\DeclareMathOperator{\cl}{cl}
\DeclareMathOperator{\cone}{cone}

\DeclareMathOperator*{\esssup}{ess\,sup}
\DeclareMathOperator{\dist}{dist}
\DeclareMathOperator{\diverg}{div}

\author{M.V. Dolgopolik\footnote{Institute for Problems in Mechanical Engineering, Russian Academy of Sciences,
Saint Petersburg, Russia}}
\title{Constrained Nonsmooth Problems of the Calculus of Variations}

\begin{document}

\maketitle

\begin{abstract}
The paper is devoted to an analysis of optimality conditions for nonsmooth multidimensional problems of the calculus of
variations with various types of constraints, such as additional constraints at the boundary and isoperimetric
constraints. To derive optimality conditions, we study generalised concepts of differentiability of nonsmooth functions
called codifferentiability and quasidifferentiability. Under some natural and easily verifiable assumptions we prove
that a nonsmooth integral functional defined on the Sobolev space is continuously codifferentiable and compute its
codifferential and quasidifferential. Then we apply general optimality conditions for nonsmooth optimisation problems in
Banach spaces to obtain optimality conditions for nonsmooth problems of the calculus of variations. Through a series of
simple examples we demonstrate that our optimality conditions are sometimes better than existing ones in terms of
various subdifferentials, in the sense that our optimality conditions can detect the non-optimality of a given point,
when subdifferential-based optimality conditions fail to disqualify this point as non-optimal.
\end{abstract}

\section*{Introduction}

Nonsmooth problems of the calculus of variations arise in various applications, such as optimisation of hydrothermal
systems \cite{BayonGrau2006,BayonGrau2009,BayonGrau2014} and nonsmooth modelling in mechanics and engineering (see
monograph \cite{QuasidiffMechanics}). Their theoretical study was started by Rockafellar in the convex
case in \cite{Rockafellar70,Rockafellar70_2,Rockafellar71}, where some existence and duality results, as well as
optimality conditions in terms of subdifferentials, were obtained. In these optimality conditions the classical
Euler-Lagrange equation and transversality condition for the problem of Bolza
\begin{equation} \label{eq:IntroProblem}
  \min \: l(x(0), x(T)) + \int_0^T L(t, x(t), \dot{x}(t)) \, dt
\end{equation}
were replaced by the following inclusions:
$$
  (\dot{p}(t), p(t)) \in \partial L(t, x(t), \dot{x}(t)) \quad \text{for a.e. } t \in (0, T), \quad
  (p(0), - p(T)) \in \partial l(x(0), x(T),
$$
where ``$\partial$'' stands for subdifferential in the sense of convex analysis \cite{IoffeTihomirov,EkelandTemam}.
Note that if the function $L$ is differentiable, then $p(t) = L'_{\dot{x}}(t, x(t), \dot{x}(t))$ and the first
inclusion is reduced to the Euler-Lagrange equation.

Further research was devoted to relaxing the convexity assumptions made by Rockafellar and replacing the
subdifferential in the sense of convex analysis by some other subdifferential defined for nonconvex functions.
Important steps in this direction were made by Clarke \cite{Clarke75,Clarke76,Clarke80,Clarke}, who studied problems
with locally Lipschitz continuous functions $l$ and $L$ and replaced the subdifferentials in the sense of convex
analysis with what now is known as the Clarke subdifferential. Apart from optimality conditions in the form of the
Euler-Lagrange inclusion, Clarke also obtained optimality conditions in the Hamiltonian form. Clarke's results were
sharpened and extended to more general variational problems by Loewen and Rockafellar
\cite{LoewenRockafellar91,LoewenRockafellar94}, while equivalence between Euler-Lagrange and Hamiltonian forms of
optimality conditions for nonsmooth variational problems was studied in
\cite{Clarke93,Rockafellar93,Ioffe97,BessisLedyaevVinter2001}. Nonlocal optimality conditions for nonconvex problems of
the calculus of variations in terms of subdifferentials in the sense of convex analysis and their connections to the
existence of minimisers were studied by Marcelli et al.
\cite{MarchelliOutkin,Marcelli2002,CupiniMarchelli,Marcelli2008}.

First optimality conditions for nonsmooth variational problems involving \textit{nonconvex} subdifferentials were
obtained by Mordukhovich \cite{Mordukhovich_book,Mordukhovich95} (see also \cite{Mordukhovich_I,Mordukhovich_II}).
Later, optimality conditions for a nonsmooth problem of Bolza in terms of limiting proximal and limiting Fr\'{e}chet
subdifferentials were studied by Loewen and Rockafellar \cite{LoewenRockafellar96,LoewenRockafellar97,Loewen}, Ioffe
and Rockafellar \cite{IoffeRockafellar}, Vinter and Zheng \cite{VinterZheng97, Vinter}, Bellaassali \cite{Bellaassali},
and Jourani \cite{Jourani2009}.

A different approach to an analysis of optimality conditions for nonsmooth problems of the calculus of variations based
on the use of \textit{codifferentials} was developed by the author in \cite{Dolgopolik_CalcVar}. Codifferentials of
nonsmooth functions were introduced by Demyanov \cite{Demyanov_InCollection_1988,Demyanov1988,Demyanov1989} in the late
1980s. A general theory of codifferentiable functions, closely related to the theory of Demyanov-Rubinov-Polyakova
quasidifferentials \cite{DemyanovDixon,QuasidiffMechanics,Quasidifferentiability_book}, was developed in the finite
dimensional case in \cite{DemRub_book}. Its infinite dimensional generalisations were studied in
\cite{Zaffaroni98,Zaffaroni,Dolgopolik_CodiffCalc,Dolgopolik_AbstrConvApprox,Dolgopolik_CodiffDescent}. 
In \cite{Dolgopolik_CalcVar} it was shown that optimality conditions for problem \eqref{eq:IntroProblem} in terms of
codifferentials are sometimes better than subdifferential-based optimality conditions. Let us also mention a completely
different approach to the derivation of optimality conditions and numerical solution of nonsmooth problems of the
calculus of variations based on the Chebyshev pseudospectral method \cite{NooriKamyad2014}.

It should be noted that in most of the aforementioned papers nonsmooth problems of the calculus of variations were not
studied by themselves, but in the context of variational problems for differential inclusions. It seems that since 
the mid-90s nonsmooth problems of the calculus of variations became just an auxiliary tool for the derivation of
optimality conditions for nonsmooth optimal control problems and nonsmooth variational problems involving differential
inclusions (cf.~\cite{Jourani2009,Ioffe2016,Ioffe2019}). As a result, relatively little attention has been paid to
nonsmooth multidimensional problems of the calculus of variations, as well as problems with additional constraints, such
as nonsmooth isoperimetric problems and problems with additional constraints at the boundary. Nonsmooth multidimensional
problems of the calculus of variations were first studied by Clarke \cite{Clarke77_Multiple} for locally Lipschitz
continuous integrands. Improved versions of the Clarke's first result were later published in monographs
\cite{Clarke,Clarke2013}, while Bousquet \cite{Bousquet} showed that one can significantly relax the growth conditions
on the integrand imposed in the Clarke's work. Bonfanti and Cellina \cite{BonfantiCellina} obtained optimality
conditions for the problem
\[
  \min \: \int_{\Omega} L(x, v(x), \nabla v(x)) \, dx \quad 
  \text{subject to} \quad v \in v_0 + W_0^{1, 1}(\Omega)
\]
in the case when the integrand $L(x, v, \xi)$ is differentiable in $v$ and convex in $\xi$. Optimality conditions for
nonsmooth multidimensional problems of the calculus of variations in terms of codifferentials were obtained in the
author's paper \cite{Dolgopolik_CalcVar}, while optimality conditions for such problems in terms of the so-called
K-subdifferential were obtained in \cite{OrlovTsygankova}. Finally, nonsmooth variational problems with additional
constraints have been explicitly studied only by Clarke \cite{Clarke,Clarke2013} and Bellaassali \cite{Bellaassali}. 

The main goal of this paper is to present a general theory of necessary optimality conditions for nonsmooth
multidimensional problems of the calculus of variations with various types of additional constraints, such as problems
with constraints at the boundary and problems with isoperimetric constraints. To this end, we significantly improve our
earlier results from \cite{Dolgopolik_CalcVar} and prove the codifferentiability of a nonsmooth integral functional
defined on the Sobolev space under natural and easily verifiable assumptions on the integrand. In comparison with our
previous paper \cite{Dolgopolik_CalcVar}, we get rid of the obscure and hard to verify assumption on the uniform
codifferentiability of the integrand with respect to the Sobolev space and do not impose any assumptions on the domain
of integration, thus extending the results of \cite{Dolgopolik_CalcVar} to the case of unbounded domains and domains
with irregular boundary (see Section~\ref{sec:CodiffIntegralFunc} for more details). Furthermore, under natural
assumptions we prove the continuity of a codifferential of the integral functional in the general case (in
\cite{Dolgopolik_CalcVar} the continuity was proved only in the case $p = + \infty$). Continuity is an important
property for an analysis of discretisation of variational problems, approximation methods, and convergence of numerical
methods. In particular, in the general case the continuity of codifferential is necessary for the global convergence of
optimisation methods based on codifferentials \cite{Dolgopolik_CodiffDescent}.

With the use of the general result on the codifferentiability of an integral functional obtained in this paper and
necessary optimality conditions for nonsmooth mathematical programming problems in Banach spaces in terms of
quasidfferentials from \cite{Dolgopolik_QuasidiffProg,Dolgopolik_MetricReg} we derive optimality conditions for
unconstrained nonsmooth problems of the calculus of variations, as well as problems with additional constraints at the
boundary and isoperimetric constraints. Each of these optimality conditions is illustrated by a simple example, in
which existing optimality conditions in terms of various subdifferentials are satisfied at a non-optimal point, while
our optimality conditions are able to detect the non-optimality of this point. Thus, the optimality conditions obtained
in this paper are in some cases better than existing subdifferential-based optimality conditions.

The paper is organised as follows. The codifferentiability of an integral functional defined on the Sobolev space is
studied in Section~\ref{sec:CodiffIntegralFunc}. Section~\ref{sec:CalcVar} is devoted to derivation of necessary
optimality conditions for constrained nonsmooth problems of the calculus of variations in terms of codifferentials. 
This section also contains several examples illustrating advantages of optimality conditions in terms of codifferentials
in comparison with subdifferential-based optimality conditions. Finally, Section~\ref{sec:Preliminaries} contains
some auxiliary definitions from nonsmooth analysis that are necessary for understanding the paper (apart from
Examples~\ref{example:UnconstrainedProblem}, \ref{example:BoundaryConstrainedProblem}, and
\ref{example:IsoperimetricProblem}, whose understanding requires some familiarity with the Clarke subdifferential
\cite{Clarke}, the limiting proximal subdifferential \cite{Vinter,ClarkeLedyaev}, and the limiting Fr\'{e}chet
subdifferential \cite{Mordukhovich_I,Mordukhovich_II}).

\section{Codifferentiable and Quasidifferentiable Functions}
\label{sec:Preliminaries}

In what follows, let $X$ be a real Banach space. Its topological dual space is denoted by $X^*$, while the
canonical duality pairing between $X$ and $X^*$ is denoted by $\langle \cdot, \cdot \rangle$, i.e.
$\langle x^*, x \rangle = x^*(x)$ for all $x^* \in X^*$ and $x \in X$. The standard topology on $\xR$ is denoted
by $\tau_{\xR}$ and the weak$^*$ topology on $X^*$ is denoted by $w^*$ or $\sigma(X^*, X)$.

We equip the Cartesian product $\xR \times X$ with the norm $\| (a, x) \| = \sqrt{|a|^2 + \| x \|^2}$ for all
$(a, x) \in \xR \times X$. It is easily seen that the topological dual space $(\xR \times X)^*$ endowed
with the weak${}^*$ topology is isomorphic (in the category of topological vector spaces) to the space 
$\xR \times X^*$ endowed with the product topology $\tau_{\xR} \times w^*$. Utilising this fact (or
arguing directly) one can check that a subset of the topological vector space 
$(\xR \times X^*, \tau_{\xR} \times w^*)$ is compact if and only if it is closed in the topology 
$\tau_{\xR} \times w^*$ and bounded with respect to the norm 
$\| (a, x^*) \| = \sqrt{|a|^2 + \| x^* \|^2}$, $(a, x^*) \in \xR \times X^*$
(see, e.g. \cite[Thrm.~2.1]{Dolgopolik_CodiffCalc}).

\begin{definition} \label{def:CodiffFunc}
Let $U \subseteq X$ be an open set. A function $f \colon U \to \xR$ is called \textit{codifferentiable} at a
point $x \in U$, if there exists a pair $D f(x) = [\underline{d} f(x), \overline{d} f(x)]$ of convex sets 
$\underline{d} f(x), \overline{d} f(x) \subset \xR \times X^*$ that are compact in the product topology 
$\tau_{\xR} \times w^*$ and satisfy the equalities $\Phi_f(x, 0) = \Psi_f(x, 0) = 0$ and
\begin{equation} \label{eq:CodiffDef}
  \lim_{\alpha \to +0} \frac{1}{\alpha } \Big| f(x + \alpha \Delta x) - f(x)
  - \Phi_f(x, \alpha \Delta x) - \Psi_f(x, \alpha \Delta x) \Big| = 0
  \quad \forall \Delta x \in X,
\end{equation}
where
\begin{equation} \label{eq:ConvexConcaveApprox}
  \Phi_f(x, \Delta x) = \max_{(a, x^*) \in \underline{d} f(x)} (a + \langle x^*, \Delta x \rangle), \quad
  \Psi_f(x, \Delta x) = \min_{(b, y^*) \in \overline{d} f(x)} (b + \langle y^*, \Delta x \rangle)
  \quad \forall \Delta x \in X.
\end{equation}
The pair $D f(x) = [\underline{d} f(x), \overline{d} f(x)]$ is called a \textit{codifferential} of $f$ at $x$,
the set $\underline{d} f(x)$ is called a \textit{hypodifferential} of $f$ at $x$, while the
set $\overline{d} f(x)$ is referred to as a \textit{hyperdifferential} of $f$ at $x$.
\end{definition}

\begin{remark}
Note that the equalities $\Phi_f(x, 0) = \Psi_f(x, 0) = 0$ simply mean that $a \le 0$ for all 
$(a, x^*) \in \underline{d} f(x)$ and $\max_{(a, x^*) \in \underline{d} f(x)} a = 0$, while
$b \ge 0$ for all $(b, y^*) \in \overline{d} f(x)$ and $\min_{(b, y^*) \in \overline{d} f(x)} b = 0$. Note also that
the maximum in the definition of $\Phi_f$ and the minimum in the definition of $\Psi_f$ are attained due to the fact
that the sets $\underline{d} f(x)$ and $\overline{d} f(x)$ are compact in the product topology $\tau_{\xR} \times w^*$.
\end{remark}

Let us comment on the definition of codifferentiability. Observe that the function $\Phi_f(x, \cdot)$ from this
definition is convex, while the function $\Psi_f(x, \cdot)$ is concave. Thus, in the definition of codifferentiable
function  one approximates the increment of a nonsmooth function $f$ with the use of the DC (difference-of-convex)
function $\Phi_f(x, \cdot) + \Psi_f(x, \cdot)$ (see~\eqref{eq:CodiffDef}). One can check that $f$ is codifferentiable at
a point $x$ if and only if its increment $f(x + \alpha \Delta x) - f(x)$ can be approximated in this way by
\textit{some} continuous DC function (see~\cite[Example~3.10]{Dolgopolik_AbstrConvApprox}). Hence, in particular, any
continuous DC function $f \colon X \to \xR$ is codifferentiable at every point $x \in X$. Let us note that the benefit
of using codifferentiable functions in comparison with DC functions consists in the existence of a well-developed
codifferential calculus, which allows one to easily compute codifferentials of many nonsmooth functions appearing in
applications (see, e.g. monograph \cite{DemRub_book}). In contrast, while it is usually fairly easy to prove
theoretically that a given function is DC, in some cases it might be very problematic to find an explicit DC
representation of a DC function.

Observe that codifferential is not uniquely defined. For instance, it is easily seen that if $D f(x)$ is a 
codifferential of $f$ at $x$, then for any convex compact subset $C$ of the space 
$(\xR \times X^*, \tau_{\xR} \times w^*)$ the pair $[\underline{d} f(x) + C, \overline{d} f(x) - C]$ is a
codifferential of $f$ at $x$ as well.

Recall that for any two nonempty subsets $A$ and $B$ of a metric space $(M, d)$ \textit{the Hausdorff metric} 
$d_H(A, B)$ is defined by
\begin{equation} \label{eq:HausdorffDist}
  d_H(A, B) = \max\{ \sup_{x \in A} \inf_{y \in B} d(x, y), \sup_{y \in B} \inf_{x \in A} d(x, y) \}.
\end{equation}
A multifunction $F$ between metric spaces $(M_1, d_1)$ and $(M_2, d_2)$ is called \textit{Hausdorff continuous} (or
continuous with respect to the Hausdorff metric) at a point $x \in M_1$, if for any $\varepsilon > 0$ one can find
$\delta > 0$ such that for all $y \in M_1$ with $d_1(y, x) < \delta$ one has $d_H(F(y), F(x)) < \varepsilon$.
Let, as above, $U \subseteq X$ be an open set.

\begin{definition}
A function $f \colon U \to \xR$ is said to be \textit{continuously codifferentiable} at a point $x \in U$, if $f$
is codifferentiable at every point in a neighbourhood $V \subseteq U$ of $x$ and there exists a codifferential mapping 
$D f(\cdot)$ defined in $V$ and such that the corresponding set-valued mappings $\underline{d} f(\cdot)$ and
$\overline{d} f(\cdot)$ are Hausdorff continuous at $x$ (in this case we say that $D f(\cdot)$ is Hausdorff continuous
at $x$). Finally, $f$ is called continuously codifferentiable on a set $A \subset U$, if it is codifferentiable at every
point of this set and there exists a codifferential mapping $D f(\cdot)$ defined on $A$ and such that the corresponding
mappings $\underline{d} f(\cdot)$ and $\overline{d} f(\cdot)$ are Hausdorff continuous on $A$.
\end{definition}

The class of continuously codifferentiable functions is closed under addition, multiplication, pointwise maximum and
minimum of finite families of functions, and composition with smooth functions 
(see \cite{DemRub_book,Dolgopolik_CodiffCalc,Dolgopolik_AbstrConvApprox,Dolgopolik_CodiffDescent} for more details). 

\begin{remark}
Let us note that in the nonsmooth case the standard necessary optimality condition $\nabla f(x) = 0$ takes the form of
set-theoretic inclusion $0 \in \partial f(x)$, involving some subdifferential of the objective function $f$. For
codifferentiable functions, necessary optimality conditions are formulated as set-theoretic inclusions involving hypo-
and hyperdifferentials:
$$
  0 \in \underline{d} f(x) + (0, y^*) \quad \forall (0, y^*) \in \overline{d} f(x).
$$
As a result, in the nonsmooth case optimality conditions no longer play the role of an equation for finding minimisers
and are often used only for verifying whether a given point is optimal and constructing optimisation methods. As we
will see in the following sections, the situation is precisely the same in the case of nonsmooth problems of the
calculus of variations. For nonsmooth variational problems, the Euler-Lagrange equation, which in the smooth case can be
used to find potential extremals, is replaced by a certain inclusion. This inclusion no longer allows one to directly
find extremals and can usually be used only to check whether a given point is potentially optimal and to study
numerical procedures for finding potential extremals.
\end{remark}

The class of codifferentiable nonsmooth functions is closely related to the class of quasidifferentiable functions.
Recall that a function $f \colon U \to \xR$ is called \textit{quasidifferentiable} at a point $x \in U$, if $f$
is directionally differentiable at $x$, i.e. for any $v \in X$ there exists the finite limit
$$
  f'(x, v) = \lim_{\alpha \to + 0} \frac{f(x + \alpha v) - f(x)}{\alpha},
$$
and the function $f'(x, \cdot)$ can be represented as the difference of continuous sublinear functions or, equivalently,
if there exists a pair $\mathscr{D} f(x) = [\underline{\partial} f(x), \overline{\partial} f(x)]$ of convex weak${}^*$
compact sets $\underline{\partial} f(x), \overline{\partial} f(x) \subset X^*$ such that
$$
  f'(x, v) = \max_{x^* \in \underline{\partial} f(x)} \langle x^*, v \rangle
  + \min_{y^* \in \overline{\partial} f(x)} \langle y^*, v \rangle \quad \forall v \in X.
$$
The set $\underline{\partial} f(x)$ is called a \textit{subdifferential} of $f$ at $x$, while the set
$\overline{\partial} f(x)$ is referred to as a \textit{superdifferential} of $f$ at $x$. Note that, just like
codifferential, quasidifferential is not uniquely defined. The interesting problem of finding a minimal (in some
sense) quasidifferential of a given nonsmooth function was studied in 
\cite{Handschug,Scholtes92,PallaschkeUrbanski93,GrzybowskiUrbanski97,Gao98,GrzybowskiUrbanski2006,Grzybowski2008,
Grzhybowski2010,Pallaschke2012}.

\begin{remark}
In what follows, we denote a codifferential of a nonsmooth function $f$ at a point $x$ by $D f(x)$, while 
a quasidifferential of this function is denoted by $\mathscr{D} f(x)$.
\end{remark}

Finally, one says that a function $f \colon U \to \xR$ is \textit{Hadamard quasidifferentiable at} $x$, if 
$f$ is quasidifferentiable at $x$ and Hadamard directionally differentiable at this point, that is,
\begin{equation} \label{eq:HadamardDer_Def}
  f'(x, v) = \lim_{[\alpha, v'] \to [+0, v]} \frac{f(x + \alpha v') - f(x)}{\alpha}
\end{equation}
(see \cite{Giannessi} for a discussion of the notation under this limit). In other words, for any $\varepsilon > 0$ and
$v \in X$ one can find $\delta > 0$ such that $| (f(x + \alpha v') - f(x)) / \alpha - f'(x, v) | < \varepsilon$
for all $0 < \alpha < \delta$ and $v' \in B(v, \delta) = \{ v' \in X \mid \| v' - v \| \le \delta \}$. We will need the
following result stating that every continuously codifferentiable function is, in fact, Hadamard quasidifferentiable and
indicating how one can compute a quasidifferential of this function. A relationship between continuously
codifferentiable functions and quasidifferentiable functions having outer semicontinuous subdifferential and
superdifferential mappings $\underline{\partial} f(\cdot)$ and $\overline{\partial} f(\cdot)$ was analysed in the finite
dimensional case by Kuntz \cite{Kuntz}.

\begin{lemma} \label{lem:ContCodiffQuasidiffUniform}
Let a function $f \colon U \to \xR$ be continuously codifferentiable at a point $x \in U$. Then $f$ is Hadamard
quasidifferentiable at this point and for any codifferential $D f(x)$ of $f$ at $x$ the pair 
$\mathscr{D} f(x) = [\underline{\partial} f(x), \overline{\partial} f(x)]$ defined by
\begin{equation} \label{eq:QuasidiffOfCodiffFunc}
  \underline{\partial} f(x) = \big\{ x^* \in X^* \bigm| (0, x^*) \in \underline{d} f(x) \big\}, \quad
  \overline{\partial} f(x) = \big\{ y^* \in X^* \bigm| (0, y^*) \in \overline{d} f(x) \big\}
\end{equation}
is a quasidifferential of $f$ at $x$.
\end{lemma}

\begin{proof}
Let $D f(x)$ be any codifferential of $f$ at $x$ and the functions $\Phi_f$ and $\Psi_f$ be defined as in
\eqref{eq:ConvexConcaveApprox}. Denote $g(\cdot) = \Phi_f(x, \cdot) + \Psi_f(x, \cdot)$. Recall that by the definition
of codifferential one has $g(0) = 0$ (see Def.~\ref{def:CodiffFunc}).

Suppose at first that $g$ is directionally differentiable at zero and fix any $v \in X$ and $\varepsilon > 0$. Then
there exists $\delta_1 > 0$ such that $| g(\alpha v) / \alpha - g'(0, v) | < \varepsilon / 2$ for 
all $\alpha \in (0, \delta_1)$. By the definition of codifferential there exists $\delta_2 > 0$ such that for all
$\alpha \in (0, \delta_2)$ one has $|f(x + \alpha v) - f(x) - g(\alpha v)| < \varepsilon \alpha / 2$. Therefore, for
any $0 < \alpha < \min\{ \delta_1, \delta_2 \}$ one has
$$
  \left| \frac{f(x + \alpha v) - f(x)}{\alpha} - g'(0, v) \right|
  \le \frac{1}{\alpha }\big| f(x + \alpha v) - f(x) - g(\alpha v) \big| +
  \left| \frac{g(\alpha v)}{\alpha} - g'(0, v) \right| < \varepsilon,
$$
i.e. $f$ is directionally differentiable at $x$ and $f'(x, \cdot) = g'(0, \cdot)$.

Let us now show that both functions $\Phi_f(x, \cdot)$ and $\Psi_f(x, \cdot)$ are directionally differentiable at zero
and compute their directional derivatives. Then one obtains that $f$ is directionally differentiable at $x$ and
$f'(x, v) = \Phi_f'(x, \cdot)(0, v) + \Psi_f'(x, \cdot)(0, v)$ for all $v \in X$.

As was pointed out above, the sets $\underline{d} f(x)$ and $\overline{d} f(x)$ are
norm-bounded due to the fact that they are compact in the product topology $\tau_{\xR} \times w^*$. Thus, there
exists $K > 0$ such that $|a| + \| x^* \| \le K$ for any $(a, x^*) \in \underline{d} f(x) \cup \overline{d} f(x)$.
Hence
$$
  |\Phi_f(x, \Delta x)| \le K + K \| \Delta x \|, \quad 
  |\Psi_f(x, \Delta x)| \le K + K \| \Delta x \| \quad \forall \Delta x \in X,
$$
i.e. the functions $\Phi_f(x, \cdot)$ and $\Psi_f(x, \cdot)$ are bounded on bounded sets. Therefore, the convex function
$\Phi_f(x, \cdot)$ is continuous by \cite[Prp.~I.2.5]{EkelandTemam}, subdifferentiable on $X$ by
\cite[Prp.~I.5.2]{EkelandTemam}, everywhere directionally differentiable by \cite[Prp.~4.1.4]{IoffeTihomirov}, its
subdifferential at zero has the form
$$
  \partial \Phi_f(x, \cdot)(0) = \{ x^* \in X^* \mid (0, x^*) \in \underline{d} f(x) \}
$$
by \cite[Thrm.~4.2.3]{IoffeTihomirov} (recall that by definition $\Phi_f(x, 0) = 0$; see Def.~\ref{def:CodiffFunc}), and
for all $v \in X$ one has $\Phi'(x, \cdot)(0, v) = \max_{x^* \in \partial \Phi_f(x, \cdot)(0)} \langle x^*, v \rangle$ 
by \cite[Prp.~4.1.1]{IoffeTihomirov}. Similarly, the concave function $\Psi_f(x, \cdot)$ is everywhere directionally
differentiable and
$$
  \Psi'_f(x, \cdot)(0, v) = \min_{y^* \in \partial \Psi_f(x, \cdot)(0)} \langle y^*, v \rangle 
  \quad \forall v \in X,
$$
where $\partial \Psi_f(x, \cdot)(0) = \{ y^* \in X^* \mid (0, y^*) \in \overline{d} f(x) \}$. Thus, one can conclude
that  $f$ is quasidifferentiable at $x$ and the pair \eqref{eq:QuasidiffOfCodiffFunc} is a quasidifferential of $f$ 
at $x$.

Let us finally show that $f$ is Hadamard directionally differentiable at $x$. Indeed, by
\cite[Crlr.~2]{Dolgopolik_CodiffDescent} the function $f$ is Lipschitz continuous near $x$, i.e. there exist $r > 0$
and $L > 0$ such that $|f(x_1) - f(x_2)| \le L \| x_1 - x_2 \|$ for all $x_1, x_2 \in B(x, r)$.

Fix any $\varepsilon > 0$ and $v \in X$. Note that $f'(x, 0) = 0$ and $x + \alpha v' \in B(x, r)$ for any
$\alpha \in (0, \sqrt{r})$ and $v' \in X$ with $\| v' \| \le \sqrt{r}$. Therefore for 
$\delta = \min\{ \sqrt{r}, \varepsilon / 2 L \}$ one has
$$
  \left| \frac{f(x + \alpha v') - f(x)}{\alpha} - f'(x, 0) \right| \le L \| v' \| < \varepsilon
  \quad \forall \alpha \in (0, \delta), \: v' \in B(0, \delta),
$$
i.e. \eqref{eq:HadamardDer_Def} with $v = 0$ holds true. Thus, one can suppose that $v \ne 0$.

By the definition of directional derivative there exists $\delta_0 > 0$ such that
$$
  \left| \frac{f(x + \alpha v) - f(x)}{\alpha} - f'(x, v) \right| < \frac{\varepsilon}{2}
  \quad \forall \alpha \in (0, \delta_0).
$$
Observe that if $0 < \alpha < r / (2 \| v \|)$ and for some $v' \in X$ one has $\| v' - v \| < \| v \|$, then
$x + \alpha v \in B(x, r / 2)$ and $\| \alpha v' - \alpha v \| = \alpha \| v' - v \| < r / 2$, i.e.
$x + \alpha v' \in B(x, r)$. Therefore, put $\delta = \min\{ \delta_0, \varepsilon / 2 L, r / (2\| v \|), \| v \| \}$.
Then for any $\alpha \in (0, \delta)$ and  $v' \in B(v, \delta)$ one has $x + \alpha v' \in B(x, r)$ and
$$
  \left| \frac{f(x + \alpha v') - f(x)}{\alpha} - f'(x, v) \right|
  \le \left| \frac{f(x + \alpha v) - f(x)}{\alpha} - f'(x, v) \right|
  + \frac{1}{\alpha} \big| f(x + \alpha v') - f(x + \alpha v) \big| \le \frac{\varepsilon}{2} + L \| v' - v \| 
  < \varepsilon.
$$
Thus, $f$ is Hadamard quasidifferentiable at $x$.
\end{proof}

\begin{remark} \label{rmrk:Codifferential_Quasidifferential}
From the proof of the lemma above it follows that if the function $f$ is codifferentiable, but \textit{not}
continuously codifferentiable at $x$, then $f$ is still quasidifferentiable at $x$ and for any codifferential $D f(x)$
of $f$ at $x$ the pair \eqref{eq:QuasidiffOfCodiffFunc} is a quasidifferential of $f$ at $x$.
\end{remark}

\section{Codifferentiability of Integral Functionals}
\label{sec:CodiffIntegralFunc}

In this section we present simple sufficient conditions for the codifferentiability of integral functional
$$
  \mathcal{I}(u) = \int_{\Omega} f(x, u(x), \nabla u(x)) \, dx, 
  \quad u = (u_1, \ldots, u_m) \in \textrm{W}^{1, p}(\Omega; \xR^m)
$$
and compute its codifferential and quasidifferential. Here $\Omega \subseteq \xR^d$ is an open set (not
necessarily bounded), $f \colon \Omega \times \xR^m \times \xR^{m \times d} \to \xR$, 
$f = f(x, u, \xi)$, is a given function, $\xWsp^{1, p}(\Omega; \xR^m)$ is the Cartesian product of $m$ copies of the
Sobolev space $\xWsp^{1, p}(\Omega)$ with $1 \le p \le +\infty$. The space $\xWsp^{1, p}(\Omega; \xR^m)$ is endowed with
the norm $\| u \|_{1, p} = \big( \| u \|_p^p + \| \nabla u \|_p^p \big)^{1 / p}$ in the case $1 \le p < + \infty$ and
$\| u \|_{1, \infty} = \max\{ \| u \|_{\infty}, \| \nabla u \|_{\infty} \}$, where $\| \cdot \|_p$ is the standard norm
in $\xL^{p}(\Omega; \xR^k)$ for any $k \in \mathbb{N}$, i.e.
$\| u \|_p = (\int_{\Omega} |u(x)|^p \, dx)^{\frac{1}{p}}$
in the case $1 \le p < + \infty$, and $\| u \|_{\infty} = \esssup_{x \in \Omega} |u(x)|$ (here $| \cdot |$ is the
Euclidean norm). Denote by $p'$ the conjugate exponent of $p$, i.e. $1/p + 1/p' = 1$.

Below we assume that for a.e. $x \in \Omega$ the function $(u, \xi) \mapsto f(x, u, \xi)$ is codifferentiable, i.e. for
a.e. $x \in \Omega$ and for all $(u, \xi) \in \xR^m \times \xR^{m \times d}$ there exist compact convex
sets $\underline{d}_{u, \xi} f(x, u, \xi), \overline{d}_{u, \xi} f(x, u, \xi) \subset 
\xR \times \xR^m \times \xR^{m \times d}$ such that for any 
$(\Delta u, \Delta \xi) \in \xR^m \times \xR^{m \times d}$ one has
$$
  \lim_{\alpha \to + 0} \frac{1}{\alpha} \Big| f(x, u + \alpha \Delta u, \xi + \alpha \Delta \xi) - f(x, u, \xi)
  - \Phi_f(x, u, \xi; \alpha \Delta u, \alpha \Delta \xi) 
  - \Psi_f(x, u, \xi; \alpha \Delta u, \alpha \Delta \xi) \Big| = 0
$$
and $\Phi_f(x, u, \xi; 0, 0) = \Psi_f(x, u, \xi; 0, 0) = 0$, where
\begin{align} \label{eq:CodiffIntegrand1}
  \Phi_f(x, u, \xi; \Delta u, \Delta \xi) =
  \max_{(a, v_1, v_2) \in \underline{d}_{u, \xi} f(x, u, \xi)}
  \big( a + \langle v_1, \Delta u \rangle + \langle v_2, \Delta \xi \rangle \big), \\
  \Psi_f(x, u, \xi; \alpha \Delta u, \alpha \Delta \xi) = 
  \min_{(b, w_1, w_2) \in \overline{d}_{u, \xi} f(x, u, \xi)} 
  \big( b + \langle w_1, \Delta u \rangle + \langle w_2, \Delta \xi \rangle \big), \label{eq:CodiffIntegrand2}
\end{align}
and $\langle \cdot, \cdot \rangle$ is the inner product in $\xR^k$. We denote a codifferential of this function at a
point $(x, u, \xi)$ by  
$D_{u, \xi} f(x, u, \xi) = [\underline{d}_{u, \xi} f(x, u, \xi), \overline{d}_{u, \xi} f(x, u, \xi)]$. Finally, recall
that a multifunction $F \colon \Omega \times Y \rightrightarrows Z$, where $Y$ and $Z$ are metric spaces, is called
\textit{a Carath\'{e}odory map}, if for every $y \in Y$ the map $F(\cdot, y)$ is measurable and for every $x \in \Omega$
the map $F(x, \cdot)$ is continuous (see~\cite[Def.~8.2.7]{AubinFrankowska}). As is well-known, in the case when 
$Z = \xR^k$ and $F$ is compact-valued, the map $F(x, \cdot)$ is continuous iff it is Hausdorff continuous.

The following definition describes natural assumptions on the integrand $f = f(x, u, \xi)$ ensuring the
codifferentiability of the functional $\mathcal{I}$.

\begin{definition} \label{def:CodiffCond}
We say that $f$ satisfies \textit{the codifferentiability conditions} of order $p$, if

(1) $f$ is a Carath\'{e}odory function satisfying \textit{the growth condition} of order $p$, i.e. there exist an a.e.
nonnegative function $\beta \in \xL^{1}(\Omega)$ and $C > 0$ such that
$$
  |f(x, u, \xi)| \le \beta(x) + C \big( |u|^p + |\xi|^p \big) 
$$
for a.e. $x \in \Omega$ and for all $(u, \xi) \in \xR^m \times \xR^{m \times d}$ in the case $1 \le p < + \infty$, and
for any $N > 0$ there exists an a.e. nonnegative function $\beta_N \in \xL^1(\Omega)$ such that 
$|f(x, u, \xi)| \le \beta_N(x)$ for a.e. $x \in \Omega$ and all $(u, \xi) \in \xR^m \times \xR^{m \times d}$ with
$\max\{ |u|, |\xi| \} \le N$ in the case $p = +\infty$;

(2) for a.e. $x \in \Omega$ the function $(u, \xi) \mapsto f(x, u, \xi)$ is 
codifferentiable on $\xR^m \times \xR^{m \times d}$ and its codifferential mapping $D_{u, \xi} f(\cdot)$
is a Carath\'{e}odory map (i.e. both $\underline{d}_{u, \xi} f(\cdot)$ and $\overline{d}_{u, \xi} f(\cdot)$ are
Carath\'{e}odory maps) satisfying \textit{the growth condition} of order $p$, i.e. there exist
$C > 0$ and a.e. nonnegative functions $\beta \in \xL^1(\Omega)$ and $\gamma \in \xL^{p'}(\Omega)$ such that
for a.e. $x \in \Omega$ and for all $(u, \xi) \in \xR^m \times \xR^{m \times d}$, 
$(a, v_1, v_2) \in \underline{d}_{u, \xi} f(x, u, \xi)$, and $(b, w_1, w_2) \in \overline{d}_{u, \xi} f(x, u, \xi)$ one
has
$$
  \max\big\{ |a|, |b| \big\} \le \beta(x) + C\big( |u|^p + |\xi|^p \big), \quad
  \max\big\{ |v_1|, |v_2|, |w_1|, |w_2| \big\} \le \gamma(x) + C\big( |u|^{p - 1} + |\xi|^{p - 1} \big)
$$
in the case $1 \le p < + \infty$, and for any $N > 0$ there exists an a.e. nonnegative function 
$\beta_N \in \xL^1(\Omega)$ such that $\max\big\{ |a|, |v_1|, |v_2|, |b|, |w_1|, |w_2| \big\} \le \beta_N(x)$
for a.e. $x \in \Omega$ and for all $(a, v_1, v_2) \in \underline{d}_{u, \xi} f(x, u, \xi)$, 
$(b, w_1, w_2) \in \overline{d}_{u, \xi} f(x, u, \xi)$, and $(u, \xi) \in \xR^m \times \xR^{m \times d}$
with $\max\{ |u|, |\xi| \} \le N$ in the case $p = + \infty$.
\end{definition}

\begin{remark}
In the case $1 \le p < + \infty$ the growth conditions from the previous definition can be weakened with the use of the
Sobolev imbedding theorem (cf.~\cite[Sect.~3.4.2]{Dacorogna}). In particular, if $d = 1$, then it is sufficient to
suppose that for any $N > 0$ there exists $C_N > 0$ and a.e. nonnegative functions $\beta_N \in \xL^1(\Omega)$ and
$\gamma_N \in \xL^{p'}(\Omega)$ such that for a.e. $x \in \Omega$, for all $(u, \xi) \in \xR^m \times \xR^m$
with $|u| \le N$ and for all 
$(a, v_1, v_2) \in \underline{d}_{u, \xi} f(x, u, \xi) \cup \overline{d}_{u, \xi} f(x, u, \xi)$ 
one has $|a| \le \beta_N(x) + C_N |\xi|^p$ and $\max\{ |v_1|, |v_2| \} \le \gamma_N(x) + C_N |\xi|^{p - 1}$.
\end{remark}

Our aim is to prove that the codifferentiability conditions from the definition above guarantee that the functional
$\mathcal{I}$ is codifferentiable on $\xWsp^{1, p}(\Omega; \xR^m)$. Due to some technical
difficulties, in the case $p = 1$ we need to assume that the set $\Omega$ has \textit{the segment property}
\cite[p.~53--54]{Adams}, i.e. that for every $x$ from the boundary of $\Omega$ there exist a neighbourhood 
$U_x \subset \xR^d$ of $x$ and a nonzero vector $y_x \in \xR^d$ such that for any 
$z \in \cl \Omega \cap U_x$ one has $z + t y_x \in \Omega$ for all $t \in (0, 1)$. The segment property ensures that the
set $\Omega$ has a $(d - 1)$-dimensional boundary and cannot simultaneously lie on both sides of any given part of its
boundary (i.e. there are no cuts). Furthermore, it ensures that the space of continuously differentiable functions is
dense in $\xWsp^{1, p}(\Omega)$ (see, e.g. \cite[Thrm.~3.18]{Adams}).

\begin{theorem} \label{thm:CodiffIntegralFunctional}
Let $f$ satisfy the codifferentiability conditions of order $p \in [1, + \infty]$ and let either $1 < p \le + \infty$
or the set $\Omega$ be bounded and have the segment property. Then the functional $\mathcal{I}$ is correctly defined on 
$\xWsp^{1, p}(\Omega; \xR^m)$, codifferentiable at every $u \in \xWsp^{1, p}(\Omega; \xR^m)$, and the pair 
$D \mathcal{I}(u) = [\underline{d} \mathcal{I}(u), \overline{d} \mathcal{I}(u)]$ with
\begin{multline} \label{eq:HypodiffIntegralFunc}
  \underline{d} \mathcal{I}(u) = \Big\{ (A, x^*) \in \xR \times (\xWsp^{1, p}(\Omega; \xR^m))^* \Bigm|
  A = \int_{\Omega} a(x) \, dx, \quad
  \langle x^*, h \rangle = 
  \int_{\Omega} \big( \langle v_1(x), h(x) \rangle + \langle v_2(x), \nabla h(x) \rangle \big) \, dx \\
  \quad \forall h \in \xWsp^{1, p}(\Omega; \xR^m), \quad
  (a(\cdot), v_1(\cdot), v_2(\cdot)) \text{ is a measurable selection of the map } 
  \underline{d}_{u, \xi} f(\cdot, u(\cdot), \nabla u(\cdot)) \Big\}
\end{multline}
and
\begin{multline*}
  \overline{d} \mathcal{I}(u) = \Big\{ (B, y^*) \in \xR \times (\xWsp^{1, p}(\Omega; \xR^m))^* \Bigm|
  B = \int_{\Omega} b(x) \, dx, \quad
  \langle y^*, h \rangle = 
  \int_{\Omega} \big( \langle w_1(x), h(x) \rangle + \langle w_2(x), \nabla h(x) \rangle \big) \, dx \\
  \quad \forall h \in \xWsp^{1, p}(\Omega; \xR^m), \quad
  (b(\cdot), w_1(\cdot), w_2(\cdot)) \text{ is a measurable selection of the map } 
  \overline{d}_{u, \xi} f(\cdot, u(\cdot), \nabla u(\cdot)) \Big\}.
\end{multline*}
is a codifferential of $\mathcal{I}$ at $u$. Furthermore, the multifunctions $\underline{d} \mathcal{I}(\cdot)$ and
$\overline{d} \mathcal{I}(\cdot)$ are Hausdorff continuous, i.e. the functional $\mathcal{I}$ is continuously
codifferentiable on $\xWsp^{1, p}(\Omega; \xR^m)$, provided either $1 < p \le + \infty$ or the set-valued maps
$\underline{d}_{u, \xi} f(\cdot)$ and $\overline{d}_{u, \xi} f(\cdot)$ have the form
\begin{equation} \label{eq:IntegrandCodiffL1}
  \underline{d}_{u, \xi} f(x, u, \xi) = \co\big\{ (f_i(x, u, \xi), v_{1i}, v_{2i}) \mid i \in I \big\}, \quad
  \overline{d}_{u, \xi} f(x, u, \xi) = \co\big\{ (g_j(x, u, \xi), w_{1j}, w_{2j}) \mid j \in J \big\}
\end{equation}
for some vectors $v_{1i}, w_{1j} \in \xR^m$, $v_{2i}, w_{2j} \in \xR^{m \times d}$ and Carath\'{e}odory
functions $f_i, g_j \colon \Omega \times \xR^m \times \xR^{m \times d} \to \xR$, 
where $i \in I = \{ 1, \ldots, \ell \}$ and $j \in J = \{ 1, \ldots, r \}$.
\end{theorem}

\begin{remark}
The assumption that in the case $p = 1$ the set-valued maps $\underline{d}_{u, \xi} f(\cdot)$ and 
$\overline{d}_{u, \xi} f(\cdot)$ have the form \eqref{eq:IntegrandCodiffL1} might seem unnatural at first glance.
However, it should be noted that this assumption is satisfied in many particular examples. Furthermore, with the use of
the codifferential calculus \cite{DemRub_book} one can easily show that this assumption is preserved under addition,
multiplication by scalar, pointwise maximum, and pointwise minimum. For example, if $d = m = 1$, 
$\Omega = (\alpha, \beta)$, and
$$
  \mathcal{I}(u) = \int_{\alpha}^{\beta} \max\{ |u'(x)| - |u(x)|,   0 \} \, dx,
$$
then $f(x, u, \xi) = \max\{ |\xi| - |u|, 0 \} = \max\{ |\xi|, |u| \} + \min\{ u, - u \}$ and 
by applying the codifferential calculus \cite{DemRub_book} one gets
$$
  \underline{d}_{u, \xi} f(x, u, \xi) = \co\left\{ 
  \begin{pmatrix} \pm \xi - \max\{ |\xi|, |u| \} \\ 0 \\ \pm 1 \end{pmatrix}, 
  \begin{pmatrix} \pm u - \max\{ |\xi|, |u| \} \\ \pm 1 \\ 0 \end{pmatrix} \right\}, \quad
  \overline{d}_{u, \xi} f(x, u, \xi) = \co\left\{ \begin{pmatrix} \pm u + |u| \\ \pm 1 \\ 0 \end{pmatrix} \right\},
$$
i.e. assumption \eqref{eq:IntegrandCodiffL1} is satisfied.
\end{remark}

We split the proof of Theorem~\ref{thm:CodiffIntegralFunctional} into four parts, each of which is formulated as a
separate lemma. Before we proceed to these lemmas, it should be remarked that Theorem~\ref{thm:CodiffIntegralFunctional}
significantly improves \cite[Thrm.~5.1]{Dolgopolik_CalcVar}, since it states that the functional $\mathcal{I}$ is
\textit{continuously} codifferentiable for any $1 \le p \le + \infty$ and demonstrates that the rather restrictive and
obscure assumption on the uniform codifferentiability of the integrand $f$ with respect to the space 
$\xWsp^{1, p}(\Omega; \xR^m)$ (see~\cite[Def.~4.18]{Dolgopolik_CalcVar}) is redundant. Furthermore, in the case 
$1 < p \le + \infty$ Theorem~\ref{thm:CodiffIntegralFunctional} extends \cite[Thrm.~5.1]{Dolgopolik_CalcVar} to the case
of unbounded domains $\Omega$ and domains not having a segment property. Finally, 
Theorem~\ref{thm:CodiffIntegralFunctional} gives a positive answer to the second question raised by the author in
\cite[Remark~4.22]{Dolgopolik_CalcVar}.

We start with a simple technical lemma on the function $\Phi_f$ defined in \eqref{eq:CodiffIntegrand1}.

\begin{lemma} \label{lem:CodiffOfConvexTerm}
Suppose that for a.e. $x \in \Omega$ and for all $(u, \xi) \in \xR^m \times \xR^{m \times d}$ the function
$(u, \xi) \mapsto f(x, u, \xi)$ is codifferentiable and let $\Phi_f$ be defined as in \eqref{eq:CodiffIntegrand1}.
Then for any $u, h \in \xWsp^{1, p}(\Omega; \xR^m)$, for a.e. $x \in \Omega$, and for all $t \in \xR$ the
function $g(t) = \Phi_f(x, u(x), \nabla u(x); t h(x), t \nabla h(x))$ is codifferentiable and for any 
$t \in \xR$ the pair $D g(t) = [\underline{d} g(t), \{ 0 \}]$ is a codifferential of $g$ at $t$, where
\begin{align*}
  \underline{d} g(t) = \Big\{ (a_g, v_g) \in \xR \times \xR \Bigm|
  a_g &= a + \langle v_1, t h(x) \rangle + \langle v_2, t \nabla h(x) \rangle - g(t), \\
  v_g &= \langle v_1, h(x) \rangle + \langle v_2, \nabla h(x) \rangle, \enspace
  (a, v_1, v_2) \in \underline{d}_{u, \xi} f(x, u(x), \nabla u(x)) \Big\}.
\end{align*}
\end{lemma}

\begin{proof}
By the definition of $\Phi_f$ for any $t, \Delta t \in \xR$ one has
\begin{multline*}
  g(t + \Delta t) - g(t) = \max \big( a + \langle v_1, (t + \Delta t) h(x) \rangle 
  + \langle v_2, (t + \Delta t) \nabla h(x) \rangle \big) - g(t) \\
  = \max \Big( \big[ a + \langle v_1, t h(x) \rangle 
  + \langle v_2, t \nabla h(x) \rangle - g(t) \big]
  + \Delta t \langle v_1, h(x) \rangle + \Delta t \langle v_2, \nabla h(x) \rangle \Big)
  = \max_{(a_g, v_g) \in \underline{d} g(t)} (a_g + v_g \Delta t),
\end{multline*}
where the first two maximums are taken over all $(a, v_1, v_2) \in \underline{d}_{u, \xi} f(x, u(x), \nabla u(x))$
and the set $\underline{d} g(t)$ is defined in the formulation of the lemma. The set $\underline{d} g(t)$ is obviously
convex and compact as the image of the set $\underline{d}_{u, \xi} f(x, u(x), \nabla u(x))$ under the affine map
$$
  (a, v_1, v_2) \mapsto 
  \Big( a + \langle v_1, t h(x) \rangle + \langle v_2, t \nabla h(x) \rangle - g(t),
  \langle v_1, h(x) \rangle + \langle v_2, \nabla h(x) \rangle \Big) \in \xR \times \xR.
$$
Moreover, one has
$$
  \max_{(a_g, v_g) \in \underline{d} g(t)} a_g 
  = \max \big( a + \langle v_1, t h(x) \rangle + \langle v_2, t \nabla h(x) \rangle \big) - g(t) 
  = g(t) - g(t) = 0,
$$
where the second maximum is taken over all $(a, v_1, v_2) \in \underline{d}_{u, \xi} f(x, u(x), \nabla u(x))$.
Thus, the function $g$ is codifferentiable at every $t \in \xR$ and the pair 
$D g(\cdot) = [\underline{d} g(\cdot), \{ 0 \}]$ is its codifferential (see~Def.~\ref{def:CodiffFunc}).
\end{proof}

Next we show that the increment $\mathcal{I}(u + \alpha h) - \mathcal{I}(u)$ of the functional $\mathcal{I}$ can be
approximated by a DC function defined via the sets $\underline{d} \mathcal{I}(u)$ and $\overline{d} \mathcal{I}(u)$
from Theorem~\ref{thm:CodiffIntegralFunctional}. Note that these sets are nonempty, since measurable selections of the
multifunctions $\underline{d}_{u, \xi} f(\cdot, u(\cdot), \nabla u(\cdot))$ and 
$\overline{d}_{u, \xi} f(\cdot, u(\cdot), \nabla u(\cdot))$ exist 
by  \cite[Thrms.~8.1.3 and 8.2.8]{AubinFrankowska}.

The statement of the following lemma coincides with that of \cite[Lemma~5.2]{Dolgopolik_CalcVar} with the only
difference being the fact that here were remove the obscure assumption on uniform codifferentiability of the integrand
with respect to the Sobolev space. In \cite{Dolgopolik_CalcVar} this assumption ensured that one can pass to the limit
under the integral sign. Below we prove that one can pass to the limit without this assumption by applying Lebesgue's
dominated convergence theorem and obtaining necessary estimates with the use of the mean value theorem for
codifferentiable functions \cite[Prp.~2]{Dolgopolik_CodiffDescent}.

\begin{lemma} \label{lem:AlmostCodiff}
Let $f$ satisfy the codifferentiability conditions of order $p \in [1, + \infty]$ and the sets 
$\underline{d} \mathcal{I}(u)$ and $\overline{d} \mathcal{I}(u)$ be defined as in
Theorem~\ref{thm:CodiffIntegralFunctional}. Then the functional $\mathcal{I}$ is correctly defined 
on $\xWsp^{1, p}(\Omega; \xR^m)$, $\underline{d} \mathcal{I}(u), 
\overline{d} \mathcal{I}(u) \subset \xR \times (\xWsp^{1, p}(\Omega; \xR^m))^*$, and 
\begin{equation} \label{eq:CodiffDecomIntegral}
  \lim_{\alpha \to + 0} \frac{1}{\alpha} \Big| \mathcal{I}(u + \alpha h) - \mathcal{I}(u)
  - \max_{(A, x^*) \in \underline{d} \mathcal{I}(u)} ( A + \langle x^*, \alpha h \rangle ) 
  - \min_{(B, y^*) \in \overline{d} \mathcal{I}(u)} ( B + \langle y^*, \alpha h \rangle ) \Big| = 0.
\end{equation}
for all $u, h \in \xWsp^{1, p}(\Omega; \xR^m)$.
\end{lemma}

\begin{proof}
Fix any $u \in \xWsp^{1, p}(\Omega; \xR^m)$. By our assumption $f$ is a Carath\'{e}odory function satisfying
the growth condition (see~Def.~\ref{def:CodiffCond}). Therefore, as is well-known, 
the function $f(\cdot, u(\cdot), \nabla u(\cdot))$ is measurable and belongs to $\xL^1(\Omega)$, which implies that
$\mathcal{I}(u)$ is correctly defined and finite.

Let us verify that the sets $\underline{d} \mathcal{I}(u)$ and $\overline{d} \mathcal{I}(u)$ are correctly defined.
Indeed, fix any measurable selection $(a(\cdot), v_1(\cdot), v_2(\cdot))$ of the multifunction
$\underline{d}_{u, \xi} f(\cdot, u(\cdot), \nabla u(\cdot))$. By the growth condition on $D_{u, \xi} f(\cdot)$
(see~Def.~\ref{def:CodiffCond}) there exist $C > 0$ and a.e. nonnegative functions $\beta \in \xL^1(\Omega)$ and 
$\gamma \in \xL^{p'}(\Omega)$ such that
$$
  |a(x)| \le \beta(x) + C \big( |u(x)|^p + |\nabla u(x)|^p \big), \quad
  \max\big\{ |v_1(x)|, |v_2(x)| \big\} \le \gamma(x) + C \big( |u(x)|^{p - 1} + |\nabla u(x)|^{p - 1} \big)
$$
for a.e. $x \in \Omega$ in the case $1 \le p < + \infty$, and there exists $\beta_N \in \xL^1(\Omega)$ such that 
for a.e. $x \in \Omega$ one has $\max\{ |a(x)|, |v_1(x)|, |v_2(x)| \} \le \beta_N(x)$ in the case $p = + \infty$ (here
$N = \| u \|_{1, \infty}$). Hence with the use of H\"{o}lder's inequality one obtains 
that $a(\cdot) \in \xL^{1}(\Omega)$, $v_1(\cdot) \in \xL^{p'}(\Omega; \xR^m)$, 
and $v_2(\cdot) \in \xL^{p'}(\Omega; \xR^{m \times d})$. Therefore the integral $A = \int_{\Omega} a(x) dx$ is correctly
defined and finite, while the functional $x^*$ defined as
$$
  \langle x^*, h \rangle = 
  \int_{\Omega} \big( \langle v_1(x), h(x) \rangle + \langle v_2(x), \nabla h(x) \rangle \big) \, dx 
  \quad \forall h \in \xWsp^{1, p}(\Omega; \xR^m)
$$
is a continuous linear functional on $\xWsp^{1, p}(\Omega; \xR^m)$, i.e. the hypodifferential $\underline{d} \mathcal{I}(u)$ is
correctly defined and $\underline{d} \mathcal{I}(u) \subset \xR \times (\xWsp^{1, p}(\Omega; \xR^m))^*$. 
The fact that $\overline{d} \mathcal{I}(u) \subset \xR \times (\xWsp^{1, p}(\Omega; \xR^m)^*$ is proved in
the same way.

Choose any $h \in \xWsp^{1, p}(\Omega; \xR^m)$ and a sequence $\{ \alpha_n \} \subset (0, + \infty)$ converging to
zero. Let us prove that
\begin{multline} \label{CodiffUnderIntegralSign}
  \lim_{n \to \infty} \frac{1}{\alpha_n} \Big| \mathcal{I}(u + \alpha_n h) - \mathcal{I}(u)
  - \int_{\Omega} \Phi_f\big( x, u(x), \nabla u(x); \alpha_n h(x), \alpha_n \nabla h(x) \big) \, dx \\
  - \int_{\Omega} \Psi_f\big( x, u(x), \nabla u(x); \alpha_n h(x), \alpha_n \nabla h(x) \big) \, dx
  \Big| = 0,
\end{multline}
where the functions $\Phi_f$ and $\Psi_f$ are defined in \eqref{eq:CodiffIntegrand1}, \eqref{eq:CodiffIntegrand2}.
Indeed, for any $n \in \mathbb{N}$ and $x \in \Omega$ denote
\begin{multline} \label{eq:CodiffDecomUnderIntegral}
  g_n(x) = \frac{1}{\alpha_n} \Big( f(x, u(x) + \alpha_n h(x), \nabla u(x) + \alpha_n \nabla h(x)) 
  - f(x, u(x), \nabla u(x)) \\
  - \Phi_f\big( x, u(x), \nabla u(x); \alpha_n h(x), \alpha_n \nabla h(x) \big) 
  - \Psi_f\big( x, u(x), \nabla u(x); \alpha_n h(x), \alpha_n \nabla h(x) \big) \Big).
\end{multline}
Our aim is to prove \eqref{CodiffUnderIntegralSign} by applying Lebesgue's dominated convergence theorem to the sequence
of functions $\{ g_n \}$. Firstly, note that by the definition of codifferential $g_n(x) \to 0$ as $n \to \infty$ for
a.e. $x \in \Omega$. Next, we show that $g_n \in \xL^1(\Omega)$ for all $n \in \mathbb{N}$.

From the fact that the integrand $f$ satisfied the growth condition it follows that the first two
terms in the definition of $g_n$ belong to $\xL^1(\Omega)$. Let us check that
the function $\eta_n(x) = \Phi_f\big( x, u(x), \nabla u(x); \alpha_n h(x), \alpha_n \nabla h(x) \big)$, $x \in \Omega$,
belongs to $\xL^1(\Omega)$ as well. The proof of this fact for the function $\Psi_f$ is exactly the same.

By the codifferentiability conditions $\underline{d}_{u, \xi} f(\cdot)$ is a Carath\'{e}odory map, which by
\cite[Thrm.~8.2.8]{AubinFrankowska} implies that the multifunction 
$\underline{d}_{u, \xi} f(\cdot, u(\cdot), \nabla u(\cdot))$ is measurable. The map
$(x, (a, v_1, v_2)) \mapsto \langle a + \alpha_n \langle v_1, h(x) \rangle + \alpha_n \langle v_2, \nabla h(x) \rangle$
is obviously a Carath\'{e}odory function. Hence by the definitions of $\Phi_f$ (see~\eqref{eq:CodiffIntegrand1}) and
$\eta_n$ and the theorem on the measurability of marginal functions \cite[Thrm.~8.2.11]{AubinFrankowska} one obtains
that the function $\eta_n$ is measurable. Moreover, by the growth condition on the codifferential mapping 
$D_{u, \xi} f(\cdot)$ (see~Def.~\ref{def:CodiffCond}) there exist $C > 0$ and a.e. nonnegative functions 
$\beta \in \xL^1(\Omega)$ and $\gamma \in \xL^{p'}(\Omega)$ such that
\begin{multline*}
  |\eta_n(x)| = \big| \Phi_f(x, u(x), \nabla u(x); \alpha_n h(x), \alpha_n \nabla h(x)) \big|
  \le \beta(x) + C\big( |u(x)|^p + |\nabla u(x)|^p \big) \\ 
  + \alpha_n \big( \gamma(x) + C( |u(x)|^{p - 1} + |\nabla u(x)|^{p - 1} ) \big) \big( |h(x)| + |\nabla h(x)| \big)
\end{multline*}
for a.e. $x \in \Omega$ in the case $1 \le p < + \infty$, and there exists an a.e. nonnegative function 
$\beta_N \in \xL^1(\Omega)$ such that
$$
  |\eta_n(x)| \le \beta_N(x) \big( 1 + \alpha_n |h(x)| + \alpha_n |\nabla h(x)| \big)
$$
for a.e. $x \in \Omega$ in the case $p = + \infty$ (here $N = \| u \|_{1, \infty}$). Hence taking into
account the fact that $u, h \in \xWsp^{1, p}(\Omega; \xR^m)$ and applying H\"{o}lder's inequality in the case 
$1 < p < + \infty$ one obtains that $\eta_n \in \xL^1(\Omega)$ for all $n \in \mathbb{N}$, which implies that 
$g_n \in \xL^1(\Omega)$ for all $n \in \mathbb{N}$ as well.

Now we prove that the sequence $\{ g_n \}$ is dominated by some integrable function. Indeed, by the mean value theorem
for codifferentiable functions \cite[Proposition~2]{Dolgopolik_CodiffDescent} for any $n \in \mathbb{N}$ and for a.e. 
$x \in \Omega$ one can find $\theta_n(x) \in (0, \alpha_n)$, and triplets
$(0, v_{1n}(x), v_{2n}(x)) \in 
\underline{d}_{u, \xi} f(x, u(x) + \theta_n(x) h(x), \nabla u(x) + \theta_n(x) \nabla h(x))$,
and
$(0, w_{1n}(x), w_{2n}(x)) \in 
\overline{d}_{u, \xi} f(x, u(x) + \theta_n(x) h(x), \nabla u(x) + \theta_n(x) \nabla h(x))$ such that
\begin{multline*}
  \frac{1}{\alpha_n} 
  \big( f(x, u(x) + \alpha_n h(x), \nabla u(x) + \alpha_n \nabla h(x)) - f(x, u(x), \nabla u(x)) \big) \\
  = \langle v_{1n}(x) + w_{1n}(x), h(x) \rangle + \langle v_{2n}(x) + w_{2n}(x), \nabla h(x) \rangle.
\end{multline*}
Hence by the growth condition on $D_{u, \xi} f(\cdot)$ (see~Def.~\ref{def:CodiffCond}) there exist $C > 0$ and a.e.
nonnegative function $\gamma \in \xL^{p'}(\Omega)$ such that
\begin{multline*}
  \frac{1}{\alpha_n} 
  \big| f(x, u(x) + \alpha_n h(x), \nabla u(x) + \alpha_n \nabla h(x)) - f(x, u(x), \nabla u(x)) \big| 
  \\
  \le 2 
  \big( \gamma(x) + C( |u(x) + \theta_n(x) h(x)|^{p - 1} + |\nabla u(x) + \theta_n(x) \nabla h(x)|^{p - 1} ) \big) 
  \big( |h(x)| + |\nabla h(x)| \big)
\end{multline*}
for a.e. $x \in \Omega$ in the case $1 \le p < + \infty$, and there exists an a.e. nonnegative function 
$\beta_N \in \xL^1(\Omega)$ such that
$$
  \frac{1}{\alpha_n} 
  \big| f(x, u(x) + \alpha_n h(x), \nabla u(x) + \alpha_n \nabla h(x)) - f(x, u(x), \nabla u(x)) \big| 
  \le 2 \beta_N(x) \big( |h(x)| + |\nabla h(x)| \big)
$$
for a.e. $x \in \Omega$ in the case $p = + \infty$, where $N = \| u \|_{1, \infty} + \alpha_* \| h \|_{1, \infty}$ and
$\alpha_* = \max_{n \in \mathbb{N}} \alpha_n$. Now, taking into account the fact that 
$u, h \in \xWsp^{1, p}(\Omega; \xR^m)$ and applying H\"{o}lder's inequality in the case $1 < p < + \infty$
one gets that the first two terms in \eqref{eq:CodiffDecomUnderIntegral} are dominated by an integrable function
independent of $n$. 

Let us now turn to the third term in \eqref{eq:CodiffDecomUnderIntegral}. The fact that the last term is dominated by
an integrable function can be proved in exactly the same way. By applying the mean value theorem for codifferentiable
functions and Lemma~\ref{lem:CodiffOfConvexTerm} one obtains that for any $n \in \mathbb{N}$ and for a.e. $x \in \Omega$
there exist $\theta_n(x) \in (0, \alpha_n)$ and 
$(a_n(x), v_{1n}(x), v_{2n}(x)) \in \underline{d}_{u, \xi} f(x, u(x), \nabla u(x))$ such that
\begin{gather*}
  \Phi_f(x, u(x), \nabla u(x); \theta_n(x) h(x), \theta_n(x) \nabla h(x))
  = a_n(x) + \langle v_{1n}(x), \theta_n(x) h(x) \rangle + \langle v_{2n}(x), \theta_n(x) \nabla h(x) \rangle \\
  \frac{1}{\alpha_n} \Phi_f\big( x, u(x), \nabla u(x); \alpha_n h(x), \alpha_n \nabla h(x) \big)
  = \langle v_{1n}(x), h(x) \rangle + \langle v_{2n}(x), \nabla h(x) \rangle
\end{gather*}
(here we used the fact that $\Phi_f(x, u(x), \nabla u(x); 0, 0) = 0$ by the definition of codifferential).
Hence utilising the growth condition on $D_{u, \xi} f(\cdot)$ in the same way as above one can easily verify that 
the third term in \eqref{eq:CodiffDecomUnderIntegral} is dominated by an integrable function independent of $n$ as
well. Consequently, applying Lebesgue's dominated convergence theorem one obtains that 
$\int_{\Omega} g_n(x) \, dx \to 0$ as $n \to \infty$ or, equivalently, \eqref{CodiffUnderIntegralSign} holds true 
(see the definition of $g_n$, formula \eqref{eq:CodiffDecomUnderIntegral}).

Let us check that
\begin{equation} \label{eq:IntegralOfMaxFunc_Filippov}
  \int_{\Omega} \Phi_f\big( x, u(x), \nabla u(x); \alpha_n h(x), \alpha_n \nabla h(x) \big) \, dx =
  \max_{(A, x^*) \in \underline{d} \mathcal{I}(u)} ( A + \langle x^*, \alpha_n h \rangle )
\end{equation}
for all $n \in \mathbb{N}$, where $\underline{d} \mathcal{I}(u)$ is defined in
Theorem~\ref{thm:CodiffIntegralFunctional}. The validity of a similar equality involving $\Psi_f$ and 
$\overline{d} \mathcal{I}(u)$ can be proved in the same way. Then applying \eqref{CodiffUnderIntegralSign} one obtains
that equality \eqref{eq:CodiffDecomIntegral} holds true and the proof is complete.

By the definition of $\Phi_f$ for any measurable selection $(a(\cdot), v_1(\cdot), v_2(\cdot))$ of the multifunction
$\underline{d}_{u, \xi} f(\cdot, u(\cdot), \nabla u(\cdot))$ one has
$$
  a(x) + \alpha_n \langle v_1(x), h(x) \rangle + \alpha_n \langle v_2(x), \nabla h(x) \rangle
  \le \Phi_f\big( x, u(x), \nabla u(x); \alpha_n h(x), \alpha_n \nabla h(x) \big)
$$
for a.e. $x \in \Omega$ and for all $n \in \mathbb{N}$, which obviously implies that the inequality
$$
  \sup_{(A, x^*) \in \underline{d} \mathcal{I}(u)} ( A + \langle x^*, \alpha_n h \rangle )
  \le \int_{\Omega} \Phi_f\big( x, u(x), \nabla u(x); \alpha_n h(x), \alpha_n \nabla h(x) \big) \, dx
$$
holds true for all $n \in \mathbb{N}$ (see~\eqref{eq:HypodiffIntegralFunc}). On the other hand, observe that by
definition
\begin{multline*}
  \Phi_f\big( x, u(x), \nabla u(x); \alpha_n h(x), \alpha_n \nabla h(x) \big) \in
  \Big\{ a + \alpha_n \langle v_1, h(x) \rangle + \alpha_n \langle v_2, \nabla h(x) \rangle \in \xR \Bigm| \\
  (a, v_1, v_2) \in \underline{d}_{u, \xi} f(x, u(x), \nabla u(x)) \Big\}
\end{multline*}
for a.e. $x \in \Omega$ and for all $n \in \mathbb{N}$. As was noted above, from the codifferentiability conditions in
follows that the multifunction $\underline{d}_{u, \xi} f(\cdot, u(\cdot), \nabla u(\cdot))$ is measurable. Therefore,
by Filippov's theorem \cite[Thrm.~8.2.10]{AubinFrankowska} for any $n \in \mathbb{N}$ there exists a measurable
selection $(a(\cdot), v_1(\cdot), v_2(\cdot))$ of the set-valued map 
$\underline{d}_{u, \xi} f(\cdot, u(\cdot), \nabla u(\cdot))$ such that
$$
  \Phi_f\big( x, u(x), \nabla u(x); \alpha_n h(x), \alpha_n \nabla h(x) \big)
  = a(x) + \alpha_n \langle v_1(x), h(x) \rangle + \alpha_n \langle v_2(x), \nabla h(x) \rangle
$$
for a.e. $x \in \Omega$, which implies that for the corresponding element $(A, x^*) \in \underline{d} \mathcal{I}(u)$
(see~\eqref{eq:HypodiffIntegralFunc}) one has
$$
  \int_{\Omega} \Phi_f\big( x, u(x), \nabla u(x); \alpha_n h(x), \alpha_n \nabla h(x) \big) \, dx 
  = A + \langle x^*, \alpha_n h \rangle.
$$
Thus, equality \eqref{eq:IntegralOfMaxFunc_Filippov} holds true and the proof is complete.
\end{proof}

Next we prove that the pair $D \mathcal{I}(u) = [\underline{d} \mathcal{I}(u), \overline{d} \mathcal{I}(u)]$ defined in
Theorem~\ref{thm:CodiffIntegralFunctional} is indeed a codifferential of $\mathcal{I}$ at $u$. According to the
definition of codifferential (see~Def.~\ref{def:CodiffFunc}), we need to prove that both sets 
$\underline{d} \mathcal{I}(u)$ and $\overline{d} \mathcal{I}(u)$ are convex and compact in the corresponding product
topology. A proof of this result in the case when $\Omega$ is bounded and has the segment property was given in 
\cite[Lemmas~5.4 and 5.6]{Dolgopolik_CalcVar}. Therefore, below we give a proof of the case $1 < p \le + \infty$ only.

Let us note that that we managed to remove the assumptions on the set $\Omega$ in the case $1 < p \le + \infty$ by
using a completely different proof technique. Instead of reducing the proof of the compactness of 
$\underline{d} \mathcal{I}(u)$ and $\overline{d} \mathcal{I}(u)$ to the proof of the closedness of an auxiliary Aumann
integral as it is done in \cite{Dolgopolik_CalcVar}, here we prove that the sets of measurable selections of the
multifunctions $\underline{d}_{u, \xi} f(\cdot, u(\cdot), \nabla u(\cdot)$ and
$\overline{d}_{u, \xi} f(\cdot, u(\cdot), \nabla u(\cdot))$ are compact in a suitable topology and then conclude that
the sets $\underline{d} \mathcal{I}(u)$ and $\overline{d} \mathcal{I}(u)$ are compact in the product topology as
continuous images of the corresponding sets of measurable selections.

\begin{lemma} \label{lem:ClosednessOfAlmostCodiff}
Let $f$ satisfy the codifferentiability conditions of order $p \in [1, + \infty]$ and let either $1 < p \le + \infty$
or the set $\Omega$ be bounded and have the segment property. Then for any $u \in \xWsp^{1, p}(\Omega; \xR^m)$ the
sets $\underline{d} \mathcal{I}(u)$ and $\overline{d} \mathcal{I}(u)$ defined in
Theorem~\ref{thm:CodiffIntegralFunctional} are convex and compact in the topology $\tau_{\xR} \times w^*$.
Furthermore, the equalities $\max\{ A \colon (A, x^*) \in \underline{d} \mathcal{I}(u) \} = 
\min\{ B \colon (B, y^*) \in \overline{d} \mathcal{I}(u) \} = 0$ hold true.
\end{lemma}

\begin{proof}
Fix any $u \in \xWsp^{1, p}(\Omega; \xR^m)$. We prove this lemma only for the hypodifferential 
$\underline{d} \mathcal{I}(u)$, since the proof for the hyperdifferential $\overline{d} \mathcal{I}(u)$ is exactly the
same. 

Choose any $(A_1, x_1^*), (A_2, x_2^*) \in \underline{d} \mathcal{I}(u)$, and let 
$z_i(\cdot) = (a_i(\cdot), v_{1i}(\cdot), v_{2i}(\cdot))$ be a measurable selection of the set-valued mapping 
$\underline{d}_{u, \xi} f(\cdot, u(\cdot), \nabla u(\cdot))$ corresponding to $(A_i, x_i^*)$, $i \in \{ 1, 2 \}$. For
a.e. $x \in \Omega$ the set $\underline{d}_{u, \xi} f(x, u(x), \nabla u(x))$ is convex by definition.
Consequently, for any $\alpha \in [0, 1]$ the map $\alpha z_1(\cdot) + (1 - \alpha) z_2(\cdot)$ is a measurable
selection of the set-valued map $\underline{d}_{u, \xi} f(\cdot, u(\cdot), \nabla u(\cdot))$, which obviously
corresponds to the pair $Z_{\alpha} = (\alpha A_1 + (1 - \alpha) A_2, \alpha x_1^* + (1 - \alpha) x_2^*))$. Therefore, 
$Z_{\alpha} \in \underline{d} \mathcal{I}(u)$ for any $\alpha \in [0, 1]$ and one can conclude that the set
$\underline{d} \mathcal{I}(u)$ is convex.

By the definition of codifferential for a.e. $x \in \Omega$ and 
for all $(a, v_1, v_2) \in \underline{d}_{u, \xi} f(x, u(x), \nabla u(x))$ one
has $a \le 0$, which obviously implies that for any $(A, x^*) \in \underline{d} \mathcal{I}(u)$ one has $A \le 0$.
Furthermore, by definition $\max\{ a \mid (a, v_1, v_2) \in \underline{d}_{u, \xi} f(x, u(x), \nabla u(x)) \} = 0$ 
for a.e. $x \in \Omega$, that is, for a.e. $x \in \Omega$ one has 
$0 \in \{ a \mid (a, v_1, v_2) \in \underline{d}_{u, \xi} f(x, u(x), \nabla u(x)) \}$. As was noted in the proof of
Lemma~\ref{lem:AlmostCodiff}, the multifunction $\underline{d}_{u, \xi} f(\cdot, u(\cdot), \nabla u(\cdot))$ is
measurable. Therefore, by Filippov's theorem \cite[Thrm.~8.2.10]{AubinFrankowska} there exists a measurable selection
$(a_0(\cdot), v_{10}(\cdot), v_{20}(\cdot))$ of $\underline{d}_{u, \xi} f(\cdot, u(\cdot), \nabla u(\cdot))$ such that
$a_0(x) = 0$ for a.e. $x \in \Omega$. Consequently, one has $(0, x^*) \in \underline{d} \mathcal{I}(u)$ for $x^*$
defined as
$$
  \langle x^*, h \rangle = 
  \int_{\Omega} \big( \langle v_{10}(x), h(x) \rangle + \langle v_{20}(x), \nabla h(x) \rangle \big) \, dx 
  \quad \forall h \in \xWsp^{1, p}(\Omega; \xR^m),
$$
which yields $\max\{ A \colon (A, x^*) \in \underline{d} \mathcal{I}(u) \} = 0$.

Now we turn to the proof of the compactness of $\underline{d} \mathcal{I}(u)$. We consider two cases.

\textbf{Case $1 < p < + \infty$.} Denote by $\mathcal{F}$ the set of all measurable selections of the set-valued map 
$\underline{d}_{u, \xi} f(\cdot, u(\cdot), \nabla u(\cdot))$. By the codifferentiability conditions 
(see Def.~\ref{def:CodiffCond}) there exist $C > 0$ and a.e. nonnegative functions $\beta \in \xL^1(\Omega)$ and 
$\gamma \in \xL^{p'}(\Omega)$ such that for any $(a(\cdot), v_1(\cdot), v_2(\cdot)) \in \mathcal{F}$ and for a.e. 
$x \in \Omega$ one has
\begin{equation} \label{eq:CodiffGrowth_FiniteP}
  |a(x)| \le \beta(x) + C \big( |u(x)|^p + |\nabla u(x)|^p \big), \quad
  \max\{ |v_1(x)|, |v_2(x)| \} \le \gamma(x) + C \big( |u(x)|^{p - 1} + |\nabla u(x)|^{p - 1} \big).
\end{equation}
Observe that the right-hand side of the first inequality belongs to $\xL^1(\Omega)$, while the right-hand side of the
second inequality belongs to $\xL^{p'}(\Omega)$ by virtue of the facts that $u \in \xWsp^{1, p}(\Omega; \xR^m)$ and
$p'(p - 1) = p$. Thus, $\mathcal{F}$ is a bounded subset of the space 
$X = \xL^1(\Omega) \times \xL^{p'}(\Omega; \xR^m) \times \xL^{p'}(\Omega; \xR^{m \times d})$.

For any $(a(\cdot), v_1(\cdot), v_2(\cdot)) \in \mathcal{F}$ denote by 
$\mathcal{T}(a(\cdot), v_1(\cdot), v_2(\cdot))$ the pair 
$(A, x^*) \in \xR \times (\xWsp^{1, p}(\Omega; \xR^m))^*)$ such that $A = \int_{\Omega} a(x) \, dx$ and
$$
  \langle x^*, h \rangle 
  = \int_{\Omega} \big( \langle v_1(x), h(x) \rangle + \langle v_2(x), \nabla h(x) \rangle \big) \, dx
  \quad \forall h \in \xWsp^{1, p}(\Omega; \xR^m).
$$
Clearly, $\mathcal{T}(\mathcal{F}) = \underline{d} \mathcal{I}(u)$ (see~\eqref{eq:HypodiffIntegralFunc}).
Furthermore, one can easily verify that $\mathcal{T}$ is a continuous linear operator from the vector space 
$X = \xL^1(\Omega) \times \xL^{p'}(\Omega; \xR^m) \times \xL^{p'}(\Omega; \xR^{m \times d})$ endowed with
the weak topology to the space $(\xR \times Y^*, \tau_{\xR} \times \sigma(Y^*, Y))$ with 
$Y = \xWsp^{1, p}(\Omega; \xR^m)$. Therefore, it is sufficient to check that the set $\mathcal{F}$ is weakly
compact. Then one can conclude that the set $\underline{d} \mathcal{I}(u)$ is compact in the topology 
$\tau_{\xR} \times \sigma(Y^*, Y)$ as the image of the compact set $\mathcal{F}$ under the continuous map $\mathcal{T}$.

By the by the Eberlein-\v{S}mulian theorem it suffice to verify that $\mathcal{F}$ is weakly sequentially compact.
Choose any sequence $z_n(\cdot) = (a_n(\cdot), v_{1n}(\cdot), v_{2n}(\cdot)) \in \mathcal{F}$, $n \in \mathbb{N}$. From
the second inequality in \eqref{eq:CodiffGrowth_FiniteP} it follows that the sequence 
$\{ (v_{1n}(\cdot), v_{2n}(\cdot)) \}$ is bounded in 
$\xL^{p'}(\Omega; \xR^m) \times \xL^{p'}(\Omega; \xR^{m \times d})$. Hence taking into account the fact that
the space $\xL^{p'}(\Omega)$ is reflexive (recall that $1 < p < + \infty$, which yields $1 < p' < + \infty$) one obtains
that there exists a subsequence $\{ (v_{1 n_k}(\cdot), v_{2 n_k}(\cdot)) \}$ weakly converging to some 
$(v_1(\cdot), v_2(\cdot))$ in $\xL^{p'}(\Omega; \xR^m) \times \xL^{p'}(\Omega; \xR^{m \times d})$.

Let us now turn to the sequence $\{ a_n(\cdot) \}$. Denote 
$g(\cdot) = \beta(\cdot) + C (|u(\cdot)|^p + |\nabla u(\cdot)|^p)$ (see \eqref{eq:CodiffGrowth_FiniteP}). Clearly, 
$g \in \xL^1(\Omega)$ and $\int_{|a_{n_k}| > g} |a_{n_k}(x)| dx = 0 < \varepsilon$ for any $\varepsilon > 0$ and all
$k \in \mathbb{N}$. Therefore, by \cite[Thrm.~4.7.20]{Bogachev} the closure of the set 
$\{ a_{n_k}(\cdot) \}_{n \in \mathbb{N}}$ in the weak topology is weakly compact in $\xL^1(\Omega)$ or, equivalently,
weakly sequentially compact in $\xL^1(\Omega)$ by the Eberlein-\v{S}mulian theorem.  Consequently, one can extract a
subsequence of the sequence $\{ a_{n_k}(\cdot) \}$, which we denote again by $\{ a_{n_k}(\cdot) \}$, weakly converging
to some $a \in \xL^1(\Omega)$.

Observe that the subsequence $z_{n_k}(\cdot) = (a_{n_k}(\cdot), v_{1 n_k}(\cdot), v_{2 n_k}(\cdot))$, 
$k \in \mathbb{N}$, weakly converges to the function $z(\cdot) = (a(\cdot), v_1(\cdot), v_2(\cdot))$ in 
$X = \xL^1(\Omega) \times \xL^{p'}(\Omega; \xR^m) \times \xL^{p'}(\Omega; \xR^{m \times d})$. By
Mazur's lemma there exists a sequence $\widehat{z}_k(\cdot)$ of convex combinations of elements of the sequence 
$z_{n_k}(\cdot)$ strongly converging to $z(\cdot)$. As is well-known (see, e.g. \cite[Exercise~6.9]{Folland}), one can
extract a subsequence $\{ \widehat{z}_{k_l}(\cdot) \}$ that converges to $z(\cdot)$ almost everywhere. From the
convexity of the hypodifferential $\underline{d}_{u, \xi} f(\cdot)$ it follows that $\widehat{z}_k(\cdot)$ is a
measurable selection of the multifunction $\underline{d}_{u, \xi} f(\cdot, u(\cdot), \nabla u(\cdot))$ for any 
$k \in \mathbb{N}$. Hence bearing in mind the fact that 
the hypodifferential $\underline{d}_{u, \xi} f(x, u(x), \nabla u(x))$ is closed for a.e. $x \in \Omega$ one obtains that
$z(\cdot)$ is a measurable selection of $\underline{d}_{u, \xi} f(\cdot, u(\cdot), \nabla u(\cdot))$, i.e. $z(\cdot)
\in \mathcal{F}$. Thus, we found a subsequence of the original sequence $\{ z_n(\cdot) \} \subset \mathcal{F}$ weakly
converging to an element of $\mathcal{F}$. In other words, $\mathcal{F}$ is weakly sequentially compact. 

\textbf{Case $p = + \infty$.} Let, as above, $\mathcal{F}$ be the set of all measurable selections of the map
$\underline{d}_{u, \xi} f(\cdot, u(\cdot), \nabla u(\cdot))$. By the codifferentiability conditions
(Def.~\ref{def:CodiffCond}) there there exists an a.e. nonnegative function $\beta_N \in \xL^1(\Omega)$ such that 
for any $(a(\cdot), v_1(\cdot), v_2(\cdot)) \in \mathcal{F}$ one has
\begin{equation} \label{eq:CodiffGrowthInftyCl}
  \max\big\{ |a(x)|, |v_1(x)|, |v_2(x)| \big\} \le \beta_N(x)
\end{equation}
for a.e. $x \in \Omega$ (here $N = \| u \|_{1, \infty}$). Thus, $\mathcal{F}$ is a bounded subset of the space
$X = \xL^1(\Omega; \xR \times \xR^m \times \xR^{m \times d})$. 

Let the operator $\mathcal{T}$ be defined as in the case $1 < p < + \infty$. Then 
$\mathcal{T}(\mathcal{F}) = \underline{d} \mathcal{I}(u)$ and, as is easily seen, $\mathcal{T}$ is a continuous linear
operator from the space $X$ equipped with the weak topology to the product space 
$(\xR \times Y^*, \tau_{\xR} \times \sigma(Y^*, Y))$ with $Y = \xWsp^{1, \infty}(\Omega; \xR^m)$.
Therefore, it suffice to check that the set $\mathcal{F}$ is weakly compact in $X$. Then one can conclude that the set 
$\underline{d} \mathcal{I}(u)$ is compact as the continuous image of a compact set.

From \eqref{eq:CodiffGrowthInftyCl} it follows that for any $\varepsilon > 0$ and for all
$(a(\cdot), v_1(\cdot), v_2(\cdot)) \in \mathcal{F}$ one has
$$
  \int_{|a| > \beta_N} |a| \, d \mu = 0 < \varepsilon, \quad
  \int_{|v_1| > \beta_N} |v_1| \, d \mu = 0 < \varepsilon, \quad
  \int_{|v_2| > \beta_N} |v_2| \, d \mu = 0 < \varepsilon,
$$
where $\mu$ is the Lebesgue measure. Consequently, by \cite[Thrm.~4.7.20]{Bogachev} the closure of the set $\mathcal{F}$
in the weak topology is weakly compact in $X$, which by the Eberlein-\v{S}mulian theorem implies that it is weakly
sequentially compact. Let us verify that the set $\mathcal{F}$ itself is weakly sequentially compact. Then by applying
the Eberlein-\v{S}mulian theorem once again we arrive at the desired result.

Indeed, let $\{ z_n \} \subset \mathcal{F}$ be an arbitrary sequence. By the weak sequential compactness of the weak
closure of $\mathcal{F}$ there exists a subsequence $\{ z_{n_k} \}$ weakly converging to some $z \in X$. By Mazur's
lemma there exists a sequence of convex combinations $\{ \widehat{z}_k \}$ of elements of the sequence 
$\{ z_{n_k} \}$ strongly converging to $z$, which implies that there exists a subsequence $\{ \widehat{z}_{k_l} \}$
converging to $z$ almost everywhere. Observe that each triplet $\widehat{z}_{k_l}(\cdot)$ is a
measurable selection of $\underline{d}_{u, \xi} f(\cdot, u(\cdot), \nabla u(\cdot))$ due to the definition of
$\mathcal{F}$ and the fact that this multifunction is convex-valued. Therefore, bearing in mind the fact that the set
$\underline{d}_{u, \xi} f(x, u(x), \nabla u(x))$ is closed for a.e. $x \in \Omega$ one obtains that 
$z(\cdot)$ is a measurable selection of the multifunction $\underline{d}_{u, \xi} f(\cdot, u(\cdot), \nabla u(\cdot))$.
Thus, $z \in \mathcal{F}$, i.e. the subsequence $\{ z_{n_k} \}$ weakly converges to an element of the set $\mathcal{F}$,
which means that this set is weakly sequentially compact.
\end{proof}

Let us finally prove that the functional $\mathcal{I}$ is, in fact, \textit{continuously} codifferentiable. For any
subset $C$ of a metric space $(M, d)$ and a point $x \in M$ denote $\dist(x, C) = \inf_{y \in C} d(x, y)$.

Let us underline that in our earlier paper \cite{Dolgopolik_CalcVar} the continuous codifferentiability of the
functional $\mathcal{I}$ was proved only in the case when the set $\Omega$ is bounded and $p = + \infty$
(see \cite[Thrm.~5.7 and Remark~5.8]{Dolgopolik_CalcVar}). Here we extend this result to the case of unbounded domains
and arbitrary $p \in [1, + \infty]$ by utilising Vitali's theorem characterising convergence in $\xL^{p}$-spaces (see,
e.g. \cite[Thrm.~III.6.15]{DunfordSchwartz}), instead of relying on certain compactness arguments as it is done in 
\cite{Dolgopolik_CalcVar}.

\begin{lemma}
Let $f$ satisfy the codifferentiability conditions of order $p \in [1, + \infty]$ and suppose that either 
$1 < p \le + \infty$ or the set-valued maps $\underline{d}_{u, \xi} f(\cdot)$ and $\overline{d}_{u, \xi} f(\cdot)$ have
the form \eqref{eq:IntegrandCodiffL1} for some vectors 
$v_{1i}, w_{1j} \in \xR^m$, $v_{2i}, w_{2j} \in \xR^{m \times d}$, and Carath\'{e}odory
functions $f_i$ and $g_j$, $i \in I = \{ 1, \ldots, \ell \}$ and $j \in J = \{ 1, \ldots, r \}$. Then the set-valued
mappings $\underline{d} \mathcal{I}(\cdot)$ and $\overline{d} \mathcal{I}(\cdot)$ defined in
Theorem~\ref{thm:CodiffIntegralFunctional} are Hausdorff continuous.
\end{lemma}

\begin{proof}
We prove the statement of the lemma only for the hypodifferential mapping $\underline{d} \mathcal{I}(\cdot)$, since the
proof of the lemma for $\overline{d} \mathcal{I}(\cdot)$ is exactly the same.

Arguing by reductio ad absurdum, suppose that the multifunction $\underline{d} \mathcal{I}(\cdot)$ is not Hausdorff
continuous at a point $u \in \xWsp^{1, p}(\Omega; \xR^m)$. Then there exist $\theta > 0$ and a sequence
$\{ u_n \} \subset \xWsp^{1, p}(\Omega; \xR^m)$ converging to $u$ such that 
$d_H( \underline{d} \mathcal{I}(u_n), \underline{d} \mathcal{I}(u) ) > \theta$ for all $n \in \mathbb{N}$. Replacing, if
necessary, the sequence $\{ u_n \}$ with its subsequence, one can suppose that $u_n$ converges to $u$ almost everywhere
and $\nabla u_n$ converges to $\nabla u$ almost everywhere.

By the definition Hausdorff distance (see~\eqref{eq:HausdorffDist}), two cases are possible. Namely, there exists a
subsequence, which we denote again by $\{ u_n \}$, such that one of the following inequalities hold true:
\begin{gather} \label{eq:HausdDiscont_1st}
  \sup_{(B, y^*) \in \underline{d} \mathcal{I}(u_n)} \inf_{(A, x^*) \in \underline{d} \mathcal{I}(u)}
  \sqrt{|B - A|^2 + \| y^* - x^* \|^2} > \theta \\
  \sup_{(A, x^*) \in \underline{d} \mathcal{I}(u)} \inf_{(B, y^*) \in \underline{d} \mathcal{I}(u_n)}
  \sqrt{|B - A|^2 + \| y^* - x^* \|^2} > \theta.	\label{eq:HausdDiscont_2nd}
\end{gather}
We start with the first case.

\textbf{Case I.} From \eqref{eq:HausdDiscont_1st} it follows that for any $n \in \mathbb{N}$
there exists $(A_n, x_n^*) \in \underline{d} \mathcal{I}(u_n)$ satisfying the inequality 
$\dist((A_n, x_n^*), \underline{d} \mathcal{I}(u)) \ge \theta$. 
Denote by $z_n(\cdot) = (a_n(\cdot), v_{1n}(\cdot), v_{2n}(\cdot))$ a measurable selection of the multifunction 
$\underline{d}_{u, \xi} f(\cdot, u_n(\cdot), \nabla u_n(\cdot))$ corresponding to the pair $(A_n, x_n^*)$
(see~\eqref{eq:HypodiffIntegralFunc}). We consider the cases $p > 1$ and $p = 1$ separately

\textbf{Case I, $1 < p = + \infty$.} Recall that $\underline{d}_{u, \xi} f(\cdot, u(\cdot), \nabla u(\cdot))$ is a
convex and compact-valued multifunction. Furthermore, as was shown in the proof of Lemma~\ref{lem:AlmostCodiff}, the
codifferentiability conditions guarantee that this multifunction is measurable. Therefore, for any $n \in \mathbb{N}$
and for a.e. $x \in \Omega$ the set 
\begin{multline*}
  R_n(x) = \Big\{ (a, v_1, v_2) \in \underline{d}_{u, \xi} f(x, u(x), \nabla u(x)) \Bigm| \\
  \dist\big(z_n(x), \underline{d}_{u, \xi} f(x, u(x), \nabla u(x)) \big)^2 = 
  |a_n(x) - a|^2 + |v_{1n}(x) - v_1|^2 + |v_{2n}(x) - v_2|^2 \Big\}
\end{multline*}
(i.e. $R_n(x)$ is the set of points at which the infimum in the definition of the distance between $z_n(x)$ and
the set $\underline{d}_{u, \xi} f(x, u(x), \nabla u(x))$ is attained) is nonempty and the set-valued mapping
$R_n(\cdot)$ is measurable by \cite[Thrm.~8.2.11]{AubinFrankowska}.

Let $z^0_n(\cdot)$ be any measurable selection of the multifunction $R_n(\cdot)$, which exists by
\cite[Thrm.~8.1.3]{AubinFrankowska}. Define function
$\widehat{z}_n(\cdot) = (\widehat{a}_n(\cdot), \widehat{v}_{1n}(\cdot), \widehat{v}_{2n}(\cdot))$ as follows:
$$
  \widehat{z}_n(x) =
  \begin{cases}
    z^0_n(x), & \text{if } z_n(x) \notin \underline{d}_{u, \xi} f(x, u(x), \nabla u(x)), \\
    z_n(x), & \text{othwerwise}.
  \end{cases}
$$
Clearly, $\widehat{z}_n(\cdot)$ is a selection of the multifunction 
$\underline{d}_{u, \xi} f(\cdot, u(\cdot), \nabla u(\cdot))$. Furthermore, it is measurable due to the fact that the
set of all those $x \in \Omega$ for which 
$z_n(x) \notin \underline{d}_{u, \xi} f(x, u(x), \nabla u(x))$ is measurable by
\cite[Crlr.~8.2.13, part~2]{AubinFrankowska}.

By the codifferentiability conditions (see~Def.~\ref{def:CodiffCond}) the multifunction 
$\underline{d}_{u, \xi} f(\cdot)$ is a Carath\'{e}odory map. Thus, for a.e. $x \in \Omega$ the set-valued map 
$(u, \xi) \mapsto \underline{d}_{u, \xi} f(x, u, \xi)$ is continuous. Therefore, for a.e. $x \in \Omega$ one has 
$$
  \lim_{n \to \infty} d_H(\underline{d}_{u, \xi} f(x, u_n(x), \nabla u_n(x)), 
  \underline{d}_{u, \xi} f(x, u(x), \nabla u(x))) = 0.
$$
Hence, in particular, $\dist(z_n(x), \underline{d}_{u, \xi} f(x, u(x), \nabla u(x)) \to 0$ as $n \to \infty$,
which implies that the sequence $\{ z_n - \widehat{z}_n \}$ converges to zero almost everywhere. Let us prove that this
sequence converges to zero in 
$\xL^1(\Omega) \times \xL^{p'}(\Omega; \xR^m) \times \xL^{p'}(\Omega; \xR^{m \times d})$. To this end, we shall
utilise Vitali's theorem characterising convergence in $\xL^{p}$-spaces with $1 \le p < + \infty$ (see, e.g. 
\cite[Theorem~III.6.15]{DunfordSchwartz}). Note that $1 \le p' < + \infty$, since we consider 
the case $1 < p \le + \infty$.

Fix any $\varepsilon > 0$. By the growth condition on the codifferential mapping $D_{u, \xi} f$
(see~Def.~\ref{def:CodiffCond}) there exist $C > 0$ and an a.e. nonnegative function $\beta \in \xL^1(\Omega)$ such
that
\begin{equation} \label{CodiffGrowthVitali}
  |a_n(x) - \widehat{a}_n(x)| \le 2 \beta(x) + C \big( |u(x)|^p + |u_n(x)|^p + |\nabla u(x)|^p + |\nabla u_n(x)|^p \big)
\end{equation}
for a.e. $x \in \Omega$ in the case $1 < p < + \infty$, and there exists an a.e. nonnegative function
$\beta_N \in \xL^1(\Omega)$ such that $|a_n(x) - \widehat{a}_n(x)| \le 2 \beta_N(x)$ for a.e. $x \in \Omega$ in the
case
$p = + \infty$ (here $N = \max_{n \in \mathbb{N}}\{ \| u \|_{1, \infty}, \| u_n \|_{1, \infty} \}$). If $p = + \infty$,
then the sequence $\{ a_n - \widehat{a}_n \}$ converges to zero in $\xL^1(\Omega)$ by Lebesgue's dominated convergence
theorem. Therefore, let us consider the case $1 < p < + \infty$.

By the absolute continuity of the Lebesgue integral there exists $\delta_1 > 0$ such that for any measurable set 
$D \subseteq \Omega$ with $\mu(D) < \delta_1$ (here $\mu$ is the Lebesgue measure) one has
$$
  \int_D \beta d \mu < \frac{\varepsilon}{10}, \quad
  \int_D |u|^p d \mu < \frac{\varepsilon}{5C}, \quad
  \int_D |\nabla u|^p d \mu < \frac{\varepsilon}{5C}.
$$
Moreover, by the ``only if'' part of the Vitali convergence theorem, the convergence of $u_n$ to $u$ in 
$\xWsp^{1, p}(\Omega; \xR^m)$ implies that there exists $\delta_2 > 0$ such that for any measurable set 
$D \subseteq \Omega$ with $\mu(D) < \delta_2$ one has
$$
  \int_D |u_n|^p d \mu < \frac{\varepsilon}{5C}, \quad
  \int_D |\nabla u_n|^p d \mu < \frac{\varepsilon}{5C} \quad \forall n \in \mathbb{N}.
$$
Hence with the use of \eqref{CodiffGrowthVitali} one obtains that for any measurable set $D \subseteq \Omega$ with
$\mu(D) < \min\{ \delta_1, \delta_2 \}$ one has $\int_D |a_n - \widehat{a}_n| d \mu < \varepsilon$ for 
all $n \in \mathbb{N}$.

Denote $\Omega_N = \{ x \in \Omega \mid |x| \le N \}$. From the fact that $\beta \in \xL^1(\Omega)$ and 
$u \in \xWsp^{1, p}(\Omega; \xR^m)$ it follows that there exists $N \in \mathbb{N}$ such that
$$
  \int_{\Omega \setminus \Omega_N} \beta d \mu < \frac{\varepsilon}{10}, \quad
  \int_{\Omega \setminus \Omega_N} |u|^p d \mu < \frac{\varepsilon}{5C}, \quad
  \int_{\Omega \setminus \Omega_N} |\nabla u|^p d \mu < \frac{\varepsilon}{5C}
$$
(see, e.g. \cite[Prp.~2.6.2]{Bogachev}). Furthermore, by the ``only if'' part of the Vitali convergence theorem there
exists a measurable set $E_{\varepsilon} \subseteq \Omega$ such that $\mu(E_{\varepsilon}) < + \infty$ and
$$
  \int_{\Omega \setminus E_{\varepsilon}} |u_n|^p d \mu < \frac{\varepsilon}{5C}, \quad
  \int_{\Omega \setminus E_{\varepsilon}} |\nabla u_n|^p d \mu < \frac{\varepsilon}{5C} 
  \quad \forall n \in \mathbb{N}.
$$
Therefore, by applying \eqref{CodiffGrowthVitali} one obtains that 
$\int_{\Omega \setminus \Omega_{\varepsilon}} |a_n - \widehat{a}_n| d \mu < \varepsilon$ for all $n \in \mathbb{N}$,
where $\Omega_{\varepsilon} = \Omega_N \cup E_{\varepsilon}$. Hence with the use of the ``if'' part of the Vitali
convergence theorem one concludes that the sequence $\{ a_n - \widehat{a}_n \}$ converges to zero in $\xL^1(\Omega)$.

Let us now consider the sequence $\{ v_{1n} - \widehat{v}_{1n} \}$. By the growth condition on the codifferential
mapping $D_{u, \xi} f$ (see~Def.~\ref{def:CodiffCond}) there exist $C > 0$ and a.e. nonnegative function 
$\gamma \in \xL^{p'}(\Omega)$ such that
\begin{align*}
  |v_{1n}(x) - \widehat{v}_{1n}(x)|^{p'} &\le 2^{p'} \big( |v_{1n}(x)|^{p'} + |\widehat{v}_{2n}(x)|^{p'} \big) \\
  &\le 2^{p'} 3^{p'} \Big( 2 |\gamma(x)|^{p'} 
  + C^{p'} \big( |u(x)|^p + |\nabla u(x)|^p + |u_n(x)|^p + |\nabla u_n(x)|^p \big) \Big)
\end{align*}
for a.e. $x \in \Omega$ in the case $1 < p < + \infty$, and there exists an a.e. nonnegative function 
$\beta_N \in \xL^1(\Omega)$ such that $|v_{1n}(x) - \widehat{v}_{1n}(x)| \le 2 \beta_N(x)$ for a.e. $x \in \Omega$ in
the
case $p = + \infty$. Now, arguing in the same way as above and applying Vitali's convergence theorem in the case 
$1 < p < + \infty$ and Lebesgue's dominated convergence theorem in the case $p = + \infty$ one can readily verify that
$\{ v_{1n} - \widehat{v}_{1n} \}$ converges to zero in $\xL^{p'}(\Omega; \xR^m)$. The convergence of
$\{ v_{2n} - \widehat{v}_{2n} \}$ to zero in $\xL^{p'}(\Omega; \xR^{m \times d})$ is proved in exactly the same
way.

Denote by $(\widehat{A}_n, \widehat{x}^*_n)$ the element of $\underline{d} \mathcal{I}(u)$ corresponding to 
the selection $\widehat{z}_n(\cdot) = (\widehat{a}_n(\cdot), \widehat{v}_{1n}(\cdot), \widehat{v}_{2n}(\cdot))$ of
the multifunction $\underline{d}_{u, \xi} f(\cdot, u(\cdot), \nabla u(\cdot))$ (see~\eqref{eq:HypodiffIntegralFunc}).
Let us check that $|A_n - \widehat{A}_n| + \| x_n^* - \widehat{x}_n^* \| \to 0$ as $n \to \infty$. Indeed, for 
any $n \in \mathbb{N}$ one has
$$
  |A_n - \widehat{A}_n| \le \int_{\Omega} |a_n(x) - \widehat{a}_n(x)| \, dx = \| a_n - \widehat{a}_n \|_1,
$$
which implies that $|A_n - \widehat{A}_n| \to 0$ as $n \to \infty$. Similarly, with the use of H\"{o}lder's inequality
for any $h \in \xWsp^{1, p}(\Omega; \xR^m)$ one has
\begin{align*}
  \big| \langle x_n^* - \widehat{x}_n^*, h \rangle \big| 
  &\le \int_{\Omega} \big| \langle v_{1n}(x) - \widehat{v}_{1n}(x), h(x) \rangle \big| \, dx
  + \int_{\Omega} \big| \langle v_{2n}(x) - \widehat{v}_{2n}(x), \nabla h(x) \rangle \big| \, dx \\
  &\le \Big( \| v_{1n} - \widehat{v}_{1n} \|_{p'} + \| v_{2n} - \widehat{v}_{2n} \|_{p'} \Big) \| h \|_{1, p},
\end{align*}
which implies that 
$\| x_n^* - \widehat{x}_n^* \| \le \| v_{1n} - \widehat{v}_{1n} \|_{p'} + \| v_{2n} - \widehat{v}_{2n} \|_{p'}$ for
all $n \in \mathbb{N}$, and $\| x_n^* - \widehat{x}_n^* \| \to 0$ as $n \to \infty$. Consequently, bearing in mind the
fact that $(\widehat{A}_n, \widehat{x}_n^*) \in \underline{d} \mathcal{I}(u)$ for all $n \in \mathbb{N}$ one obtains
that $\dist((A_n, x_n^*), \underline{d} \mathcal{I}(u)) \to 0$ as $n \to \infty$, which contradicts the inequality
$\dist((A_n, x_n^*), \underline{d} \mathcal{I}(u)) \ge \theta$. Thus, the proof of the first case for 
$1 < p \le + \infty$ is complete.

\textbf{Case I, $p = 1$.} Let $S^{\ell}$ be the standard (probability) simplex in $\xR^{\ell}$, i.e. 
$$
  S^{\ell} = \Big\{ \alpha = (\alpha^{(1)}, \ldots, \alpha^{(\ell)}) \in \xR^{\ell} \Bigm| 
  \alpha^{(1)} + \ldots + \alpha^{(\ell)} = 1, \: \alpha^{(i)} \ge 0 
  \enspace \forall i \in \{ 1, \ldots, \ell \} \Big\}.
$$
For any $\alpha \in \xR^{\ell}$, $x \in \Omega$, $u \in \xR^m$, and $\xi \in \xR^{m \times d}$
define
\begin{equation} \label{eq:HypodiffViaConvexHull}
  g(x, u, \xi, \alpha) = \sum_{i = 1}^l \alpha^{(i)} (f_i(x, u, \xi), v_{1i}, v_{2i}).
\end{equation}
It is easily seen that $g$ is a Carath\'{e}odory map and 
$g(x, u, \xi, S^{\ell}) = \underline{d}_{u, \xi} f(x, u, \xi)$ for all $(x, u, \xi)$ by the definition of convex hull
(see~\eqref{eq:IntegrandCodiffL1}). 

Recall that $z_n(\cdot) = (a_n(\cdot), v_{1n}(\cdot), v_{2n}(\cdot))$ is a measurable selection of the set-valued
map $\underline{d}_{u, \xi} f(\cdot, u_n(\cdot), \nabla u_n(\cdot))$ such that for the corresponding pair
$(A_n, x_n^*) \in \underline{d} \mathcal{I}(u_n)$ one has 
$\dist((A_n, x_n^*), \underline{d} \mathcal{I}(u)) \ge \theta$ for all $n \in \mathbb{N}$. By definition for any
$n \in \mathbb{N}$ and a.e. $x \in \Omega$ one has 
$z_n(x) \in g(x, u_n(x), \nabla u_n(x), S^{\ell})$, which by Filippov's theorem \cite[Thrm.~8.2.10]{AubinFrankowska}
implies that for any $n \in \mathbb{N}$ there exists a measurable function $\alpha_n \colon \Omega \to S^{\ell}$ such
that $z_n(x) = g(x, u_n(x), \nabla u_n(x), \alpha_n(x))$ for a.e. $x \in \Omega$. Define 
$$
  \widehat{z}_n(\cdot) = (\widehat{a}_n(\cdot), \widehat{v}_{1n}(\cdot), \widehat{v}_{2n}(\cdot)) =
  g(\cdot, u(\cdot), \nabla u(\cdot), \alpha_n(\cdot)).
$$
Clearly, $\widehat{z}_n$ is a measurable selection of the multifunction 
$\underline{d}_{u, \xi} f(\cdot, u(\cdot), \nabla u(\cdot))$. Denote by $(\widehat{A}_n, \widehat{x}_n^*)$ the element
of $\underline{d} \mathcal{I}(u)$ corresponding to this selection (see~\eqref{eq:HypodiffIntegralFunc}).

From the definition of $g$ (see~\eqref{eq:HypodiffViaConvexHull}) and the definition of $\widehat{z}_n$ it follows that 
$x_n^* = \widehat{x}_n^*$ for all $n \in \mathbb{N}$. Furthermore, for all $n \in \mathbb{N}$ and 
a.e. $x \in \Omega$ one has
$$
  |a_n(x) - \widehat{a}_n(x)| 
  \le \sum_{i = 1}^{\ell} \alpha_n^{(i)}(x) \big| f_i(x, u_n(x), \nabla u_n(x)) - f_i(x, u(x), \nabla u(x)) \big|.
$$
Hence $|a_n(x) - \widehat{a}_n(x)| \to 0$ as $n \to \infty$ for a.e. $x \in \Omega$, since by our assumptions 
$u_n \to u$ and $\nabla u_n \to \nabla u$ almost everywhere, and $f_i$ are Carath\'{e}odory functions.

By the growth condition on $D_{u, \xi} f(\cdot)$ (see~Def.~\ref{def:CodiffCond}) there exist $C > 0$ and an a.e.
nonnegative function $\beta \in \xL^1(\Omega)$ such that
$$
  |a_n(x) - \widehat{a}_n(x)| \le \beta(x) 
  + C\big( |u(x)| + |\nabla u(x)| + |u_n(x)| + |\nabla u_n(x)| \big)
$$
for a.e. $x \in \Omega$. With the use of this inequality and Vitali's convergence theorem one can check that 
$|a_n - \widehat{a}_n|$ converges to zero in $\xL^1(\Omega)$ as in the case $1 < p < + \infty$.
Hence $|A_n - \widehat{A}_n| \le \int_{\Omega} |a_n - \widehat{a}_n| d \mu$ converges to zero as $n \to \infty$,
i.e. $|A_n - \widehat{A}_n| + \| x_n^* - \widehat{x}_n^* \| \to 0$ as $n \to \infty$. Therefore, 
$\dist((A_n, x_n^*), \underline{d} \mathcal{I}(u)) \to 0$ as $n \to \infty$, which once again contradicts the
inequality $\dist((A_n, x_n^*), \underline{d} \mathcal{I}(u)) \ge \theta$.

\textbf{Case II.} Suppose now that \eqref{eq:HausdDiscont_2nd} holds true. Then for any $n \in \mathbb{N}$ there exists
$(A_n, x_n^*) \in \underline{d} \mathcal{I}(u)$ such that 
$\dist((A_n, x_n^*), \underline{d} \mathcal{I}(u_n)) \ge \theta$. For any $n \in \mathbb{N}$ denote by 
$z_n(\cdot) = (a_n(\cdot), v_{1n}(\cdot), v_{2n}(\cdot))$ a measurable selection of the set-valued mapping 
$\underline{d}_{u, \xi} f(\cdot, u(\cdot), \nabla u(\cdot))$ corresponding to $(A_n, x_n^*)$
(see~\eqref{eq:HypodiffIntegralFunc}).

\textbf{Case II, $1 < p \le + \infty$.} Denote by $z^0_n(\cdot) = (a_n^0(\cdot), v_{1n}^0(\cdot), v_{2n}^0(\cdot))$ any
measurable selection of the multifunction $\underline{d}_{u, \xi} f(\cdot, u_n(\cdot), \nabla u_n(\cdot))$ such that
$$
  \dist\big(z_n(x), \underline{d}_{u, \xi} f(x, u_n(x), \nabla u_n(x)) \big)^2
  = |a_n(x) - a_n^0(x)|^2 + |v_{1n}(x) - v_{1n}^0(x)|^2 + |v_{2n}(x) - v_{2n}^0(x)|^2
$$
for a.e. $x \in \Omega$. The existence of such selection can be proved in the same way it is done in \textbf{Case I}.
Finally, define mapping 
$\widehat{z}_n(\cdot) = (\widehat{a}_n(\cdot), \widehat{v}_{1n}(\cdot), \widehat{v}_{2n}(\cdot))$ as follows:
$$
  \widehat{z}_n(x) =
  \begin{cases}
    z^0_n(x), & \text{if } z_n(x) \notin \underline{d}_{u, \xi} f(x, u_n(x), \nabla u_n(x)), \\
    z_n(x), & \text{othwerwise}.
  \end{cases}
$$
Then $\widehat{z}_n$ is a measurable selection of the set-valued mapping 
$\underline{d}_{u, \xi} f(\cdot, u_n(\cdot), \nabla u_n(\cdot))$. 

By the codifferentiability conditions (see~Def.~\ref{def:CodiffCond}) the multifunction 
$\underline{d}_{u, \xi} f(\cdot)$ is a Carath\'{e}odory map, i.e. for a.e. $x \in \Omega$ the set-valued map 
$(u, \xi) \mapsto \underline{d}_{u, \xi} f(x, u, \xi)$ is continuous. Therefore, for a.e. $x \in \Omega$ one has 
$$
  \lim_{n \to \infty} 
  d_H(\underline{d}_{u, \xi} f(x, u_n(x), \nabla u_n(x)), \underline{d}_{u, \xi} f(x, u(x), \nabla u(x))) = 0.
$$ 
Hence, in particular, $\dist(z_n(x), \underline{d}_{u, \xi} f(x, u_n(x), \nabla u_n(x)) \to 0$ as $n \to \infty$, which
implies that the sequence $\{ z_n - \widehat{z}_n \}$ converges to zero almost everywhere.
Applying the growth condition on the codifferential mapping $D_{u, \xi} f$ and arguing in the same way as in
\textbf{Case I} one can check that this sequence converges to zero in
$\xL^1(\Omega) \times \xL^{p'}(\Omega; \xR^m) \times \xL^{p'}(\Omega; \xR^{m \times d})$. With the use of this
fact it is easy to show that $|A_n - \widehat{A}_n| + \| x_n^* - \widehat{x}_n^* \| \to 0$ as $n \to \infty$, where 
$(\widehat{A}_n, \widehat{x}_n^*)$ is the element of $\underline{d} \mathcal{I}(u_n)$ corresponding to the selection
$\widehat{z}_n$. Therefore, $\dist((A_n, x_n), \underline{d} \mathcal{I}(u_n)) \to 0$ as $n \to \infty$, which
contradicts the inequality $\dist((A_n, x_n^*), \underline{d} \mathcal{I}(u_n)) \ge \theta$.

\textbf{Case II, $p = 1$.} Arguing in the same way as in \textbf{Case I} and applying Filippov's theorem, for any 
$n \in \mathbb{N}$ one can find a measurable function $\alpha_n \colon \Omega \to S^{\ell}$ such that 
$z_n(x) = g(x, u(x), \nabla u(x), \alpha_n(x))$ for a.e. $x \in \Omega$. Define 
$$
  \widehat{z}_n(x) =(\widehat{a}_n(x), \widehat{v}_{1n}(x), \widehat{v}_{2n}(x)) 
  = g(x, u_n(x), \nabla u_n(x), \alpha_n(x))
$$
for a.e. $x \in \Omega$. Then $\widehat{z}_n(\cdot)$ is a measurable selection of the
set-valued mapping $\underline{d}_{u, \xi} f(\cdot, u_n(\cdot), \nabla u_n(\cdot))$. Denote by 
$(\widehat{A}_n, \widehat{x}_n^*)$ the element of $\underline{d} \mathcal{I} (u_n)$ corresponding to this selection.
Then $x_n^* = \widehat{x}_n^*$ for all $n \in \mathbb{N}$, and arguing in the same way as in \textbf{Case I} one can
check that $|A_n - \widehat{A}_n| \to 0$ as $n \to \infty$. Therefore 
$\dist((A_n, x_n), \underline{d} \mathcal{I}(u_n)) \to 0$ as $n \to \infty$, which once again contradicts the
inequality $\dist((A_n, x_n^*), \underline{d} \mathcal{I}(u_n)) \ge \theta$.
\end{proof}

Applying Theorem~\ref{thm:CodiffIntegralFunctional}, \cite[Crlr.~2]{Dolgopolik_CodiffDescent}, and
Lemma~\ref{lem:ContCodiffQuasidiffUniform} one obtains that in the case when the integrand $f$ satisfies the
codifferentiability conditions, the functional $\mathcal{I}$ is locally Lipschitz continuous and Hadamard
quasidifferentiable.

\begin{corollary} \label{crlr:QuasidiffIntegralFunctional}
Let $f$ satisfy the codifferentiability conditions of order $p \in [1, + \infty]$, and let either $1 < p \le + \infty$
or the set $\Omega$ be bounded and have the segment property, and the set-valued maps $\underline{d}_{u, \xi} f(\cdot)$
and $\overline{d}_{u, \xi} f(\cdot)$ have the form \eqref{eq:IntegrandCodiffL1}. Then the functional $\mathcal{I}$ is
locally Lipschitz continuous, Hadamard quasidifferentiable at every $u \in \xWsp^{1, p}(\Omega; \xR^m)$, and the
pair $\mathscr{D} \mathcal{I}(u) = [\underline{\partial} \mathcal{I}(u), \overline{\partial} \mathcal{I}(u)]$ with 
$$
  \underline{\partial} \mathcal{I}(u) 
  = \{ x^* \in (\xWsp^{1, p}(\Omega; \xR^m))^* \mid (0, x^*) \in \underline{d} \mathcal{I}(u) \}, \quad
  \overline{\partial} \mathcal{I}(u) 
  = \{ y^* \in (\xWsp^{1, p}(\Omega; \xR^m))^* \mid (0, y^*) \in \overline{d} \mathcal{I}(u) \},
$$
is a quasidifferential of $\mathcal{I}$ at $u$, where the sets $\underline{d} \mathcal{I}(u)$ and $\overline{d}
\mathcal{I}(u)$ are defined in Theorem~\ref{thm:CodiffIntegralFunctional}.
\end{corollary}

\begin{remark} \label{rmrk:QuasidiffIntegralFunctional}
Recall that by the definition of codifferential one has $a \le 0$ for any 
$(a, v_1, v_2) \in \underline{d}_{u, \xi} f(x, u, \xi)$. Hence with the use of
Theorem~\ref{thm:CodiffIntegralFunctional} and the corollary above one obtains that 
$x^* \in \underline{\partial} \mathcal{I}(u)$ if and only if there exists a measurable selection
$(0, v_1(\cdot), v_2(\cdot))$ of the multifunction $\underline{d}_{u, \xi} f(\cdot, u(\cdot), \nabla u(\cdot))$
such that
$$
  \langle x^*, h \rangle = 
  \int_{\Omega} \big( \langle v_1(x), h(x) \rangle + \langle v_2(x), \nabla h(x) \rangle \big) \, dx 
  \quad \forall h \in \xWsp^{1, p}(\Omega; \xR^m).
$$
A similar statement holds true for $\overline{\partial} \mathcal{I}(u)$ as well.
\end{remark}

As usual, denote by $\xWsp^{1, p}_0(\Omega; \xR^m)$ the closure of the space $\xC^{\infty}_c(\Omega; \xR^m)$ of infinitely
differentiable functions $\varphi \colon \Omega \to \xR^m$ with compact support in the Sobolev space 
$\xWsp^{1, p}(\Omega; \xR^m)$. To derive optimality conditions for problems with prescribed boundary conditions we
will utilise the following corollary on the quasidifferentiability of the restriction of $\mathcal{I}$ to the space
$\xWsp^{1, p}_0(\Omega; \xR^m)$. This result is almost trivial. Nevertheless, we briefly outline its proof for the
sake of completeness and mathematical rigour.

\begin{corollary} \label{crlr:Quasidiff_ZeroTraceSobolev}
Let $f$ satisfy the codifferentiability conditions of order $p \in [1, + \infty]$, 
$u_0 \in \xWsp^{1, p}(\Omega; \xR^m)$ be fixed, and let either $1 < p \le + \infty$ or the set $\Omega$ be bounded
and have the segment property, and the set-valued maps $\underline{d}_{u, \xi} f(\cdot)$
and $\overline{d}_{u, \xi} f(\cdot)$ have the form \eqref{eq:IntegrandCodiffL1}. Then the functional
$\mathcal{J} \colon \xWsp^{1, p}_0(\Omega; \xR^m) \to \xR$, $\mathcal{J}(u) = \mathcal{I}(u_0 + u)$ is
correctly defined, locally Lipschitz continuous, and Hadamard quasidifferentiable at every 
$u \in \xWsp^{1, p}_0(\Omega; \xR^m)$. Furthermore, the pair 
$\mathscr{D} \mathcal{J}(u) = [\underline{\partial} \mathcal{J}(u), \overline{\partial} \mathcal{J}(u)]$ with
\begin{multline} \label{eq:Quasidiff_ZeroTraceSobolev1}
  \underline{\partial} \mathcal{J}(u) = \Big\{ x^* \in (\xWsp^{1, p}_0(\Omega; \xR^m))^* \Bigm|
  \langle x^*, h \rangle = 
  \int_{\Omega} \big( \langle v_1(x), h(x) \rangle + \langle v_2(x), \nabla h(x) \rangle \big) \, dx \enspace
  \forall h \in \xWsp^{1, p}_0(\Omega; \xR^m), \\
  (0, v_1(\cdot), v_2(\cdot)) \text{ is a measurable selection of the set-valued map }
  \underline{d}_{u, \xi} f(\cdot, u_0(\cdot) + u(\cdot), \nabla u_0(\cdot) + \nabla u(\cdot)) \Big\}
\end{multline}
and
\begin{multline} \label{eq:Quasidiff_ZeroTraceSobolev2}
  \overline{\partial} \mathcal{J}(u) = \Big\{ y^* \in (\xWsp^{1, p}_0(\Omega; \xR^m))^* \Bigm|
  \langle y^*, h \rangle = 
  \int_{\Omega} \big( \langle w_1(x), h(x) \rangle + \langle w_2(x), \nabla h(x) \rangle \big) \, dx \enspace
  \forall h \in \xWsp^{1, p}_0(\Omega; \xR^m), \\
  (0, w_1(\cdot), w_2(\cdot)) \text{ is a measurable selection of the set-valued map }
  \overline{d}_{u, \xi} f(\cdot, u_0(\cdot) + u(\cdot), \nabla u_0(\cdot) + \nabla u(\cdot)) \Big\}
\end{multline}
is a quasidifferential of $\mathcal{J}$ at $u$.
\end{corollary}

\begin{proof}
The fact that the functional $\mathcal{J}$ is correctly defined and locally Lipschitz continuous follows directly from
its definition and Corollary~\ref{crlr:QuasidiffIntegralFunctional}. Let us prove that it is Hadamard
quasidifferentiable.

Denote $X = \xWsp^{1, p}(\Omega; \xR^m)$ and $X_0 = \xWsp^{1, p}_0(\Omega; \xR^m)$. Introduce the linear operator
$\mathcal{T} \colon X^* \to X_0^*$ that maps $x^* \in X^*$ to its restriction to $X_0$, i.e. 
$\mathcal{T}(x^*) = x^*|_{X_0}$. It is easily seen that $\mathcal{T}$ is a continuous operator from $X^*$ endowed with
the weak${}^*$ topology to $X_0^*$ endowed with the weak${}^*$ topology, since $X_0 \subset X$.

Observe that by definitions (see Corollary~\ref{crlr:QuasidiffIntegralFunctional} and
Remark~\ref{rmrk:QuasidiffIntegralFunctional}) one has 
$\underline{\partial} \mathcal{J}(u) = \mathcal{T}(\underline{\partial} \mathcal{I}(u_0 + u))$ and
$\overline{\partial} \mathcal{J}(u) = \mathcal{T}(\overline{\partial} \mathcal{I}(u_0 + u))$. Therefore, 
$\underline{\partial} \mathcal{J}(u)$ and $\overline{\partial} \mathcal{J}(u)$ are convex and weak${}^*$ compact convex
subsets of $X_0^*$ due to the fact that $\underline{\partial} \mathcal{I}(u_0 + u)$ and 
$\overline{\partial} \mathcal{I}(u_0 + u)$ are convex and weak${}^*$ compact convex subsets of $X^*$.

Fix any $u \in \xWsp^{1, p}_0(\Omega; \xR^m)$. With the use of Corollary~\ref{crlr:QuasidiffIntegralFunctional} and
the fact that for any $\alpha > 0$ and $h \in \xWsp^{1, p}_0(\Omega; \xR^m)$ one has 
$\mathcal{J}(u + \alpha h) = \mathcal{I}(u_0 + u + \alpha h)$ one obtains that
\begin{multline*}
  \lim_{[\alpha, h'] \to [+0, h]} \frac{\mathcal{J}(u + \alpha h') - \mathcal{J}(u)}{\alpha}
  = \lim_{[\alpha, h'] \to [+0, h]} \frac{\mathcal{I}(u_0 + u + \alpha h') - \mathcal{J}(u_0 + u)}{\alpha} \\
  = \mathcal{I}'(u_0 + u; h) 
  = \max_{x^* \in \underline{\partial} \mathcal{I}(u_0 + u)} \langle x^*, h \rangle 
  + \min_{y^* \in \overline{\partial} \mathcal{I}(u_0 + u)} \langle y^*, h \rangle 
  = \max_{x^* \in \underline{\partial} \mathcal{J}(u)} \langle x^*, h \rangle 
  + \min_{y^* \in \overline{\partial} \mathcal{J}(u)} \langle y^*, h \rangle 
\end{multline*}
for all $h \in \xWsp^{1, p}_0(\Omega; \xR^m)$ (here $h' \in \xWsp^{1, p}_0(\Omega; \xR^m)$ as well). Thus, the
functional $\mathcal{J}$ is Hadamard quasidifferentiable, and the pair \eqref{eq:Quasidiff_ZeroTraceSobolev1},
\eqref{eq:Quasidiff_ZeroTraceSobolev2} is its quasidifferential.
\end{proof}

\begin{remark} \label{rmrk:Quasidiff_ZeroTraceSobolev}
By Theorem~\ref{thm:CodiffIntegralFunctional}, the assumption that in the case $p = 1$ the set-valued mappings 
$\underline{d}_{u, \xi} f(\cdot)$ and $\overline{d}_{u, \xi} f(\cdot)$ have the form \eqref{eq:IntegrandCodiffL1} is
needed only to ensure the continuity of the multifunctions $\underline{d} \mathcal{I}(\cdot)$ and 
$\overline{d} \mathcal{I}(\cdot)$, i.e. to ensure that the functional $\mathcal{I}$ is \textit{continuously}
codifferentiable. Therefore, in the case when the function $f$ satisfies the codifferentiability conditions of
order $p = 1$ and $\Omega$ is bounded and has the segment property, but the set-valued mappings 
$\underline{d}_{u, \xi} f(\cdot)$ and $\overline{d}_{u, \xi} f(\cdot)$ do \textit{not} have the form
\eqref{eq:IntegrandCodiffL1}, the functional $\mathcal{J}$ from Corollary~\ref{crlr:Quasidiff_ZeroTraceSobolev} is still
quasidifferentiable and the pair \eqref{eq:Quasidiff_ZeroTraceSobolev1}, \eqref{eq:Quasidiff_ZeroTraceSobolev2} is a
quasidifferential of $\mathcal{J}$ at $u$ by Remark~\ref{rmrk:Codifferential_Quasidifferential} and
Theorem~\ref{thm:CodiffIntegralFunctional}.
\end{remark}

\section{Constrained Nonsmooth Problems of the Calculus of Variations}
\label{sec:CalcVar}

In this section we derive optimality conditions in terms of codifferentials for nonsmooth problems of the calculus of
variations with nonsmooth isoperimetric constraints and nonsmooth constraints at the boundary of the domain. By means of
several simple examples we also demonstrate that in some cases optimality conditions in terms of codifferentials are
better than optimality conditions in terms of various subdifferentials.

\subsection{Unconstrained Problems}

We start with an unconstrained problem of the form
\begin{equation} \label{eq:UnconstrainedCalcVarProbl}
  \min \: \mathcal{I}(u) = \int_{\Omega} f(x, u(x), \nabla u(x)) \, dx, \quad
  u \in u_0 + \xWsp^{1, p}_0(\Omega; \xR^m).
\end{equation}
Here, as in the previous section, $\Omega \subseteq \xR^d$ is an open set, 
$f \colon \Omega \times \xR^m \times \xR^{m \times d} \to \xR$, $f = f(x, u, \xi)$, is a nonsmooth
function, while $u_0 \in \xWsp^{1, p}(\Omega; \xR^m)$ is a fixed function. 

In essence, problem \eqref{eq:UnconstrainedCalcVarProbl} can be viewed as the classical problem of minimising
$\mathcal{I}(u)$ over the set of all those $u \in \xWsp^{1, p}(\Omega; \xR^m)$ for which 
$u|_{\partial \Omega} = \psi$ for some prespecified function $\psi$, where $\partial \Omega$ is the boundary of
$\Omega$ (simply put $\psi = u_0|_{\partial \Omega}$). However, to avoid the usage of trace operators and corresponding
assumptions on the domain $\Omega$, we pose this classical ``boundary value problem'' in the abstract form
\eqref{eq:UnconstrainedCalcVarProbl}.

In the case when the domain $\Omega$ is bounded and has the segment property, optimality conditions for this problem in
terms of codifferentials were first obtained by the author in \cite{Dolgopolik_CalcVar}. Here we rederive this
conditions in the general case to help the reader more readily understand the derivation of optimality conditions for
constrained problems, as well as due to the fact the optimality conditions for problem
\eqref{eq:UnconstrainedCalcVarProbl} are closely related to a natural constraint qualification for isoperimetric
constraints.

Recall that a function $v \in \xL^1_{loc}(\Omega)$ is called a \text{weak divergence} of a vector field
$u \in \xL^1(\Omega; \xR^d)$, if
$$
  \int_{\Omega} v \varphi dx = - \int_{\Omega} \langle u, \nabla \varphi \rangle dx
  \quad \forall \varphi \in \xC^{\infty}_c(\Omega).
$$
In this case we write $v = \diverg u$. Denote by $\xL^{p}(\Omega; \xR^{m \times d}; \diverg)$ the space of all those
functions $u \in \xL^{p}(\Omega; \xR^{m \times d})$ for which 
there exists the weak divergence 
$\diverg u = (\diverg(u_{11}, \ldots, u_{1d}), \ldots, \diverg(u_{m1}, \ldots, u_{md}))$ and
$\diverg u \in \xL^{p}(\Omega; \xR^m)$. Note that in the one-dimensional case (i.e. when $d = 1$) the weak divergence
$\diverg u$ coincides with the weak derivative $u'$, which implies that the space 
$\xL^{p}(\Omega; \xR^{m \times 1}; \diverg)$ coincides with the Sobolev space $\xWsp^{1, p}(\Omega; \xR^m)$.

\begin{theorem} \label{thrm:UnconstrainedCalcVarProbl}
Let $f$ satisfy the codifferentiability conditions of order $p \in [1, + \infty]$, and let either $1 < p \le + \infty$
or the set $\Omega$ be bounded and have the segment property. Let also $u_*$ be a locally optimal solution of problem
\eqref{eq:UnconstrainedCalcVarProbl}. Then for any measurable selection $(0, w_1(\cdot), w_2(\cdot))$ of the
set-valued map $\overline{d}_{u, \xi} f(\cdot, u_*(\cdot), \nabla u_*(\cdot))$ there exists 
$\zeta \in \xL^{p'}(\Omega; \xR^{m \times d}; \diverg)$ satisfying the Euler-Lagrange inclusion
\begin{equation} \label{eq:UnconstrainedOptCond_CalcVar}
  (0, \diverg(\zeta)(x), \zeta(x)) \in \underline{d}_{u, \xi} f(x, u_*(x), \nabla u_*(x)) + (0, w_1(x), w_2(x))
  \quad \text{for a.e. } x \in \Omega.
\end{equation}
\end{theorem}

\begin{proof}
Define $\mathcal{J}(h) = \mathcal{I}(u_* + h)$ for any $h \in \xWsp^{1, p}_0(\Omega; \xR^m)$. By
Corollary~\ref{crlr:Quasidiff_ZeroTraceSobolev} and Remark~\ref{rmrk:Quasidiff_ZeroTraceSobolev} the functional
$\mathcal{J}$ is quasidifferentiable at $u = 0$, i.e. its directional derivative at this point has the form
\begin{equation} \label{eq:UncontrMinVar_DirectDeriv}
  \mathcal{J}'(0, h) = \max_{x^* \in \underline{\partial} \mathcal{J}(0)} \langle x^*, h \rangle
  + \min_{y^* \in \overline{\partial} \mathcal{J}(0)} \langle y^*, h \rangle,
\end{equation}
where the pair $[\underline{\partial} \mathcal{J}(0), \overline{\partial} \mathcal{J}(0)]$ is from
Corollary~\ref{crlr:Quasidiff_ZeroTraceSobolev}.

Fix any measurable selection $(0, w_1(\cdot), w_2(\cdot))$ of the set-valued mapping 
$x \mapsto \overline{d}_{u, \xi} f(x, u_*(x), \nabla u_*(x))$ and define a linear functional $y_0^*$ as follows:
$$
  \langle y_0^*, h \rangle = 
  \int_{\Omega} \big( \langle w_1(x), h(x) \rangle + \langle w_2(x), \nabla h(x) \rangle \big) \, dx
  \quad \forall h \in \xWsp^{1, p}_0(\Omega; \xR^m).
$$
Observe that $y_0^* \in \overline{\partial} \mathcal{J}(0)$ by Corollary~\ref{crlr:Quasidiff_ZeroTraceSobolev}.

Recall that $u_*$ is a locally optimal solution of problem \eqref{eq:UnconstrainedCalcVarProbl}. Therefore, $h = 0$ is a
point of local minimum of the functional $\mathcal{J}$, which obviously implies that $\mathcal{J}'(0, h) \ge 0$ for
all $h \in \xWsp^{1, p}_0(\Omega; \xR^m)$. Hence by applying \eqref{eq:UncontrMinVar_DirectDeriv} one obtains that
$$
  \max_{x^* \in \underline{\partial} \mathcal{J}(0)} \langle x^*, h \rangle + \langle y_0^*, h \rangle \ge 0 
  \quad \forall h \in \xWsp^{1, p}_0(\Omega; \xR^m).
$$
Consequently, $0 \in \underline{\partial} \mathcal{J}(0) + y_0^*$, since otherwise utilising the separation theorem in
the space $(\xWsp^{1, p}_0(\Omega; \xR^m))^*$ equipped with the weak${}^*$ topology one can find 
$h \in \xWsp^{1, p}_0(\Omega; \xR^m)$ such that
$\max\{ \langle x^*, h \rangle \mid x^* \in \underline{\partial} \mathcal{J}(0) + y_0^* \} < 0$,
which is impossible. Thus, there exists $x_0^* \in \underline{\partial} \mathcal{J}(0)$ such that
$x_0^* + y_0^* = 0$. Hence by Corollary~\ref{crlr:Quasidiff_ZeroTraceSobolev} there exists a measurable selection
$(0, v_1(\cdot), v_2(\cdot))$ of the multifunction 
$\underline{d}_{u, \xi} f(\cdot, u_*(\cdot), \nabla u_*(\cdot))$ such that
$$
  \int_{\Omega} \big( \langle v_1(x) + w_1(x), h(x) \rangle 
  + \langle v_2(x) + w_2(x), \nabla h(x) \rangle  \big) \, dx = 0
  \quad \forall h \in \xWsp^{1, p}_0(\Omega; \xR^m).
$$
Define $\zeta = v_2 + w_2$. Then the equality above implies that there exists the weak divergence of $\zeta$ and
$\diverg \zeta = v_1 + w_1$. From the growth condition on the codifferential mapping $D_{u, \xi} f(\cdot)$ 
(see~the definition of codifferentiability conditions, Def.~\ref{def:CodiffCond}) it obviously follows that
$v_2 + w_2 \in \xL^{p'}(\Omega; \xR^{m \times d})$ and $v_1 + w_1 \in \xL^{p'}(\Omega; \xR^m)$. Thus,
$\zeta \in \xL^{p'}(\Omega; \xR^{m \times d}; \diverg)$, and \eqref{eq:UnconstrainedOptCond_CalcVar} holds true by
the definition of $\zeta$.
\end{proof}

Let us give an example illustrating optimality conditions from the theorem above.

\begin{example} \label{example:UnconstrainedProblem}
Let $d = 2$, $m = 1$, $p = 1$, and $\Omega = (-1, 1) \times (-1, 1)$. Consider the following problem:
\begin{equation} \label{eq:UnconstrainedCalcVar_Example}
  \min_{u \in \xWsp^{1, 1}(\Omega)} \: \mathcal{I}(u) = \int_{\Omega}  
  \big( |u'_{x^{(1)}}(x)| - |u'_{x^{(2)}}(x)| \big) \, dx, 
  \quad u|_{\partial \Omega} = 0.
\end{equation}
In this case $f(x, u, \xi) = |\xi^{(1)}| - |\xi^{(2)}|$, and we define $u_0 = 0$ (see
problem~\eqref{eq:UnconstrainedCalcVarProbl}). We want to know whether the function $u_* = 0$ is an optimal solution of
problem \eqref{eq:UnconstrainedCalcVar_Example}.

Let us apply optimality conditions in term of the Clarke subdifferential first \cite[Thrm.~4.6.1]{Clarke} 
(see also \cite[Sect.~20]{Clarke2013}). Denote $L(u, \xi) = |\xi^{(1)}| - |\xi^{(2)}|$. As is easily seen, the Clarke
subdifferential of this function at the origin has the form:
$$
  \partial_{Cl} L(0, 0) = \co\left\{ \left( \begin{smallmatrix} 0 \\ 1 \\ 1 \end{smallmatrix} \right), 
  \left( \begin{smallmatrix} 0 \\ 1 \\ -1 \end{smallmatrix} \right), 
  \left( \begin{smallmatrix} 0 \\ -1 \\ 1 \end{smallmatrix} \right), 
  \left( \begin{smallmatrix} 0 \\ -1 \\ -1 \end{smallmatrix} \right) \right\}.
$$
Therefore, for the function $\zeta = 0$ one has 
$(\diverg \zeta(x), \zeta(x)) \in \partial_{Cl} L(u_*(x), \nabla u_*(x))$ for all $x \in \Omega$, i.e. optimality
conditions in terms of the Clarke subdifferential \cite[Thrm.~4.6.1]{Clarke} are satisfied at $u_* = 0$. One can
verify that optimality conditions in terms of $K$-subdifferential from \cite{OrlovTsygankova} are satisfied for 
$u_* = 0$ as well.

Let us now check optimality conditions from Theorem~\ref{thrm:UnconstrainedCalcVarProbl}. With the use of
the codifferential calculus \cite{DemRub_book} one gets
$$
  \underline{d}_{u, \xi} f(x, u, \xi) = \co\big\{ (\pm\xi^{(1)} - |\xi^{(1)}|, 0, \pm 1, 0) \big\}, \quad
  \overline{d}_{u, \xi} f(x, u, \xi) = \co\big\{ (\pm\xi^{(2)} + |\xi^{(2)}|, 0, 0, \pm 1) \big\}
$$
(here the first coordinate is $a$, the second is $v_1$, while the third and fourth ones are $v_2$ in the notation of
the previous section). Therefore, as is readily seen, the integrand $f$ satisfies the codifferentiability conditions of
order $p = 1$.

For any $n \in \mathbb{N}$ define
\begin{equation} \label{eq:AlternStepFunc}
  \rho_n(t) = \begin{cases}
    1, & \text{if } t \in \left[-1 + \frac{2k - 2}{2n}, -1 + \frac{2k - 1}{2n} \right), k \in \{ 1, \ldots, 2n \} \\
    -1, & \text{if } t \in \left[-1 + \frac{2k - 1}{2n}, -1 + \frac{2k}{2n} \right), k \in \{ 1, \ldots, 2n \}.
  \end{cases}
\end{equation}
Clearly, the mapping $(0, w_1(\cdot), w_2(\cdot))$ with $w_1(x) = 0$ and $w_2 = (0, \rho_n(x^{(2)}))$ for all 
$x = (x^{(1)}, x^{(2)}) \in \Omega$ is a measurable selection of the multifunction 
$\overline{d}_{u, \xi} f(\cdot, u_*(\cdot), \nabla u_*(\cdot))$ for all $n \in \mathbb{N}$. To verify whether the
optimality conditions from Theorem~\ref{thrm:UnconstrainedCalcVarProbl} hold true, suppose that 
there exists $\zeta \in \xL^{\infty}(\Omega; \xR^{1 \times 2}; \diverg)$ such that
$$
  (0, \diverg(\zeta)(x), \zeta(x)) \in \underline{d}_{u, \xi} f(x, u_*(x), \nabla u_*(x)) + (0, w_1(x), w_2(x))
  = \co\big\{ (0, 0, \pm 1, \rho_n(x^{(2)})) \big\}
$$
for a.e. $x \in \Omega$. Hence $\diverg(\zeta)(x) = 0$, $|\zeta_1(x)| \le 1$, and $\zeta_2(x) = \rho_n(x^{(2)})$ for
a.e. $x \in \Omega$. Consequently, by the definition of weak divergence one has
\begin{equation} \label{eq:WeakDiverg_UnconstExample}
  \int_{\Omega} \big( \zeta_1(x) \varphi'_{x^{(1)}}(x) + \rho_n(x^{(2)}) \varphi'_{x^{(2)}}(x) \big) dx = 0
  \quad \forall \varphi \in \xC^{\infty}_c(\Omega).
\end{equation}
Since both $\rho_n$ and $\zeta_1$ belong to $\xL^{\infty}(\Omega)$, the equality above holds true for all 
$\varphi \in \xWsp^{1, 1}_0(\Omega)$. Define $\psi_n(t) = 2n \int_{-1}^t \rho_n(\tau) \, d \tau$ for all 
$t \in (-1, 1)$, and for any $n \in \mathbb{N}$ put $\varphi_n(x) = (- (x^{(1)})^2 + 1)\psi_n(x^{(2)})$. Observe that 
$\varphi_n \in \xWsp^{1, 1}_0(\Omega)$ due to the fact that $\varphi_n(x) = 0$ for all $x \in \partial \Omega$. Hence
with the use of \eqref{eq:WeakDiverg_UnconstExample} one gets that
\begin{align*}
  0 = \int_{\Omega} \langle \zeta(x), \nabla \varphi_n(x) \rangle \, dx
  &= \int_{\Omega} \big( - 2 \zeta_1(x) x^{(1)} \psi_n(x^{(2)}) + 2n (- (x^{(1)})^2 + 1) \big) \, dx \\
  &= - \int_{\Omega} 2 \zeta_1(x) x^{(1)} \psi_n(x^{(2)}) \, dx + \frac{16 n}{3} \ge - 4 + \frac{16 n}{3} > 0
  \quad \forall n \ge 1,
\end{align*}
which is impossible (the penultimate inequality follows from the fact that $|\zeta_1(x)| \le 1$ and 
$\psi_n(x^{(2)}) \in [0, 1]$ for any $x \in \Omega$; see \eqref{eq:AlternStepFunc}). Thus, the optimality conditions
from Theorem~\ref{thrm:UnconstrainedCalcVarProbl} are not satisfied at $u_* = 0$, unlike optimality conditions in terms
of the Clarke subdifferential. For the sake of completeness, let us finally note that, in actuality, the functional 
$\mathcal{I}$ is unbounded below on $\xWsp^{1, 1}_0(\Omega)$, which can be easily verified directly or by noting
that $u_* = 0$ is not an optimal solution of problem \eqref{eq:UnconstrainedCalcVar_Example} and $\mathcal{I}$ is
positively homogeneous of degree one.
\end{example}

\subsection{Problems with Constraints at the Boundary}

Next we turn to problems with additional constraints at the boundary. For the sake of simplicity we study only 
the one dimensional case (i.e. $d = 1$). Our aim is to obtain optimality conditions for the problem
\begin{equation} \label{eq:BoundaryConstrCalcVarProbl}
\begin{split}
  &\min \: 
  \mathcal{I}(u) = \int_{\alpha}^{\beta} f(x, u(x), u'(x)) \, dx + g_0(u(\alpha), u(\beta)),
  \quad u \in \xWsp^{1, p}((\alpha, \beta); \xR^m) \\
  &\text{subject to} \quad
  g_i(u(\alpha), u(\beta)) \le 0, \: i \in I, \quad g_j(u(\alpha), u(\beta)) = 0, \: j \in J.
\end{split}
\end{equation}
Here $\alpha, \beta \in \xR$, $\alpha < \beta$ (i.e. $\Omega = (\alpha, \beta)$),
$f \colon (\alpha, \beta) \times \xR^m \times \xR^m \to \xR$, and
$g_i \colon \xR^m \times \xR^m \to \xR$, $i \in I \cup J \cup \{ 0 \}$ are given nonsmooth
functions, $I = \{ 1, \ldots, \ell_1 \}$ and $J = \{ \ell_1 + 1, \ldots, \ell_2 \}$ for some 
$\ell_1, \ell_2 \in \mathbb{N} \cup \{ 0 \}$. Observe also that the set $\Omega = (\alpha, \beta)$ is obviously bounded
and has the segment property.

For any $u \in \xWsp^{1, p}((\alpha, \beta); \xR^m)$ denote $I(u) = \{ i \in I \mid g_i(u(\alpha), u(\beta)) = 0 \}$.
For any subset $C$ of a real vector space $E$ denote by
$$
  \cone C = \Big\{ \sum_{i = 1}^n \alpha_i x_i \Bigm| 
  x_i \in C, \enspace \alpha_i \ge 0, \enspace i \in \{ 1, \ldots, n \}, \enspace n \in \mathbb{N} \Big\}
$$
the \textit{conic hull} of $C$ (i.e. the smallest convex cone containing the set $C$). 

To derive optimality conditions for problem \eqref{eq:BoundaryConstrCalcVarProbl} we will use general optimality
conditions for nonsmooth mathematical programming problems in infinite dimensional spaces in terms of
quasidifferentials \cite{Dolgopolik_MetricReg,Dolgopolik_QuasidiffProg}. To this end, we will suppose that the equality
constraints are \textit{polyhedrally codifferentiable}, that is, they are codifferentiable and the sets 
$\underline{d} g_j(u(\alpha), u(\beta))$ and  $\overline{d} g_j(u(\alpha), u(\beta))$ are \textit{polytopes} (i.e.
convex hulls of a finite number of points). This assumption is needed to ensure that certain cones generated by these
sets are closed. It should be noted that this assumption can be replaced by a more restrictive constraint qualification
(see \cite{Dolgopolik_MetricReg,Dolgopolik_QuasidiffProg} for more details). For the sake of shortness, we do not
consider this alternative assumption and leave it to the interested reader. 

Let us also point out that most of codifferentiable functions appearing in applications are, in fact, polyhedrally
codifferentiable (see numerous examples in \cite{DemRub_book}). Thus, the assumption that the equality constraints are
polyhedrally codifferentiable is not very restrictive.

\begin{theorem} \label{thrm:BoundaryConstrainedProbl}
Let $f$ satisfy the codifferentiability conditions of order $p \in [1, + \infty]$, the set-valued maps
$\underline{d}_{u, \xi} f(\cdot)$ and $\overline{d}_{u, \xi} f(\cdot)$ have the form \eqref{eq:IntegrandCodiffL1} in
the case $p = 1$, and $u_*$ be a locally optimal solution of problem \eqref{eq:BoundaryConstrCalcVarProbl}. Suppose also
that the functions $g_i$, $i \in I \cup J \cup \{ 0 \}$, are continuously codifferentiable at the point 
$(u_*(\alpha), u_*(\beta))$, and the sets $\underline{d} g_j(u_*(\alpha), u_*(\beta))$ 
and $\overline{d} g_j(u_*(\alpha), u_*(\beta))$, $j \in J$, are polytopes.
Let finally vectors $(0, s_{1i}, s_{2i}) \in \overline{d} g_i(u_*(\alpha), u_*(\beta))$, $i \in I$,
$(0, s_{1j}, s_{2j}) \in \overline{d} g_j(u_*(\alpha), u_*(\beta))$, and
$(0, r_{1j}, r_{2j}) \in \underline{d} g_j(u_*(\alpha), u_*(\beta))$, $j \in J$, be such that the following constraint
qualification holds true:
\begin{gather} \label{eq:CQBoundaryConstraints1}
  C_j \cap \cone\big\{ - C_k \bigm| k \in J \setminus \{ j \} \big\} = \emptyset \quad \forall j \in J \\
  \co\big\{ \underline{\partial} g_i(u_*(\alpha), u_*(\beta)) + (s_{1i}, s_{2i}) \bigm| i \in I(u_*) \big\} 
  \cap \cone\big\{ - C_j \bigm| j \in J \} = \emptyset, \label{eq:CQBoundaryConstraints2}
\end{gather}
where $C_j = \{ \underline{\partial} g_j(u_*(\alpha), u_*(\beta)) + (s_{1j}, s_{2j}) \} 
\cup \{ - (r_{1j}, r_{2j}) - \overline{\partial} g_j(u_*(\alpha), u_*(\beta)) \}$, and the sets
$\underline{\partial} g_j(u_*(\alpha), u_*(\beta))$ and $\overline{\partial} g_j(u_*(\alpha), u_*(\beta))$
are defined as in Lemma~\ref{lem:ContCodiffQuasidiffUniform}.

Then for all $(0, s_{10}, s_{20}) \in \overline{d} g_0(u_*(\alpha), u_*(\beta))$ and for any measurable
selection $(0, w_1(\cdot), w_2(\cdot))$ of the multifunction 
$\overline{d}_{u, \xi} f(\cdot, u_*(\cdot), u'_*(\cdot))$ there exist an absolutely continuous function 
$\zeta \in \xWsp^{1, p'}((\alpha, \beta); \xR^m)$, $\lambda_i \ge 0$, 
$i \in I$, and $\underline{\mu}_j, \overline{\mu}_j \ge 0$, $j \in J$, such that 
$\lambda_i g_i(u_*(\alpha), u_*(\beta)) = 0$ for all $i \in I$, the Euler-Lagrange inclusion
\begin{equation} \label{eq:EulerLagrangeBoundary}
  (0, \zeta'(x), \zeta(x)) \in \underline{d}_{u, \xi} f(x, u_*(x), u'_*(x)) + (0, w_1(x), w_2(x))
\end{equation}
is satisfied for a.e. $x \in (\alpha, \beta)$, and the following transversality condition holds true:
\begin{multline} \label{eq:Transversality}
  (0, \zeta(\alpha), - \zeta(\beta)) \in \underline{d} g_0(u_*(\alpha), u_*(\beta)) + (0, s_{10}, s_{20})
  + \sum_{i = 1}^{\ell_1} \lambda_i \big( \underline{d} g_i(u_*(\alpha), u_*(\beta)) + (0, s_{1i}, s_{2i}) \big) \\
  + \sum_{j = \ell_1 + 1}^{\ell_2} \underline{\mu}_j 
  \big( \underline{d} g_j(u_*(\alpha), u_*(\beta)) + (0, s_{1j}, s_{2j}) \big)
  - \sum_{j = \ell_1 + 1}^{\ell_2} \overline{\mu}_j 
  \big( (0, r_{1j}, r_{2j}) + \overline{d} g_j(u_*(\alpha), u_*(\beta)) \big).
\end{multline}
\end{theorem}

\begin{proof}
Let us transform problem \eqref{eq:BoundaryConstrCalcVarProbl}. To this end, recall that 
$u \in \xWsp^{1, p}((\alpha, \beta); \xR^m))$ if and only if there exists $h \in \xL^{p}((\alpha, \beta); \xR^m)$
such that $u(x) = u(\alpha) + \int_{\alpha}^x h(\tau) d \tau$ for a.e. $x \in (\alpha, \beta)$ (see, e.g. \cite{Leoni}).
Therefore, the linear operator 
$\mathcal{T} \colon \xR^m \times \xL^{p}((\alpha, \beta); \xR^m) \to \xWsp^{1, p}((\alpha, \beta); \xR^m))$
defined as $\mathcal{T}(\eta, h)(x) = \eta + \int_{\alpha}^x h(\omega) \, d \omega$ is a continuous one-to-one
correspondence. Consequently, the pair $(u_*(\alpha), u'_*)$ is a point of local minimum of the problem
\begin{equation} \label{eq:TransformedBdConstrProblem}
  \min_{(\eta, h) \in \xR^m \times \xL^{p}((\alpha, \beta); \xR^m))} \: \mathcal{J}_0(\eta, h)
  \quad \text{s.t.} \quad \mathcal{J}_i(\eta, h) \le 0, \: i \in I, \quad
  \mathcal{J}_j(\eta, h) = 0, \: j \in J.
\end{equation}
where $\mathcal{J}_0(\eta, h) = \mathcal{I}(\mathcal{T}(\eta, h))$ and 
$\mathcal{J}_i(\eta, h) = g_i(\eta, \mathcal{T}(\eta, h)(\beta))$, $i \in I \cup J$.

By our assumption the functions $g_i$, $i \in I \cup \{ 0 \}$ are continuously codifferentiable at 
$(u_*(\alpha), u_*(\beta))$, while the functional $\int_{\alpha}^{\beta} f(x, u(x), u'(x)) \, dx$ is continuously
codifferentiable by Theorem~\ref{thm:CodiffIntegralFunctional}. Consequently, by
\cite[Thrm.~4.5]{Dolgopolik_AbstrConvApprox} the functions $\mathcal{J}_i$, $i \in I \cup J \cup \{ 0 \}$ are
continuously codifferentiable at the point $(u_*(\alpha), u'_*)$, the set
\begin{multline} \label{eq:HypodiffIntegralTransformed}
  \Big\{ (A, x^*) \in \xR \times \big( \xR^m \times \xL^{p}((\alpha, \beta); \xR^m) \big)^* \Bigm|
  A = \int_{\alpha}^{\beta} a(x) \, dx + a_0, \quad
  \forall (\eta, h) \in \xR^m \times \xL^{p}((\alpha, \beta); \xR^m)
  \\ 
  \langle x^*, (\eta, h) \rangle 
  = \Big\langle \int_{\alpha}^{\beta} v_1(x) \, dx + r_1 + r_2, \eta \Big\rangle 
  + \int_{\alpha}^{\beta} \Big\langle \int_x^{\beta} v_1(\tau) d \tau + v_2(x) + r_2, h(x) \Big\rangle \, dx, 
  \\
  (a_0, r_1, r_2) \in \underline{d} g_0(u_*(\alpha), u_*(\beta)), \enspace
  (a(\cdot), v_1(\cdot), v_2(\cdot)) \text{ is a measurable selection of }
  \underline{d}_{u, \xi} f(\cdot, u_*(\cdot), u'_*(\cdot)) \Big\}
\end{multline}
is a hypodifferential of $\mathcal{J}_0$ at $(u_*(\alpha), u'_*)$, while the set
\begin{multline} \label{eq:HypodiffOfBoundaryConstraint}
  \Big\{ (a, x^*) \in \xR \times \big( \xR^m \times \xL^{p}((\alpha, \beta); \xR^m) \big)^* \Bigm|
  \langle x^*, (\eta, h) \rangle = \langle r_1 + r_2, \eta \rangle + \int_{\alpha}^{\beta} \langle r_2, h(x) \rangle dx
  \\ \forall (\eta, h) \in \xR^m \times \xL^{p}((\alpha, \beta); \xR^m), \quad
  (a, r_1, r_2) \in \underline{d} g_i(u_*(\alpha), u_*(\beta)) \Big\}
\end{multline}
is a hypodifferential of $\mathcal{J}_i$ at $(u_*(\alpha), u'_*)$, $i \in I \cup J$. The hyperdifferentials 
$\overline{d} \mathcal{J}_i(u_*(\alpha), u'_*)$ are defined in the same way. Thus, by
Lemma~\ref{lem:ContCodiffQuasidiffUniform} the functions $\mathcal{J}_i$, $i \in I \cup J \cup \{ 0 \}$, are Hadamard
quasidifferentiable at $(u_*(\alpha), u'_*)$.

With the use of optimality conditions for nonsmooth mathematical programming problems in terms of quasidifferentials
\cite[Crlr.~4, Prp.~1, and Lemma~2]{Dolgopolik_QuasidiffProg} one obtains that if for some 
$x_j^* \in \underline{\partial} \mathcal{J}_j(u_*(\alpha), u'_*)$, $j \in J$, and
$y_i^* \in \overline{\partial} \mathcal{J}_i(u_*(\alpha), u'_*)$, $i \in I \cup J$, one has
\begin{equation} \label{eq:CQBoundaryConstraints_Rewritten}
  D_j \cap \cone\{ - D_k \mid k \in J \setminus \{ j \} \} = \emptyset \enspace \forall j \in J, \enspace
  \co\big\{ \underline{\partial} \mathcal{J}_i(u_*(\alpha), u'_*) + y_i^* \bigm| i \in I(u_*) \big\} \cap 
  \cone\{ - D_j \mid j \in J \} = \emptyset
\end{equation}
(here $D_j = \{ \underline{\partial} \mathcal{J}_j(u_*(\alpha), u'_*) + y_j^* \} \cup 
\{ - x_j^* - \overline{\partial} \mathcal{J}_j(u_*(\alpha), u'_*) \}$ for any $j \in J$), 
then for any $y_0^* \in \overline{\partial} \mathcal{J}_0(u_*(\alpha), u'_*)$ there exist 
$\lambda_i \ge 0$, $i \in I$, and $\underline{\mu}_j, \overline{\mu}_j \ge 0$, $j \in J$, such that
$\lambda_i \mathcal{J}_i(u_*(\alpha), u'_*) = 0$ for any $i \in I$ and
\begin{equation} \label{eq:BoundaryConstr_AbstractOptCond}
\begin{split}
  0 \in \underline{\partial} \mathcal{J}_0(u_*(\alpha), u'_*) + y_0^* 
  &+ \sum_{i \in I} \lambda_i \big( \underline{\partial} \mathcal{J}_j(u_*(\alpha), u'_*) + y_i^* \big) \\
  &+ \sum_{j \in J} \underline{\mu}_j \big( \underline{\partial} \mathcal{J}_j(u_*(\alpha), u'_*) + y_j^* \big)
  - \sum_{j \in J} \overline{\mu}_j \big( x_j^* + \overline{\partial} \mathcal{J}_j(u_*(\alpha), u'_*) \big).
\end{split}
\end{equation}
Let us rewrite these optimality conditions in terms of the original problem \eqref{eq:BoundaryConstrCalcVarProbl}.

For all $(s_1, s_2) \in \xR^m \times \xR^m$ define linear functional $\Theta(s_1, s_2)$
as follows: 
\begin{equation} \label{eq:LinFuncCorrespondBoundaryTerms}
  \langle \Theta(s_1,s_2), (\eta, h) \rangle 
  = \langle s_{1i} + s_{2i}, \eta \rangle + \int_{\alpha}^{\beta} \langle s_{2i}, h(x) \rangle dx
  \quad \forall (\eta, h) \in \xR^m \times \xL^{p}((\alpha, \beta); \xR^m).
\end{equation}
Fix any $(0, s_{1i}, s_{2i}) \in \overline{d} g_i(u_*(\alpha), u_*(\beta))$, $i \in I \cup J$, and
$(0, r_{1j}, r_{2j}) \in \underline{d} g_j(u_*(\alpha), u_*(\beta))$, $j \in J$ satisfying 
\eqref{eq:CQBoundaryConstraints1} and \eqref{eq:CQBoundaryConstraints2}, and put
$y_i^* = \Theta(s_{1i}, s_{2i})$, $i \in I \cup J$, and $x_j^* = \Theta(r_{1j}, r_{2j})$, $j \in J$.
Then $y_i^* \in \overline{\partial} \mathcal{J}_i(u_*(\alpha), u_*')$ for all $i \in I \cup J$ and
$x_j^* \in \underline{\partial} \mathcal{J}_j(u_*(\alpha), u_*')$ for all $j \in J$
according to \eqref{eq:HypodiffOfBoundaryConstraint} and Lemma~\ref{lem:ContCodiffQuasidiffUniform}. Let us check that
these functionals $y_i^*$ and $x_j^*$ satisfy constraint qualification \eqref{eq:CQBoundaryConstraints_Rewritten}.

Indeed, by virtue of \eqref{eq:HypodiffOfBoundaryConstraint} and Lemma~\ref{lem:ContCodiffQuasidiffUniform} one has
$\Theta(\underline{\partial} g_j(u_*(\alpha), u_*(\beta)) = \underline{\partial} \mathcal{J}_j(u_*(\alpha), u_*')$ and
the same equality holds true for the superdifferentials. Hence taking into account the fact that $\Theta$ is a linear
operator (see~\eqref{eq:LinFuncCorrespondBoundaryTerms}) one obtains that $\Theta(C_j) = D_j$, $j \in J$, and
$\Theta( \cone\{ C_k \mid k \in J \setminus \{ j \} \} ) = \cone\{ D_k \mid k \in J \setminus \{ j \} \}$ 
for all $j \in J$. One can readily verify that $\Theta$ is an injective mapping
(see~\eqref{eq:LinFuncCorrespondBoundaryTerms}). Therefore \eqref{eq:CQBoundaryConstraints1} implies the first condition
in \eqref{eq:CQBoundaryConstraints_Rewritten}. Similarly, \eqref{eq:CQBoundaryConstraints2} implies the second condition
in \eqref{eq:CQBoundaryConstraints_Rewritten}.

Thus, constraint qualification \eqref{eq:CQBoundaryConstraints_Rewritten} is satisfied. Consequently, with the use of
\eqref{eq:BoundaryConstr_AbstractOptCond}, \eqref{eq:HypodiffIntegralTransformed},
\eqref{eq:HypodiffOfBoundaryConstraint}, and Lemma~\ref{lem:ContCodiffQuasidiffUniform} one gets that for all 
$(0, s_{10}, s_{20}) \in \overline{d} g_0(u_*(\alpha), u_*(\beta))$ and for any measurable selection 
$(0, w_1(\cdot), w_2(\cdot))$ of the multifunction $\overline{d}_{u, \xi} f(\cdot, u_*(\cdot), u'_*(\cdot))$  there
exist $\lambda_i \ge 0$, $i \in I$, $\underline{\mu}_j, \overline{\mu}_j \ge 0$, $j \in J$,
vectors $(0, r_{1i}, r_{2i}) \in \underline{d} g_i(u_*(\alpha), u_*(\beta))$, $i \in I \cup \{ 0 \}$, 
$(0, \xi_{1j}, \xi_{2j}) \in \underline{d} g_j(u_*(\alpha), u_*(\beta))$, $j \in J$, and
$(0, y_{1j}, y_{2j}) \in \overline{d} g_j(u_*(\alpha), u_*(\beta))$, $j \in J$, and a measurable
selection $(0, v_1(\cdot), v_2(\cdot))$ of the multifunction $\underline{d}_{u, \xi} f(\cdot, u_*(\cdot), u'_*(\cdot))$
such that for all $i \in I$ one has $\lambda_i g_i(u_*(\alpha), u_*(\beta)) = 0$, for any $\eta \in \mathbb{R}^m$ one
has
\begin{multline} \label{eq:TransversalityRewritten}
  \bigg\langle \int_{\alpha}^{\beta} (v_1(x) + w_1(x)) \, dx + r_{10} + r_{20} + s_{10} + s_{20}
  + \sum_{i \in I} \lambda_i \big( r_{1i} + r_{2i} + s_{1i} + s_{2i} \big) \\
  + \sum_{j \in J} \underline{\mu}_j \big( \xi_{1j} + \xi_{2j} + s_{1j} + s_{2j} \big) 
  - \sum_{j \in J} \overline{\mu}_j \big( r_{1j} + r_{2j} + y_{1j} + y_{2j} \big), \eta \bigg \rangle = 0,
\end{multline}
and for any $h \in \xL^{p}((\alpha, \beta); \xR^m)$ one has
\begin{multline} \label{eq:EulerLagrangeBoundaryRewritten}
  \int_{\alpha}^{\beta} \bigg\langle \int_{x}^{\beta} (v_1(\tau) + w_1(\tau)) \, d \tau + v_2(x) + w_2(x)
  + r_{20} + s_{20} \\
  + \sum_{i \in I} \lambda_i \big( r_{2i} + s_{2i} \big)
  + \sum_{j \in J} \underline{\mu}_j \big( \xi_{2j} + s_{2j} \big)
  - \sum_{j \in J} \overline{\mu}_j \big( r_{2j} + y_{2j} \big), h(x) \bigg\rangle \, dx = 0.
\end{multline}
Denote 
$$
  \zeta(x) = - \int_{x}^{\beta} (v_1(\tau) + w_1(\tau)) \, d \tau
  - r_{20} - s_{20}
  - \sum_{i \in I} \lambda_i \big( r_{2i} + s_{2i} \big)
  - \sum_{j \in J} \underline{\mu}_j \big( \xi_{2j} + s_{2j} \big)
  + \sum_{j \in J} \overline{\mu}_j \big( r_{2j} + y_{2j} \big)
$$
for any $x \in [\alpha, \beta]$. Then $\zeta$ is an absolutely continuous function such that
$$
  (0, \zeta'(x), \zeta(x)) = (0, v_1(x) + w_1(x), v_2(x) + w_2(x)) \quad \text{for a.e. } x \in (\alpha, \beta)
$$
due to \eqref{eq:EulerLagrangeBoundaryRewritten}, $\zeta \in \xWsp^{1, p'}((\alpha, \beta); \xR^m)$ due to the
growth condition on the codifferential mapping $D_{u, \xi} f(\cdot)$ (see Def.~\ref{def:CodiffCond}),
and
\begin{align*}
  (0, \zeta(\alpha), -\zeta(\beta)) &= (0, r_{10} + s_{10}, r_{20} + s_{20})
  + \sum_{i \in I} \lambda_i \big( 0, r_{1i} + s_{1i}, r_{2i} + s_{2i} \big) \\
  &+ \sum_{j \in J} \underline{\mu}_j \big( 0, \xi_{1j} + s_{1j}, \xi_{2j} + s_{2j} \big)
  - \sum_{j \in J} \overline{\mu}_j \big( 0, r_{1j} + y_{1j}, r_{2j} + y_{2j} \big)
\end{align*}
due to \eqref{eq:TransversalityRewritten}. It remains to note that the first equality above is equivalent to
\eqref{eq:EulerLagrangeBoundary}, while the second one is equivalent to the transversality condition
\eqref{eq:Transversality}.
\end{proof}

\begin{remark} \label{rmrk:BoundaryConstraints}
{(i)~In the case when there are no equality constraints, the constraint qualification \eqref{eq:CQBoundaryConstraints1},
\eqref{eq:CQBoundaryConstraints2} from the previous theorem takes an especially simple form. Namely, it is sufficient to
suppose that for some $(0, s_{1i}, s_{2i}) \in \overline{d} g_i(u_*(\alpha), u_*(\beta))$, $i \in I(u_*)$, one has
$0 \notin \co\{ \underline{d} g_i(u_*(\alpha), u_*(\beta)) + (0, s_{1i}, s_{2i}) \mid i \in I(u_*) \}$. In the case when
there are no inequality constraints and there is only one equality constraint, the constraint qualification
\eqref{eq:CQBoundaryConstraints1}, \eqref{eq:CQBoundaryConstraints2} also takes a very simple form. One has to suppose
that $0 \notin \underline{d} g_1(u_*(\alpha), u_*(\beta)) + (0, s_1, s_2)$ and 
$0 \notin (0, r_1, r_2) + \overline{d} g_1(u_*(\alpha), u_*(\beta))$ 
for some $(0, s_1, s_2) \in \overline{d} g_1(u_*(\alpha), u_*(\beta))$ and
$(0, r_1, r_2) \in \underline{d} g_1(u_*(\alpha), u_*(\beta))$.
}

\noindent{(ii)~It should be noted that there are \textit{two} Lagrange multipliers $\underline{\mu}_j$ and
$\overline{\mu}_j$ corresponding to each equality constraint $g_j (u(\alpha), u_j(\beta)) = 0$, which is a specific
feature of optimality conditions for nonsmooth optimisation problems in terms of quasidifferentials. See
\cite{Dolgopolik_MetricReg,Dolgopolik_QuasidiffProg} for more details.
}
\end{remark}

Let us also present an example illustrating optimality conditions for problem \eqref{eq:BoundaryConstrCalcVarProbl}.

\begin{example} \label{example:BoundaryConstrainedProblem}
Let $d = 1$, $m = 2$, $p = 2$, and $\Omega = (0, 1)$. Consider the following problem:
\begin{equation} \label{eq:BoundaryConstr_Example}
\begin{split}
  &\min \: 
  \mathcal{I}(u_1, u_2) 
  = \int_0^1 \left( \frac{1}{2} \big( (u_1'(x))^2 + (u_2'(x))^2 \big) + |u_1(x)| + |u_2(x)| \right) \, dx 
  - u_1(1) + u_2(1) \\
  &\text{subject to} \quad |u_1(1)| - |u_2(1)| = 0, \quad u \in \xWsp^{1, 2}((0, 1); \xR^m).
\end{split}
\end{equation}
In this case, $f(x, u, \xi) = |u_1| + |u_2| + 0.5 (\xi_1)^2 + 0.5 (\xi_2)^2$, $g_0(u(0), u(1)) = - u_1(1) + u_2(1)$, 
$I = \emptyset$, $J = \{ 1 \}$, and $g_1(u(0), u(1)) = |u_1(1)| - |u_2(1)|$. Let us check whether the function 
$u_* = (0, 0)$ is an optimal solution of problem \eqref{eq:BoundaryConstr_Example}.

First we apply optimality conditions for problem \eqref{eq:BoundaryConstr_Example} in terms of the Clarke
subdifferential \cite[Thrm.~4.4.1]{Clarke}. Denote $L(u, \xi) = |u_1| + |u_2| + 0.5 (\xi_1)^2 + 0.5 (\xi_2)^2$.
The Hamiltonian for problem \eqref{eq:BoundaryConstr_Example} is defined as
$$
  H(u, p) = \sup_{\xi \in \mathbb{R}^2} \big( \langle p, \xi \rangle - L(u, \xi) \big)
  = - |u_1| - |u_2| + \frac{1}{2} \big( p_1^2 + p_2^2 \big).
$$
The Clarke subdifferentials of the Hamiltonian and the function $\ell(u(1)) = - u_1(1) + u_2(1)$ at the origin 
have the form:
$$
  \partial_{Cl} H(0, 0) = \co\left\{ \left( \begin{smallmatrix} 1 \\ 1 \\ 0 \\ 0 \end{smallmatrix} \right),
  \left( \begin{smallmatrix} 1 \\ -1 \\ 0 \\ 0 \end{smallmatrix} \right),
  \left( \begin{smallmatrix} -1 \\ 1 \\ 0 \\ 0 \end{smallmatrix} \right),
  \left( \begin{smallmatrix} -1 \\ -1 \\ 0 \\ 0 \end{smallmatrix} \right) \right\}, \quad
  \partial_{Cl} \ell(0) = \left\{ \left( \begin{smallmatrix} -1 \\ 1 \end{smallmatrix} \right) \right\}.
$$
Furthermore, one can verify that the Clarke normal cone $N_S(0)$ to 
the set $S = \{ (y, z) \in \xR^2 \mid |y| - |z| = 0 \}$ at the origin (see \cite[Sect.~2.4]{Clarke}) is equal 
to $\xR^2$. Therefore, for the function $p_*(x) \equiv 0$ one has
$$
  \left( \begin{smallmatrix} - p_*'(x) \\ u_*'(x) \end{smallmatrix} \right) \in \partial_{Cl} H(u_*(x), p_*(x))
  \quad \forall x \in [0, 1], \quad
  p_*(0) = 0, \quad - p_*(1) \in \partial_{Cl} \ell(u_*(1)) + N_S(u_*(1)),
$$
i.e. optimality conditions in terms of the Clarke subdifferential \cite[Thrm.~4.4.1]{Clarke} are satisfied at $u_*$.

Next we check optimality conditions in terms of the limiting proximal subdifferential from \cite{IoffeRockafellar}.
Define $\ell(u(0), u(1)) = - u_1(1) + u_2(1)$, if $|u_1(1)| - |u_2(1)| = 0$, and $\ell(u(0), u(1)) = + \infty$,
otherwise. Then problem \eqref{eq:BoundaryConstr_Example} can be rewritten as the following generalised problem of
Bolza:
\begin{equation} \label{eq:GeneralisedBolzaProblem}
  \min \: \ell(u(0), u(1)) + \int_0^1 L(u(x), u'(x)) \, dx, \quad u \in W^{1, 1}((0, 1); \xR^2).
\end{equation}
One can readily verify that the limiting proximal subdifferentials, which we denote by $\partial_p^{\infty}$, of the
functions $L$ and $\ell$ at the origin have the form:
$$
  \partial_p^{\infty} L(0, 0) = \co\left\{ \left( \begin{smallmatrix} 1 \\ 1 \\ 0 \\ 0 \end{smallmatrix} \right),
  \left( \begin{smallmatrix} 1 \\ -1 \\ 0 \\ 0 \end{smallmatrix} \right),
  \left( \begin{smallmatrix} -1 \\ 1 \\ 0 \\ 0 \end{smallmatrix} \right),
  \left( \begin{smallmatrix} -1 \\ -1 \\ 0 \\ 0 \end{smallmatrix} \right) \right\}, \quad
  \partial_p^{\infty} \ell(0, 0) = \left( \begin{smallmatrix} 0 \\ 0 \\ -1 \\ 1 \end{smallmatrix} \right)
  + \left\{ \left( \begin{smallmatrix} 0 \\ 0 \\ t \\ - t \end{smallmatrix} \right),
  \left( \begin{smallmatrix} 0 \\ 0 \\ t \\ t \end{smallmatrix} \right) \bigg| t \in \xR \right\}.
$$
Therefore, for the function $p_*(x) \equiv 0$ one has
$$
  p_*'(x) \in \co\big\{ w \in \xR^2 \bigm| (w, p_*(x)) \in \partial_p^{\infty} L(u_*(x), u_*'(x)) \big\} = 
  \co\left\{ \left( \begin{smallmatrix} 1 \\ 1 \end{smallmatrix} \right),
  \left( \begin{smallmatrix} 1 \\ -1 \end{smallmatrix} \right),
  \left( \begin{smallmatrix} -1 \\ 1 \end{smallmatrix} \right),
  \left( \begin{smallmatrix} -1 \\ -1 \end{smallmatrix} \right) \right\}
  \quad \forall x \in [0, 1]
$$
and $(p_*(0), -p_*(1)) \in \partial_p^{\infty} \ell(u_*(0), u_*(1))$. Furthermore, one also has
$$
  L(u_*(x), y) = \frac{1}{2} \big( (y_1)^2 + (y_2)^2 \big) \ge 0 
  = L(u_*(x), u_*'(x)) + \langle p_*(x), y - u_*'(x) \rangle \quad \forall y = (y_1, y_2) \in \xR^2.
$$
Thus, optimality conditions in terms of the limiting proximal subdifferential \cite[Thrm.~1]{IoffeRockafellar} are
satisfied at $u_*$ as well. In addition, one can check that the limiting proximal subdifferentials 
$\partial_p^{\infty} L(0, 0)$ and $\partial_p^{\infty} l(0, 0)$ coincide with the corresponding limiting Fr\'{e}chet
subdifferentials (which, in turn, coincide with the Mordukhovich basic subdifferentials by
\cite[Thrm.~1.89]{Mordukhovich_I}), which implies that the optimality conditions in terms of the limiting Fr\'{e}chet
subdifferential \cite[Thrm.~3.4]{Jourani2009} are satisfied at $u_*$ as well.

Let us finally check optimality conditions in terms of codifferentials from Theorem~\ref{thrm:BoundaryConstrainedProbl}.
Applying the codifferential calculus \cite{DemRub_book} one obtains that
$$
  \underline{d}_{u, \xi} f(x, u, \xi) = \co\Big\{ 
  \Big( (-1)^i u_1 - |u_1| + (-1)^j u_2 - |u_2|, (-1)^i, (-1)^j, \xi_1, \xi_2 \Big) 
  \in \xR \times \xR^2 \times \xR^2 \Bigm| i, j \in \{ 1, 2 \}\Big\}
$$
and $\overline{d}_{u, \xi} f(x, u, \xi) = \{ 0 \}$. One also gets that
$$
  \underline{d} g_0(y, z) 
  = \left\{ \left( \begin{smallmatrix} 0 \\0 \\ 0 \\ -1 \\ 1 \end{smallmatrix} \right) \right\}, 
  \quad \overline{d} g_0(y, z) = \{ 0 \}, \quad
  \underline{d} g_1(y, z) 
  = \co\left\{ \left( \begin{smallmatrix} \pm z_1 - |z_1| \\ 0 \\ 0 \\ \pm 1 \\ 0 \end{smallmatrix} \right) \right\}, 
  \quad \overline{d} g_1(y, z) 
  = \co\left\{ \left( \begin{smallmatrix} \pm z_2 + |z_2| \\ 0 \\ 0 \\ 0 \\ \pm 1 \end{smallmatrix} \right) \right\}, 
$$
for all $y = (y_1, y_2) \in \xR^2$ and $z = (z_1, z_2) \in \xR^2$. Thus, the integrand $f$ satisfies the
codifferentiability conditions of order $p = 2$, the functions $g_0$ and $g_1$ are continuously codifferentiable, and
the sets $\underline{d} g_1(u_*(0), u_*(1))$ and $\overline{d} g_1(u_*(0), u_*(1))$ are polytopes. Furthermore, observe
that for vectors $\eta_1 = (0, 0, 0, 1, 0) \in \underline{d} g_1(u_*(0), u_*(1))$ and 
$\eta_2 = (0, 0, 0, 0, 1) \in \overline{d} g_1(u_*(0), u_*(1))$ one has 
$0 \notin \underline{d} g_1(u_*(0), u_*(1)) + \eta_2$
and $0 \notin \eta_1 + \overline{d} g_1(u_*(0), u_*(1))$, i.e. the constraint qualification from
Theorem~\ref{thrm:BoundaryConstrainedProbl} is satisfied at $u_*$ (see~Remark~\ref{rmrk:BoundaryConstraints}). 

Suppose that optimality conditions from Theorem~\ref{thrm:BoundaryConstrainedProbl} are satisfied at $u_*$. Then there
exist an absolutely continuous function $\zeta \in \xWsp^{1, 2}((0, 1); \xR^2)$ and
$\underline{\mu}_1, \overline{\mu}_1 \ge 0$ such that
$$
  (0, \zeta'(x), \zeta(x)) \in \underline{d}_{u, \xi} f(x, u_*(x), u_*'(x)) = 
  \co\Big\{ \Big( 0, (-1)^i, (-1)^j, 0, 0 \Big) \in \xR \times \xR^2 \times \xR^2 \Bigm| i, j \in \{ 1, 2 \}\Big\}
$$
for a.e. $x \in (0, 1)$, and the transversality condition
$$
  (0, \zeta(0), - \zeta(1)) \in \nabla g_0(u_*(0), u_*(1)) 
  + \underline{\mu}_1 \big( \underline{d} g_1(u_*(0), u_*(1)) + \eta_2 \big)
  - \overline{\mu}_1 \big( \eta_1 + \overline{d} g_1(u_*(0), u_*(1)) \big).
$$
holds true. Therefore $\zeta(x) \equiv 0$ and the transversality condition takes the form
$$
  0 \in \begin{pmatrix} -1 \\ 1 \end{pmatrix} + 
  \underline{\mu}_1 \co\left\{ \begin{pmatrix} 1 \\ 1 \end{pmatrix}, \begin{pmatrix} - 1 \\ 1 \end{pmatrix} \right\}
  - \overline{\mu}_1 \co\left\{ \begin{pmatrix} 1 \\ 1 \end{pmatrix}, \begin{pmatrix} 1 \\ -1 \end{pmatrix} \right\}
$$
or, equivalently,
$$
  -1 - \underline{\mu}_1 - \overline{\mu}_1 \le 0 \le -1 + \underline{\mu}_1 - \overline{\mu}_1, \quad
  1 + \underline{\mu}_1 - \overline{\mu}_1 \le 0 \le 1 + \underline{\mu}_1 + \overline{\mu}_1.
$$
The third inequality implies that $1 + \underline{\mu}_1 \le \overline{\mu}_1$, while the second one yields
$1 + \overline{\mu}_1 \le \underline{\mu}_1$. Consequently, 
$2 + \underline{\mu}_1 \le 1 + \overline{\mu}_1 \le \underline{\mu}_1$, which is impossible. Hence optimality
conditions from Theorem~\ref{thrm:BoundaryConstrainedProbl} are not satisfied at $u_*$, and one can conclude that this
point is not an optimal solution of problem \eqref{eq:BoundaryConstr_Example}. Thus, optimality conditions in terms of
codifferentials detect the non-optimality of $u_*$, while optimality conditions in terms of Clarke, limiting proximal,
and limiting Fr\'{e}chet subdifferentials fail to do so.
\end{example}

\subsection{Problems with Isoperimetric Constraints}

Let us now consider problems with isoperimetic inequality constraints. With the use of optimality conditions for general
quasidifferentiable programming problems in Banach spaces \cite{Dolgopolik_MetricReg,Dolgopolik_QuasidiffProg} one can
derive optimality conditions for nonsmooth problems with both isoperimetric equality and inequality constraints.
However, this approach requires the use of constraint qualifications (similar to the ones used in
Theorem~\ref{thrm:BoundaryConstrainedProbl}), whose reformulation in the case of isoperimetric constraints leads to very
cumbersome assumptions, which we do not present here for the sake of shortness.

Consider isoperimetric problem of the form:
\begin{equation} \label{eq:IsoperimetricInequalProbl}
\begin{split}
  &\min \: \mathcal{I}_0(u) = \int_{\Omega} f_0(x, u(x), \nabla u(x)) \, dx \\ 
  &\text{subject to} \quad
  \mathcal{I}_i(u) = \int_{\Omega} f_i(x, u(x), \nabla u(x)) \, dx \le 0, \quad
  u \in u_0 + \xWsp^{1, p}_0(\Omega; \xR^m).
\end{split}
\end{equation}
Here $\Omega \subset \xR^d$ is an open set, 
$f_i \colon \Omega \times \xR^m \times \xR^{m \times d} \to \xR$, $f_i = f_i(x, u, \xi)$, are
nonsmooth functions, $i \in I \cup \{ 0 \}$, $I = \{ 1, \ldots, \ell \}$, and $u_0 \in \xWsp^{1, p}(\Omega; \xR^m)$ is
a fixed function. Denote $I(u) = \{ i \in I \mid \mathcal{I}_i(u) = 0 \}$.

Our aim is to derive optimality conditions for problem \eqref{eq:IsoperimetricInequalProbl} with the use of general
optimality conditions in terms of quasidifferentials for inequality constrained nonsmooth optimisation problem
\cite[Crlr.~5]{Dolgopolik_QuasidiffProg}. It should be noted that these optimality conditions were largely inspired by
B.N. Pschenichny work \cite{Pschenichny} and are derived with the use of the standard trick, which goes back to
Pschenichny, of reducing an inequality constrained optimisation problem to the problem of minimising the nonsmooth
max-envelope of the objective function and constraints.

\begin{theorem} \label{thrm:IsoperimetricInequalProbl}
Let $f_i$, $i \in I \cup \{ 0 \}$, satisfy the codifferentiability conditions of order $p \in [1, + \infty]$, and let
either $1 < p \le + \infty$ or the set $\Omega$ be bounded and have the segment property. Suppose also that $u_*$ is a
locally optimal solution of problem \eqref{eq:IsoperimetricInequalProbl}. 
Let finally $(0, w_{1i}(\cdot), w_{2i}(\cdot))$ be measurable
selections of the multifunctions $\overline{d}_{u, \xi} f_i(\cdot, u_*(\cdot), \nabla u_*(\cdot))$, $i \in I$, such that
there does not exist $\zeta \in \xL^{p'}(\Omega; \xR^{m \times d}; \diverg)$ satisfying the following inclusion 
for a.e. $x \in \Omega$:
\begin{equation} \label{eq:IsoperimIneq_CQ}
  (0, \diverg(\zeta)(x), \zeta(x)) \in 
  \co\big\{ \underline{d}_{u, \xi} f_i(x, u_*(x), \nabla u_*(x)) + (0, w_{1i}(x), w_{2i}(x)) 
  \bigm| i \in I(u_*) \big\}.
\end{equation}
Then for any measurable selection $(0, w_{10}(\cdot), w_{20}(\cdot))$ of the set-valued map 
$\overline{d}_{u, \xi} f_0(\cdot, u_*(\cdot), \nabla u_*(\cdot))$ one can find $\lambda_i \ge 0$, $i \in I$, and 
$\zeta \in \xL^{p'}(\Omega; \xR^{m \times d}; \diverg)$ such that $\lambda_i \mathcal{I}_i(u_*) = 0$ for any 
$i \in I$, and for a.e. $x \in \Omega$ one has
\begin{equation} \label{eq:IsoperimIneq_OptCond}
\begin{split}
  (0, \diverg(\zeta)(x), \zeta(x)) &\in 
  \underline{d}_{u, \xi} f_0(x, u_*(x), \nabla u_*(x)) + (0, w_{10}(x), w_{20}(x)) \\
  &+ \sum_{i = 1}^{\ell} \lambda_i 
  \Big( \underline{d}_{u, \xi} f_i(x, u_*(x), \nabla u_*(x)) + (0, w_{1i}(x), w_{2i}(x)) \Big).
\end{split}
\end{equation}
\end{theorem}

\begin{proof}
For any $h \in \xWsp^{1, p}_0(\Omega; \xR^m)$ define $\mathcal{J}_0(h) = \mathcal{I}_0(u_* + h)$, and
$\mathcal{J}_i(h) = \mathcal{I}_i(u_* + h)$, $i \in I$. By Corollary~\ref{crlr:Quasidiff_ZeroTraceSobolev}
and Remark~\ref{rmrk:Quasidiff_ZeroTraceSobolev} the functions $\mathcal{J}_i$ are correctly defined and
quasidifferentiable at $h = 0$. Moreover, the point $h = 0$ is a locally optimal solution of the problem
$$
  \min_{h \in \xWsp^{1, p}_0(\Omega; \xR^m)} \mathcal{J}_0(h) \quad \text{subject to} \quad
  \mathcal{J}_i(h) \le 0, \quad i \in I,
$$
since $u_*$ is a locally optimal solution of problem \eqref{eq:IsoperimetricInequalProbl}. Hence by applying optimality
conditions for quasidifferentiable programming problems with inequality constraints
\cite[Crlr.~5]{Dolgopolik_QuasidiffProg} one obtains that if $y_i^* \in \overline{\partial} \mathcal{J}_i(0)$, 
$i \in I$, are such that 
\begin{equation} \label{eq:IsoperimIneq_CQ_Abstract}
  0 \notin \co\{ \underline{\partial} \mathcal{J}_i(0) + y_i^* \mid i \in I(u_*) \},
\end{equation}
then for any $y_0^* \in \overline{\partial} \mathcal{J}_0(0)$ one can find $\lambda_i \ge 0$, $i \in I$, such that 
$\lambda_i \mathcal{J}_i(0) = 0$ for any $i \in I$ and
\begin{equation} \label{eq:IsoperimIneq_OptCond_Abstract}
  0 \in \underline{\partial} \mathcal{J}_0(0) + y_0^* 
  + \sum_{i \in I} \lambda_i \big( \underline{\partial} \mathcal{J}_i(0) + y_i^* \big).
\end{equation}
Let us reformulate these optimality conditions in term of problem \eqref{eq:IsoperimetricInequalProbl}.

Fix any $(0, w_{1i}(\cdot), w_{2i}(\cdot))$, $i \in I$, satisfying the assumptions of the theorem. Define
\begin{equation} \label{eq:FunctionalViaSelect_Isoperim}
  \langle y_i^*, h \rangle 
  = \int_{\Omega} \big( \langle w_{1i}(x), h(x) \rangle + \langle w_{2i}(x), \nabla h(x) \rangle \big) \, dx
  \quad \forall h \in \xWsp^{1, p}_0(\Omega; \xR^m).
\end{equation}
Then by Corollary~\ref{crlr:Quasidiff_ZeroTraceSobolev} one has $y_i^* \in \overline{\partial} \mathcal{J}_i(0)$, 
$i \in I$. Let us check that constraint qualification \eqref{eq:IsoperimIneq_CQ_Abstract} holds true. Indeed, arguing by
reductio ad absurdum suppose that condition \eqref{eq:IsoperimIneq_CQ_Abstract} is not satisfied. Then for any $i \in I$
there exist $x_i^* \in \underline{\partial} \mathcal{J}_i(0)$ and $\alpha_i \ge 0$ such that
$$
  \sum_{i = 1}^{\ell} \alpha_i (x_i^* + y_i^*) = 0, \quad \sum_{i = 1}^{\ell} \alpha_i = 1.
$$
Hence with the use of Corollary~\ref{crlr:Quasidiff_ZeroTraceSobolev} one obtains that for any $i \in I$ there exists a
measurable selection $(0, v_{1i}(\cdot), v_{2i}(\cdot))$ of the multifunction
$\underline{d}_{u, \xi} f_i(\cdot, u_*(\cdot), \nabla u_*(\cdot))$ such that
\begin{equation} \label{eq:IsoperimCQ_Rewritten}
  \int_{\Omega} \big( \langle \omega_1(x), h(x) \rangle + \langle \omega_2(x), \nabla h(x) \rangle \big) \, dx = 0
  \quad \forall h \in \xWsp^{1, p}_0(\Omega; \xR^m),
\end{equation}
where
$$
  \omega_1(x) = \sum_{i = 1}^m \alpha_i( v_{1i}(x) + w_{1i}(x) ), \quad
  \omega_2(x) = \sum_{i = 1}^m \alpha_i( v_{12}(x) + w_{12}(x) ).
$$
Equality \eqref{eq:IsoperimCQ_Rewritten} implies that there exists the weak divergence of the function 
$\zeta = \omega_2$, and $\diverg \zeta = \omega_1$. By the growth condition on the codifferential mappings 
$D_{u, \xi} f_i(\cdot)$ (see~Def.~\ref{def:CodiffCond}) one has $\omega_2 \in \xL^{p'}(\Omega; \xR^{m \times d})$ and 
$\omega_1 \in \xL^{p'}(\Omega; \xR^m)$. Thus, there exists $\zeta \in \xL^{p'}(\Omega; \xR^{m \times d}; \diverg)$ such
that
$$
  (0, \diverg(\zeta)(x), \zeta(x)) 
  = \sum_{i = 1}^{\ell} \alpha_i \big( (a_i(x), v_{1i}(x), v_{2i}(x)) + (0, w_{1i}(x), w_{2i}(x)) \big)
$$
for a.e. $x \in \Omega$, which contradicts \eqref{eq:IsoperimIneq_CQ}. Therefore constraint qualification
\eqref{eq:IsoperimIneq_CQ_Abstract} holds true.

Choose any measurable selection $(0, w_{10}(\cdot), w_{20}(\cdot))$ of the set-valued mapping 
$\overline{d}_{u, \xi} f_0(\cdot, u_*(\cdot), \nabla u_*(\cdot))$. Define linear functional $y_0^*$ in the same way
as in \eqref{eq:FunctionalViaSelect_Isoperim}. Then by Corollary~\ref{crlr:Quasidiff_ZeroTraceSobolev} one has 
$y_0 \in \overline{\partial} \mathcal{J}_0(0)$. Consequently, there exist $\lambda_i \ge 0$, $i \in I$, such that 
$\lambda_i \mathcal{J}_i(0) = \lambda_i \mathcal{I}_i(u_*) = 0$ for any $i \in I$ and
\eqref{eq:IsoperimIneq_OptCond_Abstract} holds true. Now, arguing in the same way as in the proof of
Theorem~\ref{thrm:UnconstrainedCalcVarProbl} one can readily verify that optimality condition
\eqref{eq:IsoperimIneq_OptCond_Abstract} is equivalent to \eqref{eq:IsoperimIneq_OptCond}.
\end{proof}

Let us give an example illustrating optimality conditions for isoperimetric problems from the theorem above.

\begin{example} \label{example:IsoperimetricProblem}
Let $d = m = p = 1$ and $\Omega = (0, 1)$. Consider the following problem:
\begin{equation} \label{eq:IsoperimetricProb_Example}
\begin{split}
  &\min \: \mathcal{I}_0(u) = \int_0^1 \max\big\{ - |u(x)|, - |u'(x)| \big\} \, dx \\
  &\text{subject to } \mathcal{I}_1(u) = \int_0^1 u(x) \, dx \le 0, \quad
  u(0) = u(1) = 0, \quad u \in \xWsp^{1, 1}(0, 1).
\end{split}
\end{equation}
In this case $f_0(x, u, \xi) = \max\{ -|u|, -|\xi| \}$, $I = \{ 1 \}$ and $f_1(x, u, \xi) = u$. Let us
check whether optimality conditions are satisfied at $u_* \equiv 0$. It is easily seen that this function is not a
locally optimal solution of problem \eqref{eq:IsoperimetricProb_Example}, since 
for the function $u_{\alpha}(x) = \alpha x(x - 1)$ one obviously has $\mathcal{I}_0(u_{\alpha}) < 0$ and
$\mathcal{I}_1(u_{\alpha}) = - \alpha / 6 < 0$ for any $\alpha > 0$. In actuality, $u_*$ is a point of unconstrained
global maximum of $\mathcal{I}_0(u)$.

To the best of the author's knowledge, optimality conditions for nonsmooth variational problems with isoperimetric
constraints have been obtained earlier only in \cite[Thrm.~3.5.1]{Bellaassali}. Let us verify whether these optimality
conditions hold true at $u_*$. The limiting Fr\'{e}chet subdifferential of the function $u \mapsto f_0(x, u, \xi)$ with
respect to $\xi$ (see~\cite{JouraniThibault}) at the point $(x, 0, 0)$, which we denote by 
$\partial_{F, u}^{\infty} f_0(x, \cdot)(0, 0)$, is equal to $[-1, 1]$. Therefore, for the function $p(\cdot) \equiv 0$
and for all $x \in [0, 1]$ one has
$$
  p'(x) \in \partial_{F, u}^{\infty} f_0(x, \cdot)(u_*(x), u_*'(x)), \quad
  p(x) u_*'(x) - f_0(x, u_*(x), u_*'(x)) = 0 
  = \max_{v \in \xR} \big( p(x) v - f_0(x, u_*(x), v) \big),
$$
that is, the optimality conditions from \cite[Thrm.~3.5.1]{Bellaassali} are satisfied for $p(\cdot) \equiv 0$, 
$\lambda = 1$, and $\gamma = 0$.

To apply other optimality condition to problem \eqref{eq:IsoperimetricProb_Example}, one needs to transform this
problem to an equivalent one without isoperimetric constraints. Such transformation can be done in many different ways.
Following \cite[Example~4.5.4]{Clarke} we can reformulate problem \eqref{eq:IsoperimetricProb_Example} as the following
Mayer problem with nonholonomic inequality constraints:
\begin{equation} \label{eq:IsoperimetricExample_Clarke}
\begin{split}
  &\min f(x(1)) = x_3(1) \quad \text{subject to} \quad  x(0) = 0, \quad x_1(1) = 0, \quad x_2(1) \le 0,\\
  &\varphi_1(x(t), \dot{x}(t)) = x_1(t) - \dot{x}_2(t) \le 0, \quad
  \varphi_2(x(t), \dot{x}(t)) = \max\{ -|x_1(t)|, -|\dot{x}_1(t)| \} - \dot{x}_3(t) \le 0, \quad t \in [0, 1].
\end{split}
\end{equation}
Let us verify optimality conditions for this problem \cite[Corollary~4.5.1]{Clarke} at the point $x_* \equiv 0$,
which corresponds to the point $u_* \equiv 0$ in problem \eqref{eq:IsoperimetricProb_Example}. Indeed, the Clarke
subdifferentials of the functions $\varphi_i$ at the origin have the form:
$$
  \partial_{Cl} \varphi_1(0) = \{ (1, 0, 0, 0, -1, 0) \}, \quad
  \partial_{Cl} \varphi_2(0) = \co\Big\{ \big( \pm 1, 0, 0, 0, 0, -1 \big), \big( 0, 0, 0, \pm 1, 0, -1 \big) \Big\}.
$$
As is readily seen, the constraint qualifications from \cite[Corollary~4.5.1]{Clarke} is satisfied at $x_*$. Note also
that the Clarke normal cone to the set $S = \{ (x_1, x_2, x_3) \in \xR^3 \mid x_1 = 0, x_2 \le 0 \}$ at the origin has
the form $N_S(0) = \{ (t, s, 0) \in \xR^3 \mid t \in \xR, s \ge 0 \}$. Therefore, for $p(t) \equiv (0, 0, -1)$,
$\lambda_1(t) \equiv 0$, $\lambda_2(t) \equiv 1$, and $\lambda_0 = 1$ one has
$$
  (\dot{p}(t), p(t)) \in \lambda_1(t) \varphi_1(x_*(t), \dot{x}_*(t)) + 
  \lambda_2(t) \varphi_2(x_*(t), \dot{x}_*(t)) \quad \forall t \in [0, 1]
$$
and $-p(1) \in \lambda_0 \partial_{Cl} f(x_*(1)) + N_S(x_*(1))$. Thus, optimality conditions for problem
\eqref{eq:IsoperimetricExample_Clarke} in terms of the Clarke subdifferential \cite[Corollary~4.5.1]{Clarke} are
satisfied at the point $x_*(\cdot) = 0$.

Problem \eqref{eq:IsoperimetricProb_Example} can also be rewritten as the following nonsmooth optimal control problem:
\begin{align*}
  &\min \: \mathcal{J}(x, u) = \int_0^1 \max\big\{ - |x_1(t)|, - |u(t)| \big\} \, dt \quad \text{subject to} \\
  &\dot{x}_1(t) = u(t), \quad \dot{x}_2(t) = x_1(t), \quad u(t) \in \xR, 
  \quad t \in [0, 1],\quad x_1(0) = x_1(1) = 0, \quad x_2(0) = 0, \quad x_2(1) \le 0.
\end{align*}
One can verify that the pair $(x_*, u_*) \equiv (0, 0)$ satisfies various existing optimality conditions for this
problem in terms of subdifferentials and normal cones
\cite{Clarke,Mordukhovich_II,Jourani2009,Loewen,Vinter,Clarke2005,ClarkeDePinho2009,Clarke2013,Ioffe2016,Polovinkin2015,
Polovinkin2018,Polovinkin2019}. We leave the laborious task of verifying these conditions to the interested reader.
Instead, let us check here whether optimality conditions from Theorem \ref{thrm:IsoperimetricInequalProbl} are satisfied
at $u_*$. 

The function $f_0$ can be rewritten as 
$$
  f_0(x, u, \xi) = \max\{ |u|, |\xi| \} - |u| - |\xi|
  = \max\{ u, -u, \xi, -\xi \} + \min\{ u, - u \} + \min\{ \xi, - \xi \}. 
$$
Hence with the use of the codifferential calculus \cite{DemRub_book} one gets that
$$
  \underline{d}_{u, \xi} f_0(x, u, \xi) = \co\left\{ 
  \left( \begin{smallmatrix} \pm u - g(u, x) \\ \pm1 \\ 0 \end{smallmatrix} \right),
  \left( \begin{smallmatrix} \pm \xi - g(u, x) \\ 0 \\ \pm1 \end{smallmatrix} \right) \right\},
  \overline{d}_{u, \xi} f_0(x, u, \xi) = \co\left\{ 
  \left( \begin{smallmatrix} (-1)^i u + (-1)^j \xi + |u| + |\xi| \\ (-1)^i \\ (-1)^j \end{smallmatrix} \right)
  \Bigm| i, j \in \{ 1, 2 \} \right\},
$$
where $g(u, \xi) = \max\{ |u|, |\xi| \}$, while $\underline{d}_{u, \xi} f_1(x, u, \xi) = \{ (0, 1, 0) \}$ and
$\overline{d}_{u, \xi} f_1(x, u, \xi) = \{ 0 \}$. Therefore, as one can readily see, both functions $f_0$ and $f_1$
satisfy the codifferentiability conditions of order $p = 1$. The set $\Omega = (0, 1)$ is obviously bounded and has the
segment-property. Moreover, if for some $\zeta \in \xL^{\infty}((0, 1); \xR, \diverg) = \xWsp^{1, \infty}(0, 1)$ one has
$(0, \zeta'(x), \zeta(x)) \in \underline{d}_{u, \xi} f_1(x, u_*'(x), u_*(x)) = \{ (0, 1, 0) \}$ for a.e. 
$x \in (0, 1)$, then $\zeta(x) = 0$, while $\zeta'(x) = 1$ for a.e. $x \in (0, 1)$, which is impossible. Thus,
constraint qualification \eqref{eq:IsoperimIneq_CQ} holds true at $u_*$.

Suppose that optimality conditions from Theorem \ref{thrm:IsoperimetricInequalProbl} are satisfied at $u_*$. Then for
any measurable selection $(0, w_{10}(\cdot), w_{20}(\cdot))$ of the 
set-valued map $\overline{d}_{u, \xi} f_0(\cdot, u_*(\cdot), u_*'(\cdot))$ there exist $\lambda_1 \ge 0$ and 
$\zeta \in \xWsp^{1, \infty}(0, 1)$ such that for a.e. $x \in (0, 1)$ one has
$$
  (0, \zeta'(x), \zeta(x)) \in 
  \underline{d}_{u, \xi} f_0(x, u_*(x), u_*'(x)) + (0, w_{10}(x), w_{20}(x))
  +  \lambda_1 \underline{d}_{u, \xi} f_1(x, u_*(x), u_*'(x)).
$$
Define $z_0(x) = (0, w_{10}(x), w_{20}(x)) = (0, 1, 1) \in \overline{d} f_0(x, 0, 0)$, if $x \in [0, 0.5]$ and
$z_0(x) = (0, 1, -1)  \in \overline{d} f_0(x, 0, 0)$, if $x \in (0.5, 1]$. Then there exist $\lambda_1 \ge 0$ and 
$\zeta \in \xWsp^{1, \infty}(0, 1)$ such that
\begin{align} \label{eq:OptCond_IsoperimetricExample_1}
  \left( \begin{smallmatrix} 0 \\ \zeta'(x) \\ \zeta(x) \end{smallmatrix} \right) &\in
  \co\left\{ \left( \begin{smallmatrix} 0 \\ 2 + \lambda_1 \\ 1 \end{smallmatrix} \right),
  \left( \begin{smallmatrix} 0 \\ \lambda_1 \\ 1 \end{smallmatrix} \right),
  \left( \begin{smallmatrix} 0 \\ 1 + \lambda_1 \\ 0 \end{smallmatrix} \right),
  \left( \begin{smallmatrix} 0 \\ 1 + \lambda_1 \\ 2 \end{smallmatrix} \right) \right\}
  \quad \text{for a.e. } x \in [0, 0.5], \\
  \left( \begin{smallmatrix} 0 \\ \zeta'(x) \\ \zeta(x) \end{smallmatrix} \right) &\in
  \co\left\{ \left( \begin{smallmatrix} 0 \\ 2 + \lambda_1 \\ -1 \end{smallmatrix} \right),
  \left( \begin{smallmatrix} 0 \\ \lambda_1 \\ -1 \end{smallmatrix} \right),
  \left( \begin{smallmatrix} 0 \\ 1 + \lambda_1 \\ 0 \end{smallmatrix} \right),
  \left( \begin{smallmatrix} 0 \\ 1 + \lambda_1 \\ -2 \end{smallmatrix} \right) \right\}
  \quad \text{for a.e. } x \in (0.5, 1]. \label{eq:OptCond_IsoperimetricExample_2}
\end{align}
Consequently, $\zeta'(x) \ge 0$ for a.e. $x \in (0, 1)$, $\zeta(x) \ge 0$ for a.e. $x \in (0, 0.5)$, and
$\zeta(x) \le 0$ for a.e. $x \in (0.5, 1)$. Redefining, if necessary, the function $\zeta$ on a set of measure zero one
can suppose that $\zeta$ is Lipschitz continuous (see, e.g.~\cite[Thrm.~7.17]{Leoni}). Therefore, from the inequality 
$\zeta'(\cdot) \ge 0$ it follows that the function $\zeta$ is non-decreasing. Hence with the use of the
inequalities $\zeta(x) \ge 0$ for a.e. $x \in (0, 0.5)$ and $\zeta(x) \le 0$ for a.e. $x \in (0.5, 1)$ one obtains that 
$\zeta(x) \equiv 0$ and $\zeta'(x) \equiv 0$, which contradicts the fact that the zero vector does not belong to 
the right-hand sides of \eqref{eq:OptCond_IsoperimetricExample_1} and \eqref{eq:OptCond_IsoperimetricExample_2}.
Thus, optimality conditions from Theorem \ref{thrm:IsoperimetricInequalProbl} are not satisfied at $u_*$, and once again
optimality conditions in terms of codifferentials were able to detect the non-optimality of the point $u_*$, when
subdifferential-based optimality conditions failed to do so.
\end{example}

\section{Conclusions}

In this paper we presented a general theory of first order necessary optimality conditions for nonsmooth
multidimensional problems of the calculus of variations on arbitrary (not necessarily bounded) domains.  This theory is
based on the concepts of codifferentiability and quasidifferentiability of nonsmooth functions developed in the finite
dimensional case by Demyanov, Rubinov, and Polyakova (see~\cite{DemyanovDixon,DemRub_book,Quasidifferentiability_book}).
We proved that a nonsmooth integral functional defined on the Sobolev space is continuously codifferentiable and
computed its codifferential and  quasidifferential under the assumption that the integrand satisfies the
codifferentiability conditions introduced in this paper. These conditions, in essence, mean that the integrand is
continuously codifferentiable and satisfies, along with its codifferential, some natural growth conditions. In
comparison with our previous paper \cite{Dolgopolik_CalcVar}, in this work we proved the codifferentiability of the
integral functional without the assumption that the domain of integration is bounded and has the segment property
(provided $p > 1$), demonstrated that the obscure and hard to verify assumption on \textit{uniform codifferentiability
with respect to the Sobolev space} is completely redundant (thus, giving a positive answer to the second question raised
in \cite[Remark~4.22]{Dolgopolik_CalcVar}), and proved the \textit{continuous} codifferentiability of the integral
functional for all $1 \le p \le + \infty$ (in \cite{Dolgopolik_CalcVar} the continuity of the codifferential mapping
was proved only in the case $p = + \infty$).

The explicit expressions for a codifferential and a quasidifferential of the integral functional obtained in this
article allowed us to apply general necessary optimality conditions for constrained nonsmooth optimisation problems in
Banach spaces in terms of quasidifferentials \cite{Dolgopolik_QuasidiffProg,Dolgopolik_MetricReg} to easily obtain
necessary optimality conditions for constrained nonsmooth problems of the calculus of variations, including problems
with additional constraints at the boundary and problems with isoperimetric constraints. As is demonstrated by a series
of simple examples, our optimality conditions are sometimes better than the existing ones in terms of various
subdifferentials, since they are able to detect the non-optimality of a given point, when subdifferential-based
optimality conditions fail to disqualify this point as non-optimal.

\bibliographystyle{abbrv}  
\bibliography{NonsmoothCalcVar_bibl}

\end{document}